\documentclass[10pt]{amsart}

\usepackage[french,english]{babel}
\usepackage{url}

\usepackage[all,2cell,emtex]{xy}
\usepackage{amssymb,amsmath,bbm,xr,a4wide}
\usepackage[mathscr]{euscript}
\usepackage{mathrsfs}

\usepackage[toc,page]{appendix}
\input cyracc.def
\DeclareFontFamily{U}{russian}{}
\DeclareFontShape{U}{russian}{m}{n}
        { <5><6> wncyr5
        <7><8><9> wncyr7
        <10><10.95><12><14.4><17.28><20.74><24.88> wncyr10 }{}
\DeclareSymbolFont{Russian}{U}{russian}{m}{n}
\DeclareSymbolFontAlphabet{\mathcyr}{Russian}
\makeatletter
\let\@math@cyr\mathcyr
\renewcommand{\mathcyr}[1]{\@math@cyr{\cyracc #1}}
\makeatother

\DeclareMathOperator{\dtr}{degtr}
\newcommand{\base}{\mathscr S}
\newcommand{\basex}[1]{\mathscr S_{/#1}}
\newcommand{\res}[1]{\kappa_{#1}}
\newcommand{\HI}{HI}

\DeclareMathOperator{\codim}{codim}

\newcommand{\C}{\mathscr C} 
\newcommand{\D}{\mathscr D} 

\newcommand{\hrt}[1]{#1^\heartsuit}
\DeclareMathOperator{\ohrt}{\otimes^H}

\newcommand{\cG}{\mathcal G}
\newcommand{\gen}[1]{\langle #1 \rangle_+}

\DeclareMathOperator{\Th}{Th}

\newcommand{\bS}{\mathbf S}

\DeclareMathOperator{\uK}{\underline K}

\newcommand{\Mb}{M^{BM}}
\newcommand{\dtw}[1]{\langle#1\rangle}
\newcommand{\gtw}[1]{\{#1\}}

\DeclareMathOperator{\Ker}{Ker}

\newcommand{\pplim}[1] {\underset{#1}{"\!\varprojlim\!"}}
\newcommand{\pplimN}{{"\varprojlim\!\!"}}
\newcommand{\pro}[1] {\operatorname{pro}-#1}

\newcommand{\Setx}[1]{\mathscr S^o_{/#1}}
\newcommand{\bSetx}[1]{\bar{\mathscr S}^o_{/#1}}
\newcommand{\Petx}[1]{\mathscr P^o_{/#1}}
\newcommand{\hMb}{\hat M^{BM}}

\newcommand{\mH}{\mathcal H} 
\newcommand{\hmH}{\hat{\mathcal H}}

\newcommand{\HM}{H_M}

\newcommand{\rH}{\hat H} 
\DeclareMathOperator{\hocolim}{hocolim}

\DeclareMathOperator{\hmtr}{\Pi^{tr}_*}
\DeclareMathOperator{\hmod}{\Pi^{}_*}
\DeclareMathOperator{\sh}{\Pi^\delta}
\DeclareMathOperator{\she}{\Pi^{\delta-\eff}}
\DeclareMathOperator{\sht}{\tilde \Pi^{\delta}}
\DeclareMathOperator{\shte}{\tilde \Pi^{\delta-\eff}}

\newcommand{\hrep}{h_0^\delta}

\newcommand{\grp}{\mathbf{cGrp}}
\DeclareMathOperator{\shet}{\mathbf{Sh}_\et}
\newcommand{\HB}{H_\mathcyr B}

\newcommand{\pure}{\mathrm{pur}}
\newcommand{\hgrp}[1]{\underline{#1}}

\newcommand{\sab}{s\mathscr{A}}

\newcommand{\T}{\mathscr T} 

\newcommand{\M}{\mathscr M} 

\DeclareMathOperator{\pts}{Pt}

\DeclareMathOperator{\MTh}{\mathit MTh}

\DeclareMathOperator{\Der}{D}

\newcommand{\smod}[1]{#1\!-\!mod}
\newcommand{\MGL}{\mathbf{MGL}}
\newcommand{\MZ}{\mathbf M\ZZ}

\newcommand{\MGLmod}{\smod{\MGL}}
\newcommand{\MZmod}{\smod{\MZ}}

\newcommand{\DA}{\Der_{\AA^1}}
\newcommand{\DAx}[1]{\Der_{\AA^1,#1}}

\newcommand{\sm}{\mathit{Sm}}

\newcommand{\ab}{\mathscr Ab}
\newcommand{\SH}{\mathrm{SH}}

\newcommand{\eff}{\mathit{eff}}
\newcommand{\qf}{\mathit{qf}}

\newcommand{\DM}{\mathrm{DM}}
\newcommand{\DMe}{\mathrm{DM}^{\eff}}

\newcommand{\DMB}{\DM_\mathcyr B}

\DeclareMathOperator{\Hom}{Hom}

\DeclareMathOperator{\uHom}{\underline{Hom}} 

\DeclareMathOperator{\Pic}{Pic}
\DeclareMathOperator{\spec}{Spec}

\newcommand{\ilim} { \varinjlim }
\newcommand{\plim} { \varprojlim }

\DeclareMathOperator{\Comp}{C}

\newcommand{\derL}{\mathbf{L}}
\newcommand{\derR}{\mathbf{R}}

\newcommand{\NN} {\mathbf N}
\newcommand{\ZZ} {\mathbf Z}
\newcommand{\QQ} {\mathbf Q}

\renewcommand{\AA} {\mathbf A}
\newcommand{\PP} {\mathbf P}
\newcommand{\GG} {\mathbf{G}_m}

\newcommand{\E}{\mathbb E} 
\newcommand{\F}{\mathbb F}
\newcommand{\G}{\mathbb G}

\newcommand{\un}{\mathbbm 1} 
\newcommand{\pur}{\mathfrak p} 

\newcommand{\A}{\mathscr A}

\newcommand{\et}{\mathrm{\acute{e}t}}
\newcommand{\cdh}{{\mathrm{cdh}}}

\newcommand{\Mikhail}{}

\title[Dimensional homotopy t-structures]{Dimensional homotopy t-structures in motivic homotopy theory}

\author{Mikhail Bondarko}
\address{St. Petersburg State University, Department of
Mathematics and Mechanics, \\ Bibliotechnaya Pl. 2, 198904, St.
Petersburg, \\ Russia}
\email{m.bondarko@spbu.ru}

\author{Fr\'ed\'eric D\'eglise}
\address{E.N.S. Lyon - UMPA\\
46\\all\'ee d'Italie\\
69364 Lyon Cedex~07\\
France}
\email{frederic.deglise@ens-lyon.fr}
\urladdr{http://perso.ens-lyon.fr/frederic.deglise/}

\thanks{Sections 1 and 3.4 of the paper were written under the support of the Russian Science Foundation grant no. 16-11-10200.
 The second  author  was partially supported by the ANR
 (grant No. ANR-12-BS01-0002) and the LABEX MILYON (ANR-10-LABX-0070)
 of the Universit\'e de Lyon,
 within the program "Investissements d'Avenir"
 (ANR-11-IDEX-0007) operated by the French National Research Agency (ANR).
Finally, this project was conceived during a visit supported by the ANR
 (grant No. ANR-12-BS01-0002) of the first author to the ENS of Lyon}
\date{November 2016}

\newtheorem{thm}{Theorem}[subsection]
\newtheorem{prop}[thm]{Proposition}
\newtheorem{lm}[thm]{Lemma}
\newtheorem{cor}[thm]{Corollary}

\theoremstyle{remark} 
\newtheorem{rem}[thm]{Remark}

\newtheorem{ex}[thm]{Example}

\theoremstyle{definition} 
\newtheorem{df}[thm]{Definition}
\newtheorem{num}[thm]{}

\numberwithin{equation}{thm}

\newtheorem{thm*}{Theorem}

\begin{document}

\begin{abstract}
The aim of this work is to construct certain homotopy t-structures
 on various categories of motivic homotopy theory,
 extending works of Voevodsky, Morel, D\'eglise
 and Ayoub.
 We prove these $t$-structures possess many good properties,
 some analogous to those of the perverse $t$-structure of
 Beilinson, Bernstein and Deligne. We compute
 the homology of certain motives, notably in the case of relative
 curves.
 We also show that the hearts of these $t$-structures provide convenient
 extensions of the theory of homotopy invariant sheaves
 with transfers, extending some of the main results of Voevodsky.
 These t-structures are closely related to Gersten weight structures
 as defined by Bondarko.
%
%
\end{abstract}

\maketitle

\setcounter{tocdepth}{3}
\tableofcontents

\section*{Introduction}

\noindent \textbf{Background.} The theory of motivic complexes over a perfect field $k$,
 invented by Voevodsky after a conjectural description of Beilinson,
 is based on the notion of homotopy invariant sheaves with transfers. These
 sheaves have many good properties: they form an abelian category $\HI(k)$ with a
 tensor product, their (Nisnevich) cohomology over smooth $k$-schemes
 is homotopy invariant, and they admit Gersten resolutions.
 An upshot of the construction of the category of motivic complexes
 $\DM^{eff}_-(k)$ is the existence of a canonical $t$-structure
 whose heart is exactly $\HI(k)$. It was called the
 \emph{homotopy $t$-structure} by Voevodsky in \cite[chap. 5]{FSV}.
 Though this $t$-structure is not the motivic $t$-structure, its role
 is fundamental in the theory of Voevodsky. For example, the main conjectures
 of the theory have a clean and simple formulation in terms of the homotopy
 $t$-structure.\footnote{Here are some classical examples:
\begin{itemize}
\item The \emph{Beilinson-Lichtenbaum conjecture} (now a theorem):
  for any prime $l \in k^\times$, the motivic complex
  $\ZZ/l\ZZ(n)$ is isomorphic to $\tau_{\leq n} \derR\pi_*(\mu_l^{\otimes,n})$
  where $\tau_{\leq n}$ is the truncation functor,
	with cohomological convention,
	for the homotopy $t$-structure and $\pi_*$ the forgetful functor from
	\'etale to Nisnevich sheaves.
\item The \emph{Beilinson-Soul\'e conjecture}: for any integer $n>0$,
 the motivic complex $\ZZ(n)$ is concentrated in cohomological degrees $[1,n]$
 for the homotopy $t$-structure.
\item A weak form of \emph{Parshin conjecture}: for any integer $n>0$,
 the motivic complex $\QQ(n)$ over a finite field is concentrated in 
 cohomological degree $n$ for the homotopy $t$-structure.
 \end{itemize}}

Many works have emerged around the homotopy $t$-structure.
 The Gersten resolution for homotopy invariant sheaves with transfers has been
 promoted into an equivalence of categories between $\HI(k)$ and
 a full subcategory of the category of Rost's cycle modules (\cite{Ros}) by
 the second-named author in \cite{Deg9}. At this occasion, it was realized
 that to extend this equivalence to the whole category of Rost's cycle modules,
 one needed an extension of $\HI(k)$ to a non effective category
 (with respect to the $\GG$-twist), in the spirit of stable homotopy.
 This new abelian category was interpreted as the heart of the category
 now well known as $\DM(k)$, the category of \emph{stable motivic complexes}.

In the meantime, Morel has extended the construction of the homotopy $t$-structure
 to the context of the stable homotopy category $\SH(k)$ of schemes over a
 field $k$ (non necessarily perfect), and its effective analogue,
 the $S^1$-stable category $\SH^{S^1}(k)$ in \cite{dmtilde3}.
 In this topological setting, the $t$-structure appeared as an incarnation
 of the \emph{Postnikov tower} and the heart of the homotopy $t$-structure
 as the natural category in which stable (resp. $S^1$-stable) homotopy
 groups take their value. It was later understood that the difference
 between the heart of $\SH(k)$ and $\DM(k)$ was completely encoded
 in the action of the \emph{(algebraic) Hopf map} (a conjecture of Morel
 proved in \cite{Deg10}). Moreover
 the computations of the stable homotopy groups
 that we presently know use the language of the homotopy $t$-structure
 in an essential way.

The final foundational step up to now was taken by Ayoub who extended
 the definition of Voevodsky and Morel to the relative setting in
 his Ph. D. thesis \cite{ayoub1}, in his abstract framework of
 \emph{stable homotopy functors}. Under certain assumptions,
 he builds a $t$-structure
 over a scheme $S$ with the critical property
 that it is ``compatible with gluing'': given any closed subscheme $Z$
 of $X$ with complement $U$, the $t$-structure over $S$ is
 uniquely determined by its restrictions over $Z$ and $U$
 (cf. \cite[2.2.79]{ayoub1}). This
 is an analogous property to that of the \emph{perverse $t$-structure}
 on complexes of \'etale sheaves underlined in the fundamental work 
 \cite{BBD} of Beilinson, Bernstein and Deligne. Besides, when $S$ is the
 spectrum of a field,
 Ayoub establishes that his $t$-structure agrees with that of Morel
 (cf. \cite[2.2.94]{ayoub1}).
 After the work that was done on the foundations of
 motivic complexes in the relative context in \cite{CD3, CD5}, 
 this agreement immediately extends to the definition of Voevodsky.
 However, the assumptions required by Ayoub in this pioneering work are
 quite restrictive. They apply to the stable homotopy category
 only in characteristic $0$
 (integrally or with rational coefficients,
 cf. \cite[2.1.166, 2.1.168]{ayoub1}),
 and, after the results of \cite{CD3, CD5},
 to the triangulated category of mixed motives,
 integrally in characteristic $0$,
 and rationally over a field or a valuation ring
 (cf. \cite[2.1.171, 2.1.172]{ayoub1}). Lastly, he does not prove
 his $t$-structure is non-degenerate.

\bigskip

\noindent \textbf{Our main construction in the stable case.} The purpose
 of this text is to provide a new approach on the problem of
 extending the homotopy $t$-structure of Voevodsky and that of Morel
 over a base scheme $S$, assumed to
 be noetherian, excellent and finite dimensional. In this introduction,
 we will describe our
 constructions by fixing triangulated categories $\T(S)$ for
 suitable schemes $S$ assuming they are equipped with
 Grothendieck's six operations.\footnote{We
 do not have to be more precise for describing our definitions
 but the reader can find a list of our main examples
 in Example \ref{ex:motivic_cat}.}
 
 Our first technical tool, as in the work
 of Ayoub, is the possibility of defining a $t$-structure on
 a compactly generated\footnote{see our conventions
 below for a recall on this property of triangulated categories;}
 triangulated category by fixing a family of compact objects, called \emph{generators}
 and set them as being homologically non-$t$-negative.\footnote{Following
 Morel and Ayoub we will use homological conventions for our $t$-structures;
 see Section \ref{sec:t-struct}.}
 Then the class of $t$-negative objects are uniquely determined
 by the orthogonality property and the existence of truncation functors
 follows formally (see again Section \ref{sec:t-struct} for details).
 Our approach differs from that of Ayoub in the choice of generators.

The guiding idea is that, as in the case of motivic complexes,
 motives of smooth $S$-schemes should be homologically non-$t$-negative.
 However, these generators are not enough to ensure the compatibility
 with gluing. In particular, we should be able to use non smooth
 $S$-schemes, in particular closed subschemes of $S$. We deal with
 this problem with two tools: the exceptional direct image functor $f_!$
 along with a choice of a dimension function\footnote{dimension functions
 appeared for the first time in \cite[XIV, 2.2]{SGA4} in the local case.
 They were later formalized by Gabber, cf. \cite{PS_dim}. 
 We recall this notion in Section \ref{sec:dim};}
 on the scheme $S$
 which will add a correcting shift in our generators, related to the dimension
 of the fibers of $f$. A dimension function on $S$ is a map
 $\delta:S \rightarrow \ZZ$ such that $\delta(x)=\delta(y)+1$
 if $y$ is a codimensional one point of the closure of $x$ in $S$.
 When $S$ is of finite type over a filed $k$, the canonical example
 is the Krull dimension:
\begin{equation}\label{eq:dim_fn_field}\tag{Intro.a}
\delta_k(x)=\dim\big(\overline{\{x\}}\big)=\dtr(\kappa(x)/k)
\end{equation}
where $\kappa(x)$ is the residue field of $x$ in $S$. For any $S$-scheme
 $X$ of finite type, the dimension function $\delta$ induces a canonical
 dimension function on $X$, and we denote by $\delta(X)$ the maximum
 of this induced function on $X$.
 The last fact of which the reader must be aware is that Voevodsky's
 homotopy $t$-structure in the stable case is stable under $\GG$-twists:
 the functor $K \mapsto K(1)[1]$ is $t$-exact for Voevodsky's
 homotopy $t$-structure (beware that this is false in the effective setting).
 
Putting all these ideas together, we define the
 \emph{$\delta$-homotopy $t$-structure over $S$} by taking the following
 generators for homologically positive objects:
\begin{equation}\label{intro:stable_generators}\tag{Intro.b}
f_!(\un_X)(n)[\delta(X)+n], \text{ $f$ separated of finite type, $n$ any integer}.
\end{equation}
With this definition, it is clear that the endofunctor 
 $K \mapsto K(1)[1]$ of $\T(S)$ becomes $t$-exact. This already shows that
 our $t$-structure covers a phenomenon which is special to motives:
 on triangulated categories
 such as $D^b(S_{\et},\ZZ/l\ZZ)$, $l$ invertible on $S$, it is not
 reasonable because if $S$ contains a primitive $l$-th root of unity, we
 get an isomorphism $\ZZ/l\ZZ(1) \simeq \ZZ/l\ZZ$ so that the $t$-structure
 with the above generators is degenerate. The same argument
 will apply to the category of integral \'etale motives
 (cf. \cite{ayoub2, CD4})
 and to modules over algebraic $K$-theory
 (cf. \cite[\textsection 13.3]{CD3} and \cite{BL16}).
 
However, on categories such as $\SH$ and $\DM$
 (see Example \ref{ex:motivic_cat} for a detailed list), the resulting $t$-structure
 is very reasonable and we easily deduce,
 without any further assumption on $S$,
 that the $\delta$-homotopy $t$-structure
 is compatible with gluing as well as some basic exactness:
 $f_*=f_!$ (resp. $f^*=f^!$) is $t$-exact when $f$ is finite (resp. \'etale).
 Note that the $\delta$-homotopy $t$-structure does not really
 depend on the choice of $\delta$. Rather, a change of $\delta$
 shifts the $t$-structure. So when formulating $t$-exactness properties,
 we  have to  be precise about the choices of dimension functions. We refer
 the reader to the text for these details.

After this definition, our first theorem comes from the fact that
 we can improve the description of the generators for the $\delta$-homotopy
 $t$-structure in several ways. Actually in the description
 \eqref{intro:stable_generators} we can add one of the following assumptions:
\begin{enumerate}
\item $f$ is proper (even projective).
\item $X$ is regular.
\item $f$ is proper and $X$ is regular.
\end{enumerate}
Points 1 and 2 are easily obtained (see Proposition \ref{prop:generators1})
 but they already yield interesting results. Most notably,
 they give the comparison of our definition with the classical ones.
 When $S$ is the spectrum of a perfect field $k$
 of characteristic exponent $p$ and $\delta_k$ is the
 dimension function of \eqref{eq:dim_fn_field},
 the $\delta_k$-homotopy $t$-structure on $\DM(k)[1/p]$ (resp. $\SH(k)$) coincides
 with that of Voevodsky (resp. Morel):
 see Example \ref{ex:delta_htp-t-struct_field}.
 We also give a comparison with Ayoub's $t$-structure under favourable
 assumptions: see Corollary \ref{cor:compare_Ayoub}.

The strongest form of this improvement on generators, point 3 above,
 is obtained at the cost of using a convenient resolution
 of singularities (Hironaka refined by Temkin in characteristic zero,
 de Jong-Gabber for rational coefficients,
 Gabber for $\ZZ[1/p]$-coefficients in characteristic $p>0$) together with good properties for the category $\T$ (the precise hypotheses
 are given in Paragraph \ref{num:resolution}).
 This description (Th. \ref{thm:strong_generators}) will be a key point
 for our main theorem. The method is classical though it uses a trick of
 Riou to go beyond the cases that were known up to now.
 In particular, it implies an interesting result of motivic homotopy
 theory which was not known till now and that deserves a separate
 formulation:
\begin{thm*}[Th. \ref{thm:duality} and Ex. \ref{ex:duality_SH}]
Let $p$ be a prime or $p=1$ and assume the residue fields of $S$
 have characteristic exponent $p$.
Let $\SH_c(S)[1/p]$ be the constructible part (\emph{i.e.} made of compact spectra)
 of the stable homotopy category $\SH(S)[1/p]$ over $S$ with $p$ inverted.
\begin{enumerate}
\item Then $\SH_c[1/p]$ is stable under the six operations if one restricts to
 schemes of characteristic exponent $p$ and to morphisms $f$ of finite type
 for the operation $f_*$.
\item For any separated morphism $f:X \rightarrow S$ of finite type
 such that $S$ is regular,
 the spectrum\footnote{in this formula, $S^0$ is the sphere spectrum over $S$;}
 $f^!(S^0)$ is dualizing in the triangulated monoidal
 category $\SH_c(X)[1/p]$.
\end{enumerate}
\end{thm*}
Up to now, this fact was only known when $p=1$ according to \cite{ayoub1}.

\bigskip

\noindent \textbf{The effective case.}
The preceding theorem is deduced from our results on the
 $\delta$-homotopy $t$-structure because there is a strong relation between
 the generators of a $t$-structure and the generators in the whole 
 triangulated category.\footnote{See our conventions on triangulated categories
 for recall on that later notion.} Indeed, a generated $t$-structure
 is left non-degenerate if and only if its generators are also generators
 of the triangulated category (see Lemma \ref{lm:non-neg-degenerated}).
 This leads us to the following considerations, 
 based on the fact that Tate twist is not $\otimes$-invertible on
 effective motives: it is natural to consider the generators
 \eqref{intro:stable_generators} with the restriction that
 $n \geq \delta(X)$ (rather than $n \geq 0$, for duality reasons).
 Then, if we want this new $t$-structure to be non-degenerate, we are lead
 to consider the triangulated localizing subcategory of $\T(S)$
 generated by objects of the form:
\begin{equation}\label{intro:effective_generators}\tag{Intro.c}
f_!(\un_X)(n)[\delta(X)+n], \text{ $f$ separated of finite type, $n\geq \delta(X)$}.
\end{equation}
We denote this triangulated category by $\T^{\delta-\eff}(S)$ and call
 it the \emph{$\delta$-effective subcategory}. The $t$-structure generated
 by the above generators is called the \emph{effective $\delta$-homotopy $t$-structure}.
 We show that this triangulated category, with its $t$-structure,
 has many good properties, and especially, it satisfies
 the gluing (or localization) property (see Proposition \ref{prop:localization_effective}
 and \ref{delta-htp_eff_glued}). Besides, we are able to
 describe its stability by many of the 6 operations
 (see Section \ref{sec:delta_effective} for details).

The main case of this definition is given by 
 the category of $R$-linear motives $\DM(S,R)$ over a scheme $S$
 in the following two cases:
\begin{itemize}
\item $S$ has characteristic exponent $p$ and $p \in R^\times$
 (\cite{CD5}).
\item no assumptions on $S$ and $R$ is a $\QQ$-algebra (\cite{CD3}).
\end{itemize}
Then it is shown in Example \ref{ex:delta_effective_fields}
 that when $S=\spec(k)$ is the spectrum of a perfect field
 and $\delta_k$ is the dimension function of \eqref{eq:dim_fn_field},
 the triangulated category $\DM^{\delta-\eff}(k,R)$ is equivalent to
 Voevodsky's category of motivic complexes over $k$.
  Besides, the category $\DM^{\delta-\eff}(k,R)$ is invariant
  under purely inseparable extensions.
 Thus the categories $\DM^{\delta-\eff}(S,R)$ for positive dimensional
 scheme $S$ are essentially uniquely determined by the gluing property
 and their value on fields.
 Therefore, our construction provides a good extension
 of the theory of motivic complexes, with their homotopy $t$-structure,
 which was missing in motivic homotopy theory given that the natural
 category $\DM^{\eff}(S,R)$ of \cite[11.1.1]{CD3} is not known to satisfy
 the gluing property.

\bigskip

\noindent \textbf{The niveau filtration.}
As visible from the Gersten resolution of homotopy invariant sheaves
 with transfers and their comparison with cycle modules, Voevodsky's
 homotopy $t$-structure shares an intimate relation with the
 classical coniveau filtration. In fact,
 it was proved in \cite{Deg11} that the coniveau spectral sequence
 associated with the cohomology of a smooth $k$-scheme with
 coefficients in an arbitrary motivic complex $K$ coincides  from $E_2$-on
 with the spectral sequence associated with the truncations of $K$ for
 the homotopy $t$-structure. In \cite{Bon10a} and \cite{Bon15} this relationship was
 recast into the framework of weight structures,
 as the coniveau filtration for any cohomology
 was interpreted as the \emph{weight filtration} coming from
 a particular case of a weight  structure called the \emph{Gersten weight structure}.

In our generalized context, this relationship is at the heart of
 our main theorem. Our use of a dimension function $\delta$ and 
 the will to work over a base scheme $S$ rather than a field lead us
 to consider niveau filtration measured by the given dimension
 function: in other words, we filter the scheme $S$ by looking at closed
 subschemes $Z \subset S$ such that $\delta(Z) \leq n$ for various
 integers $n$. Moreover, to deal with singular $S$-schemes,
 we are lead to consider homology rather than cohomology
 following in that point the classical work \cite{BO} of Bloch and Ogus.
 We extend their ideas by using the so-called \emph{Borel-Moore homology}
 relative to the base scheme $S$; note that the homology considered by Bloch and Ogus
 is the Borel-Moore homology relative to the base field. In terms of
 the 6 functors formalism, the Borel-Moore homology of a separated
 $S$-scheme $X$ with coefficients
 in any object $\E$ in $\T(S)$ can be defined as follows:
$$
E^{BM}_{p,q}(X/S)=\Hom_{\T(X)}\big(\un_X(q)[p],f^!(\E)\big)
$$
where $f$ is the structural morphism of $X/S$.
When $S$ is regular and $\E$ is the unit object this can be interpreted
 in good situations (for example in the situation of the preceding theorem)
 as cohomology in degree $(-p,-q)$
 with coefficients in the \emph{dualizing object} over $X$.
 Moreover, by adjunction, Borel-Moore homology corresponds to
 the following abelian groups:
$$
\Hom_{\T(X)}\big(\un_X(q)[p],f^!(\E\big))
=\Hom_{\T(S)}\big(f_!(\un_X(q)[p]),\E\big)
$$
so that it is natural, following Riou, to call $f_!(\un_X)$
 the \emph{Borel-Moore object} (motive, spectrum, etc...)
 associated with $f$ (or $X/S$).

The six functors formalism immediately yields that Borel-Moore homology,
 like Chow groups, is covariant with respect to proper morphisms
 (see Section \ref{sec:six} for recall). These ingredients
 altogether allow us to extend the consideration of Bloch and Ogus
 and to build what we call the $\delta$-niveau spectral sequence:
$$
^\delta E^1_{p,q}=\bigoplus_{x \in X_{(p)}} \hat \E^{BM}_{p+q,n}(x/S)
 \Rightarrow \E^{BM}_{p+q,n}(X/S)
$$
where $X_{(p)}=\{x \in X \mid \delta(x)=p\}$ and
 $\hat \E^{BM}_{**}(x/S)$ is computed by taking the limit 
 of the Borel-Moore homology of $(\overline{\{x\}} \cap U)/S$
 for open neighbourhoods $U$ of $x$ in $X$.
 A first application of the $\delta$-niveau spectral sequence
 allows us to compute Borel-Moore homology in case of mixed motives:
\begin{thm*}[Th. \ref{prop:comput_BM_w_Chow}]\label{prop*:comput_BM_w_Chow}
Consider a regular scheme $S$ and a localization $R$
 of the ring of integers $\ZZ$ satisfying one of the following conditions:
\begin{itemize}
\item $R=\QQ$;
\item $S$ is a $\QQ$-scheme;
\item $S$ is an $\F_p$-scheme and $p \in R^\times$.
\end{itemize}
 Let $\delta$ be the dimension function
 on $S$ such that $\delta=-\codim_S$
 (see Example \ref{ex:can_dim_fn}).

Then for any separated $S$-scheme $X$ of finite type
 and any integer $n \in \ZZ$,
 one has a canonical isomorphism:
$$
H_{2n,n}^{BM}(X/S,R) \simeq CH_{\delta=n}(X) \otimes R
$$
where the left hand side is Borel-Moore motivic homology
 of $X/S$ with coefficients in $R$ and the right hand side
 is the Chow group of $R$-linear algebraic cycles $\sum_i n_i.x_i$ of $X$ 
 such that $\delta(x_i)=n$
 (see also \cite[Chap. 41, Def. 9.1]{stack}).
\end{thm*}
 Note in particular that, under the assumptions of the above theorem
 and assuming further that $X/S$ is equidimensional of dimension $d$,
 the preceding isomorphism can be written:
$$
H_{2n,n}^{BM}(X/S,R) \simeq CH^{d-n}(X) \otimes R
$$
where the right hand side is the group of codimension $d-n$ cycles
 (see Corollary \ref{cor:chow_codim}).

The preceding theorem generalizes the case where $S$ is the spectrum of 
 a (perfect) field. It is a new indication of the relevance
 of dimension functions.
As indicated in \cite[Rem. 7.1.12(4)]{CD5}, when $\ell$ is 
 a prime invertible on the scheme $S$,
 this gives a generalized cycle class:
$$
\sigma:CH_{\delta=n}(X) \otimes \QQ
 \simeq H_{2n,n}^{BM}(X/S,\QQ) \rightarrow H_{2n,n}^{BM,\et}(X/S,\QQ_\ell)
$$
where the right hand side is the rational $\ell$-adic Borel-Moore
 \'etale cohomology of $X/S$ and the map is induced by
 $\ell$-adic realization functor of \cite[7.2.24]{CD5}.
 For example, when all the generic points $\eta$ of $X$ satisfy
 $\delta(\eta)=n$ for a chosen integer $n$, we get
 by looking at the image of
 the fundamental cycle of $X$ on the left hand side under the map $\sigma$
 a canonical map in the $\ell$-adic derived category of Ekedahl over
 the small \'etale site of $X$:
$$
\QQ_\ell(n)[2n] \rightarrow f^!(\QQ_\ell)
$$
which is a generalization of the local fundamental class
 (\cite{SGA4D}, when $X$ is a closed subscheme of $S$)
 and of the fundamental class when $S=\spec(k)$ (cf. \cite{BO}).

The $\delta$-niveau spectral sequence will moreover be essential
 in the study of the $\delta$-homotopy $t$-structure.
 In this perspective, we are lead to introduce the following
 definition for a given object $\E$ of $\T(S)$.
 Let $x:\spec(E) \rightarrow S$ be an $E$-valued point of $S$,
 essentially of finite type. Then $E$ is the field of functions
 of an integral affine $S$-scheme of finite type, say $X$.
 We define the \emph{fiber $\delta$-homology} $\rH^\delta_p(\E)$
 of $\E$ as the map which to a point $x$ as above
 and a twist $n \in \ZZ$ associates the following abelian
 group:\footnote{It is shown this definition does not depend
 on the chosen model $X$; see Definition \ref{df:fiber_hlg}.}
\begin{equation}\label{intro:fiber}\tag{Intro.d}
\rH^\delta_p(\E)(x,n)
 =\ilim_{U \subset X} \E^{BM}_{2\delta(x)+p-n,\delta(x)-n}(U/S).
\end{equation}
We also consider the $\delta$-effective version
 of this definition, $\rH^{\delta-\eff}_p(\E)$,
 by restricting to the twists $n \leq 0$.
 
 In fact, the $\delta$-niveau spectral sequence with coefficients
 in $\E$ can be rewritten in terms of the fiber $\delta$-homology
 of $\E$. In particular, we deduce from the convergence of
 the $\delta$-niveau spectral sequence that the sequence
 of functors $\E \mapsto H_p^{\delta}(\E)$ indexed by integers
 $p \in \ZZ$ is conservative. Moreover, our main theorem can be stated
 as follows:
\begin{thm*}[Th. \ref{thm:hlg&htp_t}]\label{thm*:hlg&htp_t}
Under suitable assumptions on $\T$ satisfied in the following
 examples:
\begin{itemize}
\item $k$ is a field of characteristic exponent $p$,
 $S$ is a $k$-scheme essentially of finite type,
 and $\T$ is $\SH[1/p]$
 (resp. $\DM_R$ for a ring $R$ such that $p\in R^\times$);
\item $R$ is a $\QQ$-algebra, $S$ is a scheme 
 essentially of finite type over any excellent scheme
 of dimension less than 4, and $\T=\DM_R$;
\end{itemize}
 an object $\E$ of $\T(S)$
 is positive (resp. negative) for the $\delta$-homotopy
 $t$-structure if and only if $\rH_p^\delta(\E)=0$
 for $p\leq 0$ (resp. $p \geq 0$).
Moreover the same assertion holds in the $\delta$-effective case.
\end{thm*}
This theorem has several interesting consequences.
 It immediately implies the $\delta$-homotopy $t$-structure,
 both in the stable and $\delta$-effective case,
 is non-degenerate, a result which was not known for Ayoub's
 homotopy $t$-structure. It also gives new exactness properties,
 for a morphism $f:T \rightarrow S$ essentially of finite type,
 a dimension function $\delta$ on $S$, $T$ being equipped with
 the dimension function induced by $\delta$:
\begin{itemize}
\item if the dimensions of the fibers of $f$ are bounded by $d$ then the functor
 $f_*:\T(S) \rightarrow \T(T)$ has homological amplitude 
 $[0,d]$ for the $\delta$-homotopy $t$-structure;
\item if $f$ is separated of finite type then
 $f^!:\T(T) \rightarrow \T(S)$ is $t$-exact 
 for the $\delta$-homotopy $t$-structure.
\end{itemize}
One also deduces a characterization of positivity (resp. negativity)
 of an object $\E$
 for the $\delta$-homotopy $t$-structure in terms of vanishing of 
 Borel-Moore homology with coefficients in $\E$ in a certain
 range. For all these facts,
 we refer the reader to the corollaries stated
 after Theorem \ref{thm:hlg&htp_t}.

The preceding theorem opens the way to computations. The case of the
 constant motive $\un_S$ in $\DM(S,R)$, under the assumptions
 of this theorem, is already interesting.
 When $S=\spec(\QQ)$
 (or is regular and admits a characteristic $0$ point),
 we obtain that $\un_S$ has infinitely many homology motives
 in non-negative degrees for the $\delta$-homotopy $t$-structure
 (Example \ref{ex:compute_hlg_constant1}).
 On the contrary, if $S$ is regular, $\un_S$ is concentrated
 in degree $0$ for the \emph{effective} $\delta$-homotopy
 $t$-structure (see Example \ref{ex:unit_eff_heart}).
 This is the advantage of the effective version of the $\delta$-homotopy
 $t$-structure: in general,
 we hope for interesting boundedness properties only in that case. 
 The final computation shows that
 when $S$ is singular, then $\un_S$ may no longer be concentrated
 in degree $0$: thus, the homology sheaves for the effective
 $\delta$-homotopy $t$-structure detect singularities (see Remark
 \ref{rem:compute_hlg_constant2}).

We finally deduce (Corollary \ref{cor:restriction_fields})
 from the preceding theorem the following
 punctual characterization of the $\delta$-homotopy $t$-structure.
 An object $\E$ of $\T(S)$ is homologically positive 
 (resp. negative) for the $\delta$-homotopy $t$-structure
 if and only if the following condition holds:
$$
\forall x \in S, H_p^{\delta_x}(i_x^!\E)=0
 \text{ when  $p\leq \delta(x)$ (resp. $p \geq \delta(x)$).}
$$
Here $\delta_x$ is the obvious dimension function on the
 spectrum of the residue field $\kappa(x)$ of $x$ in $S$,
 and $i_x^!=j^*i^!$ for the factorization
 $\spec{\kappa(x)} \xrightarrow j \overline{\{x\}} \xrightarrow i S$.
 It is striking that this condition is exactly the same (in the homological
 notation) as the one characterizing the perverse $t$-structure
 in \cite[4.0.1, 4.0.2]{BBD}.

\bigskip

\noindent \textbf{The $\delta$-homotopy heart.}
From the results already stated, one deduces several good properties
 of the heart $\hrt{\T(S)}$ of the $\delta$-homotopy $t$-structure.
 Under the assumption of Theorem \ref{thm*:hlg&htp_t},
 this abelian category is a Grothendieck abelian category
 (see Theorem \ref{thm:finally_Grothendieck}),
 with a good family of generators and satisfying
 nice functoriality properties; most notably, the gluing property.
 For suitable choices of $\delta$, it admits a closed monoidal structure.
 In fact, it shares some of the fundamental properties of its model $\HI(k)$.
 In particular, each point $\spec(E) \rightarrow S$, with value
 in a field $E$ which is finitely generated over the corresponding
 residue field of $S$, defines a fiber functor
 to graded abelian groups and the collection of these functors
 is conservative: this is the transformation which to an object
 of the heart $\F$ associates the application $\rH^\delta_0(\F)$
 defined by formula \eqref{intro:fiber}.

Our main example, under the assumptions of the above theorem,
 is the $R$-linear triangulated
 category $\DM_R$ of mixed motives (resp. $\DAx R$,
 the $\AA^1$-derived category of Morel)
 for which we denote the heart by $\sh(S,R)$ (resp. $\sht(S,R)$)
 and $\she(S,R)$ (resp. $\shte(S,R)$) in the $\delta$-effective case.
 It is called the category of \emph{$\delta$-homotopy modules}
 (resp. \emph{generalized $\delta$-homotopy modules}) over $S$,
 and we simply add the adjective \emph{effective} when considering
 the heart of the $\delta$-effective category.
 Note that in each respective case, the effective heart is
 a full abelian subcategory of the general (stable) one
 (cf. Paragraph \ref{num:hrt_basic}).

The category of homotopy modules is a convenient extension
 over an arbitrary base of the category $\HI(k)$ of homotopy invariant
 sheaves with transfers over a perfect field $k$: its effective subcategory
 over $k$ coincides with $\HI(k)$ up to a canonical equivalence,
 the whole category and its effective subcategory are invariant under
 purely inseparable extensions and satisfy the gluing property.
 In fact these properties guarantee it is essentially the unique
 extension of the theory of Voevodsky which satisfies the gluing
 property. Note on the contrary that we cannot prove at the moment
 that generalized homotopy modules are invariant under inseparable
 extensions. On the other hand, for a perfect field $k$, the category
 $\sht(k,R)$ is equivalent to Morel's category of homotopy modules
 (\cite[Def. 5.2.4]{Mor1}). 

Though these two notions of homotopy modules seem very different,
 we prove the following comparison theorem.
\begin{thm*}[see Th. \ref{thm:compute_heart}]\label{thm*:compute_heart}
Let $R$ be a ring and $S$ be a scheme
 such that $R$ is a $\QQ$-algebra or $S$ is a $\QQ$-scheme.
 Then there is a canonical fully faithful functor:
$$
\gamma_*:\sh(S,R) \rightarrow \sht(S,R)
$$
whose image 
 consists of objects on which Morel's algebraic Hopf
 map\footnote{Recall the algebraic Hopf map is the endomorphism of
 the sphere spectrum induced by the obvious morphism:
 $(\AA^2-\{0\}) \rightarrow \PP^1$.}
 $\eta$ acts trivially.
\end{thm*}
The case where $R$ is a $\QQ$-algebra is not surprising
 given that the canonical functor
 $$\DM(S,R) \xrightarrow{\gamma_*} \DA(S,R)=\SH(-S) \otimes R$$
 is fully faithful and its essential image has the same description
 through the Hopf map (cf. \cite[Th. 16.2.13]{CD3}).
 The case $R=\ZZ$ for $\QQ$-schemes is much stronger given that
 the map $\gamma_*$ is no longer full (because of the existence
 of the Steenrod operations). Thanks to our preceding results,
 the proof of the above theorem is obtained by reduction
 to the case of fields, 
 which was proved in \cite{Deg10}.

Note also that we sketch 
 the full proof of the fact that the heart of the category of modules
 of algebraic cobordism over a scheme $S$ of characteristic exponent
 $p$ is equivalent, with $\ZZ[1/p]$-coefficients, to the category
 $\sh(S,\ZZ[1/p])$. This reflects the fact that the spectra
 in the essential image of the functor $\gamma_*$ can be characterized 
 by the property of being orientable.
 
\bigskip

As a conclusion, the category of homotopy modules
 with rational coefficients
 seems to be the appropriate generalization of the notion of
 homotopy invariant sheaves with transfers over an
 arbitrary base $S$ with a dimension function $\delta$
 (the category itself does not depend on $\delta$ up to a canonical
 equivalence). 
An important class of such sheaves comes from semi-abelian
 varieties: the sheaf corresponding to a semi-abelian variety $A$ admits canonical 
 transfers and is homotopy invariant. In the last part of this paper,
 we extend this result to the relative setting, using the works of
 Ancona, Pepin Lehalleur and Huber (\cite{AHP})
 and Pepin Lehalleur (\cite{Pep}).
 Following these works,
 one associates to any commutative group $S$-scheme $G$ an object
 $\hgrp G$ of $\DM(S,\QQ)$. Our last theorem is the following one:
\begin{thm*}[Th. \ref{thm:sabelian&hmod}]
Let $S$ be a regular scheme. For any semi-abelian $S$-scheme $G$,\footnote{for us,
 a semi-abelian $S$-scheme is a commutative group $S$-scheme
 which is a global extension of an abelian $S$-scheme by an $S$-torus.}
 the motive $\hgrp G$ is concentrated in degree $0$ for the
 effective $\delta$-homotopy $t$-structure on $\DM^{\delta-\eff}(S,\QQ)$.

 Moreover, the functor
$$
G \mapsto \hgrp G
$$
from the category of semi-abelian
 $S$-schemes up to isogeny to the category of $\delta$-effective
 homotopy modules with rational coefficients
 is fully faithful, exact and
 its essential image is stable under extensions.
\end{thm*}
Note also that, thanks to a theorem of Kahn and Yamazaki,
 we are able to completely determine the fibers 
 of the homotopy module $\hgrp G$ at a
 point $x:\spec(E) \rightarrow S$. In particular,
 we obtain: $\rH^\delta_0(G)(x,0)=G(E)$
 while when $G=A$ is an abelian scheme and the abelian variety
 $A_x$ is isogeneous to the Jacobian of a given smooth projective curve
 $C/E$, $\rH^\delta_0(A)(x,1)$ is the Bloch group $V(C)$ (cf. \cite{Blo})
 attached to $C$
 (see Th. \ref{thm:compute_sab_hlg} for more details).
 This shows that the $\ZZ$-grading
 of homotopy modules has an important arithmetic meaning.

As an application of this result, we finally give the generalization
 of the computation of the motive of a smooth curve
 originally due to Lichtenbaum and Voevodsky
 (cf. \cite[chap. 5, 3.2.4]{FSV}).
\begin{thm*}[Prop. \ref{prop:compute_hlg_curves} and
 \ref{prop:compute_hlg_curves_aff}]
Let $S$ be a regular scheme with dimension function $\delta=-\codim_S$.
Let $C/S$ be a smooth geometrically connected curve with a section
 $x \in C(S)$. Assume one of the following conditions holds:
\begin{itemize}
\item $C/S$ is projective;
\item $C$ admits a smooth compactification whose complement is 
 \'etale over $S$.
\end{itemize}
In these cases, the homology of the $\delta$-effective motive $M_S(X)$
 for the $\delta$-homotopy $t$-structure can be computed as follows:
$$
H_i^{\delta-\eff}(M_S(C)) \simeq \begin{cases}
\un_S \oplus \hgrp{J(C)} & \text{if } i=0, \\
\un_S\gtw 1 & \text{if $i=1$ and $C/S$ is projective}, \\
0 & \text{otherwise},
\end{cases}
$$
where $J(C)$ is the semi-abelian (dual) Albanese scheme
 of $C/S$.
\end{thm*}
%

\bigskip

\noindent \textbf{Future work.}
The theory developed in this paper opens the way to several
 interesting tracks which we plan to treat separately.
 The first one is the study of the spectral sequence associated
 with the truncations
 for the $\delta$-homotopy $t$-structure of motives/spectra of the form
 $f_*(\E)$ for a given morphism $f$, which we will call the 
 \emph{Leray $\delta$-homotopy spectral sequence}.
 We already know this spectral sequence is
 a generalization of Rost spectral sequence introduced in \cite{Ros}
 which was used recently by Merkurjev to prove a conjecture of Suslin.
 Secondly, we plan to prove a result which was conjectured by Ayoub
 for his perverse homotopy $t$-structure and which generalizes the result
 already obtained by the second named author in his Ph. D. thesis:
 the category of homotopy modules over a $k$-scheme $S$ is equivalent
 to Rost's category of cycle modules over $S$ (at least after
 inverting the characteristic exponent). This result is clear
 enough at the moment so that we can safely announce its proof
 (as well as a generalization without a base field for
 rational coefficients).

\section*{Detailed plan}
The first section is devoted to a recall on the main technical tools
 we will be using: dimension functions,  generated $t$-structures
 and various aspects of the motivic homotopy formalism. \\
The second section will first be devoted to the definition
 of our $\delta$-homotopy $t$-structures, as well as of the
 setting of  $\delta$-effective categories.
 These definitions are 
 illustrated by examples, which we hope
 will help the reader, and by the study of the basic properties
 of our definitions with respect to the six operations,
 such as the gluing property.
 The last two subsections of Section 2 are devoted to the
 exposition of several improved descriptions of the generators
 for the $\delta$-homotopy $t$-structure (stable and effective case)
 as well as of their application, such as the comparison to the
 definitions of Voevodsky, Morel and Ayoub. \\
The third section is devoted to the proof of our main
 theorem (number \ref{thm*:hlg&htp_t} in the above introduction),
 whose proof is given in subsection \ref{sec:main_thm}.
 The section starts by an extension of the ideas of Bloch and Ogus
 on the niveau filtration and associated spectral sequence
 using dimension functions. This spectral sequence is applied
 to get the generalized link between Borel-Moore motivic homology
 and Chow groups (without a base field) as stated above
 (theorem number \ref{prop*:comput_BM_w_Chow} above).
 Then we define and study the \emph{fiber $\delta$-homology}
 and relate it with the $E_1$-term of the $\delta$-niveau spectral
 sequence. After the proof of the main theorem, several examples
 and corollaries are stated. \\
The last section contains a detailed study of the heart
 of the $\delta$-homotopy $t$-structure, both in the stable
 and effective case, the comparison of the hearts of
 motives and spectra, and finally the link with semi-abelian schemes
 over regular bases.

\section*{Special thanks}
The authors want to thank Alexey Ananyevskiy, Denis-Charles Cisinski,
 Ofer Gabber, Pablo Pelaez and Igor Zhukov for helpful discussions
 and motivations.
 They also warmly thank Simon Pepin Lehalleur for detailed
 and useful comments on a preliminary version of these notes.
 Finally, the authors are grateful towards the referee
 for his very careful reading and many remarks that have led us
 to improve the exposition and mathematics of this paper.

\section*{Conventions} \label{conventions}

\noindent \textbf{Schemes}. In all this work,
 we will fix a subcategory $\base$
 of the category of excellent schemes of finite
 dimension which is closed under pullbacks by morphisms of finite type,
 finite sums and such that if $S$ is a scheme in $\base$,
 any localization of a quasi-projective $S$-scheme is in $\base$.
 Without precision, any scheme is supposed to be an object of $\base$.

Given a scheme $S$, a \emph{point} of $S$
 will be a morphism $x:\spec(E) \rightarrow S$
 such that $E$ is a finitely generated extension
 of the residue field of the image of $x$ in $S$.
 When $E$ is a given field, 
  we also say that $x$ is an $E$-valued point
  and denote by $S(E)$ the set of $E$-valued points.
The field $E$ will be called the field of definition of the point $x$,
 and we sometimes denote it generically by $K_x$.
 We will also denote by $\kappa_x$ the residue field of the image of $x$
 in $S$.
 Note that the extension fields $K_x/\kappa_x$ is of finite transcendence
 degree by assumption. We will call simply \emph{transcendence degree}
 of $x$ the transcendence degree of the field extension $K_x/\kappa_x$. 

When we will need to be precise,
 we will say that $x$ is a \emph{set-theoretic point} of $X$
 if $\kappa_x=K_x$. This is also assumed
 by the sentence: ``let $x \in X$ be a point''.

\bigskip \noindent \textbf{Morphisms}. Unless
 explicitly stated, separated 
 (resp. smooth) morphisms of schemes
 are assumed to be of finite type (resp.
 separated of finite type). We will say \emph{lci}
 for ``local complete intersection''.

A morphism $f:X \rightarrow S$ is said to
 be \emph{essentially of finite type} if $X$ is
 pro-\'etale over an $S$-scheme of finite
 type.\footnote{Traditionally, the condition
 pro-open replaces pro-\'etale but the
 arguments used here does not change in
 this greater generality.}
 To be consistent with this convention,
 we will say that an $S$-scheme $X$ is \emph{essentially separated}
 if it is essentially of finite type and separated.
 Given a scheme $S$ in $\base$,
 we will denote as usual by $\basex S$ the category
 of schemes in $\base$ over $S$.

Unless explicitly stated, quasi-finite morphisms
 will be assumed to be essentially of finite type.

\bigskip \noindent \textbf{Triangulated categories}.
 Recall the following classical conventions.

A subcategory of a triangulated category which
 admits coproducts is called \emph{localizing}
 if it is stable under extensions, suspensions
 and arbitrary coproducts (which implies it is stable under
 retracts).
In a triangulated category $\C$, a class of objects
 $\cG$ is called \emph{generating} 
 -- we also say that \emph{$\cG$ is a generating class}
 of the triangulated category $\C$ --
 if for any object $K$ of $\C$ one has the implication:
$$
\big(\forall X \in \cG, \forall i \in \ZZ, \Hom_\C(X[i],K)=0\big)
 \Rightarrow K=0.
$$
When $\C$ admits arbitrary coproducts,
 we say that an object $K$ of $\T$ is \emph{compact} if
 for any family $(X_i)_{i \in I}$ of objects the
 canonical map
$$
\bigoplus_{i \in I} \Hom(K,X_i) \rightarrow 
 \Hom(K,\oplus_{i \in I} X_i)
$$
is an isomorphism.

\bigskip \noindent \textbf{Motivic homotopy theory}. 
We will use the language of premotivic and motivic
 triangulated categories developed in \cite{CD3}
 as an extension of the work of Ayoub \cite{ayoub1}.
 Given a premotivic triangulated category $\T$,
 objects of $\T(S)$ will be called \emph{$\T$-spectra over
 $S$}.\footnote{Here the terminology differs slightly from that of
  \cite{CD3}, but it seems now more natural -- motives should be seen
	 as particular kind of spectra.}
 Recall we say that $\T$ is
 \emph{compactly generated by its Tate twists}
 (\emph{op. cit.} Definition 1.3.16)
 if for any scheme $S$,
 the geometric $\T$-spectra of the form $M_S(X,\T)(i)$
 for a smooth $S$-scheme $X$
 and an integer $i \in \ZZ$ are compact and form a set of generators of
 the triangulated category $\T(S)$.
 Recall also that when $\T$ is a \emph{motivic} triangulated category
  (\emph{op. cit.} Definition 2.4.45), the fibred category $\T$ is equipped
	with the 6 functors formalism as stated in \emph{op. cit.} Theorem 2.4.50.

{\bf By convention, all premotivic triangulated categories
 appearing in this text will be assumed to be compactly generated
 by their Tate twists.}

We will have to use special combinations of twists and shifts,
 so we adopt the following conventions for any $\T$-spectrum $K$
 and any integer $n \in \ZZ$:
\begin{align*}
K\dtw n&:=K(n)[2n], \\
K\gtw n&:=K(n)[n].
\end{align*}

In the text, we will fix a triangulated motivic category
 $\T$ over $\base$ (compactly generated by its Tate twists).
 When no confusion can arise, 
  objects of $\T$ will abusively be called \emph{spectra}.
 
\section{Preliminaries}

\subsection{Reminders on dimension functions}\label{sec:dim}

\begin{num}
Let $X$ be a scheme and $x$, $y$ two points of $X$.
 Recall that $y$ is called a \emph{specialization} of $x$ if $y$ belongs to the closure
 $Z$ of $x$ in $X$. Moreover, $y$ is called an \emph{immediate specialization} of $x$
 if $\codim_Z(y)=1$. Recall the following definition
 from \cite[2.1.10, 2.1.6]{PS_dim}:
\end{num}
\begin{df}\label{df:dimensional}
A \emph{dimension function} on $X$ is a map
 $\delta:X \rightarrow \ZZ$ such that for any immediate specialization $y$
  of a point $x \in X$, $\delta(x)=\delta(y)+1$.

In this case, we will say that $(X,\delta)$ is a \emph{dimensional scheme}.
\end{df}
We will say that $\delta$ is \emph{non-negative},
 and write $\delta \geq 0$,
 if its image lies in the set of non-negative integers.
 
\begin{rem}
According to our global assumptions, $X$ is noetherian.
 Thus if $X$ admits a dimension function, it is automatically universally
 catenary (see \emph{loc. cit.} 2.2.6).
\end{rem}

The following lemma is obvious according to the above definition:
\begin{lm}\label{lm:dim+locally_ct}
If $\delta$ and $\delta'$ are dimension functions on a scheme $X$,
 the function $\delta-\delta'$ is Zariski locally constant.
\end{lm}
In particular, up to an element in $\ZZ^{\pi_0(X)}$, there exists
 at most one dimension function on $X$. 

\begin{ex}\label{ex:can_dim_fn}
If $X$ is universally catenary and integral (resp. equicodimensional\footnote{\emph{i.e.}
 the codimension of any closed point of $X$ is equal to the dimension of its connected
 component. The main example is that of disjoint unions of equidimensional
 schemes of finite type over a field or over the ring of integers;})
 then 
$$
\delta(x)=-\dim(\mathcal O_{X,x})=-\codim_X(x) \
 \left(\text{resp. } \delta(x)=\dim\big(\overline{\{x\}}\big)\right)
$$
 is a dimension function on $X$ -- see \cite[2.2.2, resp. 2.4.4]{PS_dim}.
 In the \emph{resp.} case, $\delta$ is called the \emph{Krull dimension
 function on $X$}. Note that when $\dim(X)>0$, these two dimension
 functions, if defined, do not coincide.
\end{ex}

\begin{rem}\label{rem:dim_fn_bounded}
Since $X$ is noetherian
 finite dimensional (according to our conventions),
 a dimension function on $X$ is always bounded. 
 So the preceding lemma implies that if $X$
 admits a dimension function, then it admits at least one
 non-negative dimension function.
\end{rem}

\begin{num}\label{num:induced_delta}
Let $f:Y \rightarrow X$ be a morphism of schemes
 and $\delta$ be a dimension function on $X$.

If $f$ is quasi-finite,
 then $\delta \circ f$ is a dimension function
 on $Y$ (cf. \cite[2.1.12]{PS_dim}).

Assume now that $f$ is essentially of finite type.
 For any point $y \in Y$, $x=f(y)$, we put:
$$
\delta^f(y)=\delta(x)+\dtr\lbrack\res y/\res x\rbrack.
$$
According to \cite[2.5.2]{PS_dim}
 and the preceding remark, this defines a dimension function on $Y$
 -- when $f$ is pro-\'etale, one simply has $\delta^f=\delta \circ f$.
 Moreover, we will put:
$$
\delta(Y)=\max_{y \in Y} \Big(\delta^f(y)\Big).
$$
We easily deduce from that definition that
 if $Z \xrightarrow g Y \xrightarrow f X$ are morphisms essentially
 of finite type,
 then $(\delta^f)^g=\delta^{gf}$.
 This implies in particular that $\delta(Z)=\delta^f(Z)=\delta^{gf}(Z)$.
\end{num}

\begin{rem}\label{rem:bound_dim_not}
According to the preceding remark, for any essentially of finite type
 $X$-scheme $Y$ the following integers
$$
\delta_-(Y)=\min_{y \in Y}(\delta(y)), \
 \delta_+(Y)=\max_{y \in Y}(\delta(y)),
$$
are well defined. Note that one can check that the image of $\delta$
 is exactly $[\delta_-(Y),\delta_+(Y)]$.
 Note that $\delta_+(Y)=\delta(Y)$; we will use the notation 
 $\delta_+(Y)$ only in conjunction with the notation $\delta_-(Y)$.
\end{rem}

Recall that given a point $x$ of a scheme $X$,
 the \emph{dimension of $X$ at $x$}, denoted by $\dim_x X$,
 is the integer $\lim_U\big(\dim(U)\big)$, where $U$ runs over the open neighbourhoods
 of $x$ in $X$.
We will adopt the following definition, which is taken
 from several theorems of Chevalley as stated in \cite[section 13]{EGA4}:
\begin{df}
Let $f:Y \rightarrow X$ be a morphism essentially of finite type.
 The \emph{relative dimension of $f$} is the function denoted
 by $\dim(f)$ which to a point $y \in Y$,
 associates the integer $\dim_y f^{-1}f(y)$.
\end{df}
Beware $\dim(f)$ is not a dimension function on $Y$; in relevant
 cases (for example smooth, flat
 or more generally equidimensional
 morphisms\footnote{\label {foot:equidim}Recall
 from \cite[13.3.2]{EGA4} that a morphism $f:Y \rightarrow X$
 is equidimensional if the following conditions are fulfilled:
\begin{itemize}
\item $f$ is of finite type;
\item $\dim(f)$ is Zariski locally constant on $Y$;
\item $f$ maps generic points of $Y$ to generic points of $X$.
\end{itemize}
According to \cite[chap. 2, 2.1.8]{FSV},
 a flat morphism $f:Y \rightarrow X$ is equidimensional
 if for any connected component $Y'$ of $Y$
 and any generic point $x \in X$,
 $Y'_x$ is equidimensional.}),
 it is locally constant.
 In general, we will use upper bounds on this function.

\begin{ex}\label{ex:quasi-finite}
Let $f$ be a morphism of schemes which is locally of finite type.
 Then $\dim(f)=0$ if and only if it is locally quasi-finite
 (cf. \cite[28.29.5]{stack}).
 This can obviously be extended to the case where $f$ is essentially
 of finite type.
\end{ex}

\begin{rem}\label{rem:trivial_inequality}
It is clear that $\dtr(\kappa_y/\kappa_x) \leq \dim_y f^{-1}(x)$
 where $x=f(y)$ and equality happens if
 $y$ is a generic point of $f^{-1}(x)$ whose residue
 field is of maximal transcendence degree over $\res x$.
\end{rem}

\begin{prop}\label{prop:delta-base_change}
Consider a cartesian square of schemes
$$
\xymatrix@=10pt{
Y\ar^g[r]\ar_q[d] & X\ar^p[d] \\
T\ar^f[r] & S
}
$$
made of essentially of finite type morphisms, and let $\delta$
 be a dimension function on $S$.
\begin{enumerate}
\item If $\delta \geq 0$ then $\delta(Y) \leq \delta(X)+\delta(T)$.
\item If $\dim(f) \leq d$ then $\delta(Y) \leq \delta(X)+d$.
\item Assume $f$ is surjective on generic points of $S$
 and $\dim(f)=d$. \\
 Then $\delta(T)=\delta(S)+d$.
\end{enumerate}
\end{prop}
\begin{proof}
Put $h=f \circ q$.
Consider a point $y \in Y$ whose image is $x$, $t$, $s$ respectively
 in $X$, $T$ and $S$. Recall that $\res y$ is a composite extension
 field of $\res x/\res s$ and $\res t/\res s$ (\cite[3.4.9]{EGA1}).
 Thus, its transcendence degree is the sum of the transcendence
 degrees of this last two extension fields.

We prove $(1)$ as follows:
\begin{align*}
\delta^h(y)&=\delta(s)+\dtr(\res y/\res s)
 =\delta(s)+\dtr(\res x/\res s)+\dtr(\res t/\res s) \\
 &\leq \delta(s)+\dtr(\res x/\res s)+\delta(s)+\dtr(\res t/\res s)
 =\delta^p(x)+\delta^f(t) \leq \delta(X)+\delta(T);
\end{align*}
and (2) as follows (using Remark \ref{rem:trivial_inequality}):
\begin{align*}
\delta^h(y)&=\delta(s)+\dtr(\res y/\res s)
 =\delta(s)+\dtr(\res x/\res s)+\dtr(\res t/\res s)
 = \delta^p(x)+\dtr(\res t/\res s) \\
 & \leq \delta(X)+d.
\end{align*}
To prove point (3), it remains to show the converse inequality of (2)
 in the particular case $X=S$ and $\dim(f)=d$.
 Consider a generic point $s$ of $S$ such that $\delta(s)=\delta(S)$
 (it exists since  $\delta$ is bounded, Remark \ref{rem:dim_fn_bounded}).
 By our assumption on $f$, $f^{-1}(s) \neq \emptyset$.
 Let $x$ be a generic point of $f^{-1}(s)$ whose residue field
 is of maximal transcendence degree over $\res s$.
 According to Remark \ref{rem:trivial_inequality} and the assumption
 on $f$, one has $\dtr(\res x/\res s)=d$.
 Then the following computation allows us to conclude:
$$
\delta^f(x)=\delta(s)+\dtr(\res x/\res s)
 =\delta(S)+d.
$$
\end{proof}
 
\subsection{Reminders on t-structures}\label{sec:t-struct}

\begin{num}\label{num:conventions_t}
Following Morel (\cite{Mor1}), our conventions for t-structures
 will be homological.\footnote{One goes from homological to cohomological conventions
  by the usual rule: $\C^{\leq n}=\C_{\geq -n}$.}
 Apart from that, we will follow the classical
 definitions of \cite[Sec. 1.3]{BBD}.

This means in particular that a t-structure $t$
 on a triangulated category $\C$
 is a pair of subcategories $(\C_{\geq 0},\C_{<0})$ such that
\begin{enumerate}
\item $\Hom_\C(\C_{\geq 0},\C_{<0})=0$;
\item $\C_{\geq 0}$ (resp. $\C_{<0}$) is stable by suspension $[+1]$
 (resp. desuspension $[-1]$).
\item for any object $M$ in $\C$, there exists a distinguished triangle
 of the form
\begin{equation}\label{eq:proto_truncation}
M_{\geq 0} \rightarrow M \rightarrow M_{<0} \xrightarrow{+1}
\end{equation}
such that $M_{\geq 0} \in \C_{\geq 0}$ and $M_{<0} \in \C_{<0}$.
\end{enumerate}
 Then we denote as usual $\C_{\geq n}:=\C_{\geq 0}[n]$
 (resp. $\C_{\leq n+1}:=\C_{<n}:=\C_{<0}[n]$) for any integer $n \in \ZZ$.
 An object $M$ of $(\C_{\geq 0}$ (resp. $\C_{<0})$) is called
 \emph{non-$t$-negative}
 (resp. \emph{$t$-negative}); sometimes we indicate this by the notation
 $M \geq 0$ (resp. $M<0$) and similarly for any  $n \in \ZZ$.
 
Recall also that the triangle \eqref{eq:proto_truncation} is unique.
Thus, we get a well-defined functor $\tau^t_{\geq 0}:\C \mapsto \C_{\geq 0}$
 (resp. $\tau^t_{<0}:\C \mapsto \C_{<0}$)
 which is right (resp. left) adjoint to the inclusion functor
 $\C_{\geq 0} \rightarrow \C$ (resp. $\C_{<0} \rightarrow \C$).
 For any integer $n \in \ZZ$ and any object $M$ of $\C$,
  we put more generally: $\tau^t_{\geq n}(M)=\tau^t_{\geq 0}(M[-n])[n]$
	(resp. $\tau^t_{<n}(M)=\tau^t_{<0}(M[-n])[n]$).
 
The pair $(\C,t)$ is called a \emph{$t$-category} and,
 when $t$ is clear, it is convenient not to indicate it in the notation.
\end{num}

\begin{num} Let $\C$ be a $t$-category. The \emph{heart} of $\C$
 is the category $\hrt \C=\C_{\leq 0} \cap \C_{\geq 0}$.
 It is an abelian category and the canonical functor:
$$
H_0=\tau_{\leq 0}\tau_{\geq 0}=\tau_{\geq 0}\tau_{\leq 0}:\C
 \rightarrow \hrt \C
$$
is a \emph{homological functor} -- \emph{i.e.} converts distinguished
 triangles into long exact sequences. We put $H_n:=H_0(.[-n])$.
\end{num}

We give the following definition,
 slightly more precise than usual.
\begin{df}
Let $t$ be a $t$-structure on a triangulated category $\C$.
 We say that $t$ is \emph{left} (resp. \emph{right})
 \emph{non-degenerate} if $\cap_{n \in \ZZ} \C_{\leq n}=\{0\}$
 (resp. $\cap_{n \in \ZZ} \C_{\geq n}=\{0\}$).).
 We say that $t$ is \emph{non-degenerate} if it is left
 and right non-degenerate.
\end{df}
Another way of stating that $t$ is non-degenerate
 is by saying that the family of homology functors $(H_n)_{n \in \ZZ}$
 (defined above) is conservative.

\begin{num}
Consider two $t$-categories $\C$ and $\D$. One says that
 a functor $F:\C \rightarrow \D$ is \emph{left $t$-exact}
 (resp. \emph{right $t$-exact})
 if it respects negative (resp. positive) objects.
 We will say that $F$ is \emph{$t$-exact} if it is left and right
 $t$-exact. We say $F$ is an \emph{equivalence of $t$-categories} 
 if it is $t$-exact and an equivalence of the underlying categories.
 Then, the functor $H_0F$ induces an equivalence between
 the hearts.

More generally, for any $n \in \ZZ$,
 one says $F$ has \emph{(homological) amplitude}
 less (resp. more) than $n$ if $F(\C_{\leq 0}) \subset \C_{\leq n}$
 (resp. $F(\C_{\geq 0}) \subset \C_{\geq n}$).
 Similarly, $F$ has \emph{amplitude} $[n,m]$ is $F$
 has amplitude more than $n$ and less than $m$.

Recall that given a pair of adjoint functors
 $F:\C \leftrightarrows \D:G$,
 $F$ is right $t$-exact if and only if $G$ is left $t$-exact.
 In that case, we will say that $(F,G)$ is an \emph{adjunction 
 of $t$-categories}. According to \cite[1.3.17]{BBD},
 such an adjunction induces a pair of adjoint functors:
\begin{equation}\label{eq:proto_t_adj}
\xymatrix@C=40pt{
H_0F:=(\tau_{\geq 0} \circ F):\hrt \C\ar@<0pt>[r] & 
 \hrt \D:(\tau_{\leq 0} \circ G)=:H_0G.\ar@<-4pt>[l]
}
\end{equation}
 
Suppose in addition that $\C$ is a monoidal triangulated category
 with tensor product $\otimes$ and neutral object $\un$.
 One says that the tensor structure is \emph{compatible with the
 $t$-category $\C$} if the bi-functor $\otimes$ sends
 two non-negative objects
 to a non-negative object and $\un$ is a non-negative object.
 In this case, one checks that the formula:
$$
K \ohrt L:=\tau_{\leq 0}(K \otimes L)=H_0(K \otimes L)
$$
defines a monoidal structure on the abelian category $\hrt\C$
such that the functor $H_0$ is monoidal.
If moreover the tensor structure is closed,
 then for any $K \geq 0$ and $L \leq 0$,
 one has $\uHom(K,L) \leq 0$. Thus 
 the bifunctor $\tau_{\geq 0} \uHom$
 defines an internal Hom in $\hrt \C$.
\end{num}

\begin{num}\label{num:generated_aisle}
Let $\C$ be a compactly generated triangulated category which
 admits arbitrary small coproducts.
 We will say that a subcategory $\C_0$ of $\C$ is
 \emph{stable by extensions}
 if given any distinguished triangle $(K,L,M)$ of $\C$
 such that $K$ and $M$ belong to $\C_0$, the object $L$
 also belongs to $\C_0$.

Given a family of objects $\cG$ of $\C$,
 we will denote by $\gen \cG$ the smallest subcategory
 of $\C$ which contains $\cG$ and is stable under extensions,
 positive suspensions, and arbitrary (small) coproducts.
The following theorem was proved in \cite[A.1]{AJS};
 see also \cite[tome I, 2.1.70]{ayoub1}.
\end{num}
\begin{thm}\label{thm:gen_tstructures}
Adopt the previous assumptions.
 Then $\gen \cG$ is the category of non-negative objects
 of a $t$-structure $t_\cG$ on $\cG$.
Moreover, for this $t$-structure, positive objects are
 stable under coproducts, and the functors $\tau_{\leq 0}$,
 $\tau_{\geq 0}$, $H_0$ commutes with coproducts.
\end{thm}

\begin{df}\label{df:gen_tstruct}
In the assumptions of the previous theorem,
 the $t$-structure $t_\cG$ will be called
 the \emph{$t$-structure generated by $\cG$}.
\end{df}

\begin{rem}\label{rem:exist_colimits}
According to the preceding theorem,
 the heart $\hrt{\T}$ of the $t$-structure generated by $\cG$ admits small coproducts.
 So in particular, it admits small colimits.

Moreover, as $H_0$ commutes with coproducts,
 small coproducts are exact in $\hrt{\T}$. In other words,
 $\hrt{\T}$ satisfies Grothendieck properties (AB3) and (AB4) of \cite[1.5]{tohoku}.
 Checking whether filtered colimits are exact in $\hrt{\T}$ and finding generators for 
 this abelian category appears to be a much more complicated problem; see 
 \cite{PaSa} or \cite{Bon16}.
 In the case of interest for this paper, see \ref{thm:finally_Grothendieck}.
\end{rem}

A nice remark about such generated t-structures is the
 following elementary lemma.
\begin{lm}\label{lm:non-neg-degenerated}
Consider the notations of the above theorem.
 Then the following conditions are equivalent:
\begin{enumerate}
\item[(i)] $\cG$ is a generating family of the triangulated
 category $\C$;
\item[(ii)] $t_\cG$ is left non-degenerate.
\end{enumerate}
\end{lm}
\begin{proof}
This immediately follows from the equality:
$$
\cap_{n \in \ZZ} \C_{\geq n}=
 \big\{L \in \C
 \mid \forall K \in \cG, \forall n \in \ZZ, \Hom(K[n],L)=0\big\}.
$$
\end{proof}

\begin{rem}
The (right) non-degeneracy of a $t$-structure generated
 as above is much more delicate.  We will
 obtain such a result for our particular $t$-structures in 
 Corollary \ref{cor:nice_ppties_dhtp}.
 See also \cite[2.1.73]{ayoub1} for an abstract criterion.
\end{rem}

\subsection{Reminders on motivic homotopy theory}\label{s1.3}\label{sec:six}

\begin{ex}\label{ex:motivic_cat}
This work can be applied to
 the following triangulated motivic categories:
\begin{enumerate}
\item\label{i1} The stable homotopy category $\SH$ of Morel and
 Voevodsky (see \cite{ayoub1}).
\item\label{i2} The $\AA^1$-derived category $\DAx R$ of Morel
 (see \cite{ayoub1}, \cite{CD3}) with coefficients in a ring $R$.
\item\label{i3} The category $\MGLmod$ of $\MGL$-modules
 (cf. \cite[7.2.14, 7.2.18]{CD3}).
 Its sections over a scheme $S$ are made
 by the homotopy category of the model category
 of modules over the (strict) ring spectrum $\MGL_S$
 in the monoidal category of symmetric spectra.
\item\label{i4} The category of Beilinson motives $\DMB$
 (see \cite{CD3}) which is equivalent
 (by \cite[16.2.22]{CD3})
 to the \'etale $\AA^1$-derived category $\mathrm{DA}_{\text{\'et},\QQ}$
 with rational coefficients introduced in \cite{ayoub1}.
 It can have coefficients in any $\QQ$-algebra $R$.
\item\label{i5} Let $\MZ$ be the ring spectrum in $\SH(\spec(\ZZ))$
 constructed by Spitzweck in \cite{Spi}.
 Then the category $\MZmod$ is a motivic category
 over the category of all schemes (see \cite{Spi}).
\item\label{i6} Let $F$ be a prime field of characteristic exponent
 $p$ and $\base$ be the category of (excellent) $F$-schemes.
 The category of $\cdh$-motivic complexes $\DM_\cdh(-,R)$
 with coefficients in an arbitrary $\ZZ[1/p]$-algebra $R$,
 is a motivic category over $\base$ (see \cite{CD4}).

 According to \cite[Th. 5.1]{CD4} and \cite[Th. 9.16]{Spi},
 this motivic category for $R=\ZZ[1/p]$
 is equivalent to $\MZmod[1/p]$ restricted
 to $\basex F$ through the unique map $\spec(F) \rightarrow \spec(\ZZ)$.

Recall also from \cite[5.9]{CD4} that for any regular $F$-scheme $X$,
 the canonical $\cdh$-sheafifi\-cation map
\begin{equation}
\label{eq:DM&DM_cdh}
\DM(X,R) \rightarrow \DM_\cdh(X,R)
\end{equation}
is an equivalence of triangulated monoidal categories, where the
 left hand side is the non-effective version of Voevodsky's derived
 category of motivic complexes (see \cite{CD1,CD3,CD4}).
\end{enumerate}
All these categories are compactly generated by their Tate twists
  fulfilling the requirements of our conventions
 (see page \pageref{conventions}).

Note that $\DM_\cdh$ is a premotivic category over the category
 of all noetherian schemes, but we only know it satisfies the localization
 property over $F$-schemes as explained above.

Recall also that  examples (\ref{i3})--(\ref{i6}) are oriented (\cite[2.4.38]{CD3}),
 whereas examples (\ref{i1}) and (\ref{i2}) are not.
\end{ex}

\begin{num}\label{num:convention_DM}
 To gather examples (4), (5) and (6) together,
 we will adopt a special convention. We fix a ring
 $R$ and denote by $\DM_R$ one of the following motivic categories:
\begin{enumerate}
\item if $R$ is a $\QQ$-algebra,
 it is understood that $\DM_R=\DMB$
 and $\base$ is the category of all schemes;
\item if $\base$ is the category of $F$-schemes for a prime field $F$
 with characteristic exponent $p$, it is understood that $p \in R^\times$
 and $\DM_R=\DM_\cdh(-,R)$;
\item if $R=\ZZ$, it is implicitly understood that $\DM_R=\MZmod$
 and $\base$ is the category of all schemes.
\end{enumerate}
\end{num}

\begin{ex}
For any ring $R$ satisfying one of the above assumptions,
 the triangulated motivic categories of the previous example
 are related by the following diagram:
\begin{equation}\label{eq:premotivic_adj}
\begin{split}
\xymatrix@R=4pt{
 & \ \MGLmod & &\\
SH\ar^/-3pt/{\phi^*}[ru]\ar_{\delta^*}[rd] & & \\
 & \DAx R\ar_{\gamma^*}[r] & \DM_R. &
}
\end{split}
\end{equation}
In the above diagram,
 the label used stands for the left adjoints;
 the right adjoints will be written with a lower-star
 instead of an an upper-star.

The references are as follows: $\delta^*$, \cite[5.3.35]{CD3};
 $\gamma^*$, \cite[11.2.16]{CD3};
 $\rho^*$, \cite[\textsection 9]{CD4};
 $\phi^*$, \cite[7.2.13]{CD3}.
 Recall also that the right adjoints of these premotivic
 adjunctions are all conservative.
\end{ex}

\begin{rem}\label{rem:MGL->DM}
Note that there exists a premotivic adjunction:
$$
\MGLmod \rightarrow \DM_R
$$
for any ring $R$ satisfying the assumption
 of \ref{num:convention_DM} though the whole details have
 not been written down.
 In fact, this will follow from the following two facts:
\begin{itemize}
\item $\DM_R(S)$ is equivalent to the homotopy category 
 of modules over the $E_\infty$-ring spectrum $\mathbf HR_S$
 with coefficients in $R$ representing motivic cohomology. Beware
 that in case (2), $\mathbf HR_S$ really means the
 \emph{relative motivic Eilenberg-Mac Lane spectrum} of $S/F$
 in the sense of \cite[3.8]{CD5}. It coincides with Voevodsky's motivic
 Eilenberg-Mac Lane spectrum in case $S$ is regular
 (see \cite[3.9]{CD5}) but this is not known when $S$ is singular.\footnote{It
 may be false but note that the cohomology represented by
 the ring spectra $\mathbf HR_S$ for various $F$-schemes $S$ is the unique one
 which coincides with motivic cohomology on regular simplicial schemes
 and satisfies $\cdh$-descent.}

In any cases of \ref{num:convention_DM}, the equivalence
 of $\DM_R(S)$ with $\mathbf HR_S$-modules is known:
 for (1) see \cite[14.2.11]{CD3},
 for (2) see the construction of \cite[5.1]{CD5}, though one
 does not consider $E_\infty$-ring spectra,
 for (3) it is obvious.
 Note however that some special care is needed about the $E_\infty$-structure
 as well as its compatibility with pullback isomorphisms:
$$
\derL f^*(\mathbf HR_S) \rightarrow \mathbf HR_T
$$
for $f:T \rightarrow S$.
This is the first part that needs a careful writing.
\item The ring spectrum $\mathbf HR_S$ is oriented,
 or equivalently, it is a $\MGL_S$-algebra which is also
 true in any of the cases of \ref{num:convention_DM}.
 The more difficult point is that the corresponding map
$$
\MGL_S \rightarrow \mathbf HR_S
$$
 is a morphism of $E_\infty$-ring spectra. The proof of this
 fact is indicated in \cite[11.2]{Spi}. The latter reference also indicates
 that this isomorphism is compatible with pullbacks
 in $S$ -- \emph{i.e.} with the cartesian structure on
 the family of ring spectra $(\MGL_S)$ and $(\mathbf HR_S)$
 indexed by schemes $S$ in $\base$.
\end{itemize}
Then the required map follows by functoriality of the construction of
 the homotopy category of modules over an $E_\infty$-ring spectrum.
\end{rem}

\begin{num} \textit{Constructible $\T$-spectra}.

Let us consider an abstract motivic triangulated category
 satisfying our general assumptions (see page \pageref{conventions}).
 In particular
 the $\T$-spectra of the form $M_S(X)(n)$ for a smooth $S$-scheme $X$
 and an integer $n \in \ZZ$ are compact.

Following the usual terminology (see \cite[4.2.3]{CD3}),
 we define the category  $\T_c(S)$ 
 of \emph{constructible $\T$-spectra over $S$}
 as the smallest thick triangulated subcategory of $\T(S)$
 containing the $\T$-spectra of the form $M_S(X)(n)$ as above.
 Then according to \cite[1.4.11]{CD3},
 a $\T$-spectrum is constructible if and only if it is compact.

Moreover, according to \cite[4.2.5, 4.2.12]{CD3},
 constructible $\T$-spectra are stable by the operations $f^*$,
 $p_!$ (p separated), and tensor product.
\end{num}

The following terminology first appeared in \cite{Riou}.
\begin{df}\label{df:BM}
Let $p:X \rightarrow S$ be a separated morphism.
 We define the Borel-Moore $\T$-spectrum associated with $X/S$
 as the following object of $\T(S)$: $\Mb(X/S):=p_!(\un_X)$.
\end{df}
Note that according to the previous paragraph,
 $\Mb(X/S)$ is a compact object of $\T(S)$.

\begin{num}\label{num:virtual_vb}
Let us recall (see \cite[2.4.12]{CD3}) that we define the Thom $\T$-spectrum
 associated with a vector space $E/X$ as the $\T$-spectrum:
$$
\MTh(E)=p_\sharp s_*(\un_X),
$$
where $p$ (resp. $s$) is the canonical projection
 (resp. zero section) of $E/X$. Recall for example
 that for any integer $r\geq 0$, we get under
 our conventions that $\MTh(\AA^r_S)=\un_S\dtw r$.

The Thom $\T$-spectrum $\MTh(E)$ is $\otimes$-invertible
 with inverse $s^!(\un_E)$.
 Following \cite[4.1.1]{Riou} (see also \cite{ayoub1}),
 the functor $\Th$ can be uniquely extended to a functor
 from the Picard category $\uK(X)$ of
 \emph{virtual vector bundles over $X$}
 (defined in \cite[4.12]{Del})
 to the category of $\otimes$-invertible $\T$-spectra
 and this functor satisfies the relations:
\begin{align*}
\MTh(v+w)&=\MTh(v) \otimes \MTh(w) \\
\MTh(-v)&=\MTh(v)^{\otimes,-1}.
\end{align*}
It is obviously compatible with base change. Note also that we will
 adopt a special notation when $X$ is a smooth $S$-scheme, with structural
 morphism $f$. Then for any virtual vector bundle $v$ over $X$,
 we put:
\begin{equation}\label{eq:base_Thom}
\MTh_S(v)=f_\sharp\big(\MTh(v)\big).
\end{equation}

Thom spaces are fundamental because they appear in the 
 following purity
 isomorphism of a smooth morphism $f:X \rightarrow S$
 with tangent bundle $T_f$:
\begin{equation}\label{eq:purity_iso}
f^! \simeq \MTh(T_f) \otimes f^*.
\end{equation}
Recall finally that, when $\T$ is oriented (\cite[2.4.38]{CD3}),
 for any virtual vector bundle $v$ over $S$ of (virtual) rank $r$,
 there exists a canonical  \emph{Thom isomorphism}:
\begin{equation}\label{eq:Thom_iso}
\MTh_S(v) \xrightarrow{\ \sim\ } \un_S\dtw{r}
\end{equation}
which is coherent with respect to exact
 sequences of vector bundles and compatible with base change
 (cf. \emph{loc. cit.}).
\end{num}

\begin{num}\label{num:BM}
In the text below,
 we will use the following properties of Borel-Moore $\T$-spectra
 which follow from the six functors formalism:
\begin{itemize}
\item[(BM1)] $\Mb(X/S)$ is contravariant in $X$ with 
 respect to proper morphisms of separated $S$-schemes.
\item[(BM2)] $\Mb(X/S)$ is covariant with 
 respect to \'etale morphisms $f:Y \rightarrow X$ of separated
 $S$-schemes.\footnote{This covariance can be extended to the
 case of any smooth morphism $f$ if one adds a twist
 by the Thom $\T$-spectrum of the tangent bundle of $f$.}
\item[(BM3)] (\emph{Localization}) For any closed immersion
 $i:Z \rightarrow X$ with complementary
 open immersion $j:U \rightarrow X$, one has a distinguished
 triangle:
\begin{equation} \label{eq:localization}
\Mb(U/S) \xrightarrow{j_*} \Mb(X/S)
 \xrightarrow{i^*} \Mb(Z/S)
 \xrightarrow{\partial_i} \Mb(U/S)[1].
\end{equation}
\item[(BM4)] (\emph{Purity}) For any smooth $S$-scheme $X$ of relative
 dimension $d$, there exists a canonical isomorphism in $\T(S)$:
$$
\pur_X:\MTh_S\big(-T_f\big) \rightarrow \Mb(X/S)
$$
where $T_f$ is the tangent bundle of $f$, and we have followed
 convention \eqref{eq:base_Thom}.
 When $\T$ is oriented, we get an isomorphism in $\T(S)$:
$$
\pur_X:M_S(X)\dtw{-d} \rightarrow \Mb(X/S)
$$
where $d$ is the relative dimension of $X/S$.
 \item[(BM5)] (\emph{K\"unneth}) For any separated $S$-schemes
 $X$, $Y$, there exists a canonical isomorphism:
$$
\Mb(X/S) \otimes \Mb(Y/S) \rightarrow \Mb(X \times_S Y/S).
$$
\item[(BM6)] (\emph{Base change}) For any morphism $f:T \rightarrow S$,
 $f^*\Mb(X/S)=\Mb(X \times_S T/T)$.
\end{itemize}
Each property follows easily from the six functors formalism;
 using the numeration of the axioms as in \cite[Th. 2.4.50]{CD3}:
\begin{itemize}
\item (BM1): (2) and the unit map of the adjunction $(f^*,f_*)$;
\item (BM2): (3) and the unit map of the adjunction $(f_!,f^!)$;
\item (BM3): the localization property ($\text{Loc}_i$);
\item (BM4): point (3)
 \emph{i.e.} the dual assertion of \eqref{eq:purity_iso}; (BM5): point (5) and (4); (BM6): point (4).
\end{itemize} 
\end{num}

\begin{rem}\label{rem:Mb_generators}
It easily follows from (BM4) that 
 $\T$-spectra of the form $\Mb(X/S)(i)$ for $X/S$ separated and $i\in \mathbb{Z}$
 are compact generators of the triangulated category $\T(S)$.
\end{rem}

\begin{df}
Let $X$ be a separated $S$-scheme. For any couple of integers
 $(p,q) \in \ZZ^2$,
 we define the \emph{Borel-Moore homology} of $X/S$ with coefficients
  in $\E$ and degree $(p,q)$ as the following abelian group:
$$
\E^{BM}_{p,q}(X/S)=\Hom_{\T(S)}\big(\Mb(X/S)(q)[p],\E\big).
$$
\end{df}
It will also be convenient to consider two other kinds of twists
 and therefore use the following notations:
\begin{align*}
\E^{BM}_{p,\gtw q}(X/S)&=\Hom_{\T(S)}\big(\Mb(X/S)\gtw q[p],\E\big), \\
\E^{BM}_{p,\dtw q}(X/S)&=\Hom_{\T(S)}\big(\Mb(X/S)\dtw q[p],\E\big).
\end{align*}

Besides, when dealing with a non orientable motivic category $\T$, 
 it is useful to adopt the following extended notations in cohomology
 -- note similar notations may be introduced for Borel-Moore homology
 but we will not used the latter in the present paper.
\begin{df}
Let $X$ be a scheme, $n$ an integer
 and $v$ a virtual vector bundle over $X$.
 We define the $\T$-cohomology of $X$ in bidegree $(n,v)$
 as the following abelian group:
$$
H^{n,v}(X,\T):=\Hom_{\T(X)}\big(\un_X,\Th(v)[n]\big).
$$
\end{df}

\begin{num}\label{num:notation_virtual_vb}
To simplify notations, when $f:Y \rightarrow X$ is a morphism
 of schemes and $v$ is a virtual vector bundle over $X$,
 we will denote by $H^{n,v}(Y,\T)$ the cohomology of $Y$
 with degree $(n,f^{-1}v)$.
 In particular, the pullback functor $f^*$ induces a pullback morphism
 on cohomology:
$$
f^*:H^{n,v}(X,\T) \rightarrow H^{n,v}(Y,\T).
$$
 We will denote by $\dtw m$ the virtual bundle associated
 with the $m$-dimensional affine space $\AA^m$ over $\ZZ$.
 Then by definition, one gets:
$$
H^{n,\dtw m}(X,\T)=\Hom(\un_X,\un_X\dtw m[n])
 =\Hom(\un_X,\un_X(m)[2m+n])
$$
recovering the index of twisted cohomologies that we have
 introduced above.
\end{num}

\begin{rem}\label{rem:Thom&orient}
Recall that when $\T$ is oriented, for any virtual $X$-vector bundle $v$
 of virtual rank $r$,
 there exists a canonical isomorphism (see for example \cite{Deg12}),
 called the \emph{Thom isomorphism}:
$$
\Th(v) \xrightarrow \sim \dtw r.
$$
In particular, for any $\T$-spectrum $\E$,
 one has a canonical isomorphism
 $\E^{n,v}(X) \simeq \E^{n,\dtw r}(X)$.
\end{rem}

The preceding notation is natural and useful
 to formulate the absolute purity
 property that will be used in this paper.
\begin{df}
A \emph{closed pair} is a pair of schemes $(X,Z)$
 such that $Z$ is a closed subscheme of $X$.
 One says $(X,Z)$ is \emph{regular}
 if the immersion of $Z$ in $X$ is regular.
 If $X$ is an $S$-scheme, one says $(X,Z)$ is a smooth $S$-pair
 if $X$ and $Z$ are smooth over $S$.

Given an integer $n \in \ZZ$ and a virtual vector bundle over $Z$,
 one defines the cohomology of $(X,Z)$ -- or cohomology of $X$ with
 support in $Z$ -- in degree $(n,v)$ as the abelian group:
$$
H^{n,v}_Z(X,\T):=\Hom_{\T(X)}\big(i_*(\Th(-v)),\un_X[n]\big).
$$
\end{df}

\begin{num}
Given a morphism of schemes $f:Y \rightarrow X$,
 on easily checks that the pullback functor $f^*$
 induces a morphism of abelian groups:
$$
f^*:H^{n,v}_Z(X,\T) \rightarrow H^{n,v}_{T}(Y,\T)
$$
where $T=f^{-1}(Z)$ and
 we have used the same convention as before for
 the group on the right hand side.
 We will also say $f$ is
 a \emph{cartesian morphism} of closed pairs
 $(Y,T) \rightarrow (X,Z)$.
\end{num}

\begin{num}\label{num:deformation_space}
The following observations are extracted
 from \cite{Deg12}.
Let $(X,Z)$ be a regular closed pair.
 We will denote by $N_ZX$ (resp. $B_ZX$) the normal cone
 (resp. blow-up) associated with $(X,Z)$.
 Recall the deformation space
 $D_ZX=B_{Z \times 0}(\AA^1_X)-B_ZX$.
 The scheme $D_ZX$ contains $D_ZZ=\AA^1_Z$ as a closed
 subscheme. It is fibred over $\AA^1$ (by a flat morphism),
 and its fibre over $1$ is $X$ while its fibre over $0$ is $N_ZX$.
 Thus we get a \emph{deformation diagram} of closed pairs:
$$
(X,Z) \xrightarrow{d_1} (D_ZX,\AA^1_Z)
  \xleftarrow{d_0} (N_ZX,Z).
$$
For the next definitions, we will fix a subcategory
 $\base_0$ of $\base$.
  We will call closed $\base_0$-pair any closed pair $(X,Z)$ 
 such that $X$ and $Z$ belongs to $\base_0$.
\end{num}
\begin{df}\label{df:abs_purity}
Consider the above assumptions.

We say that $\T$ satisfies \emph{$\base_0$-absolute purity}
 if for any regular closed $\base_0$-pair $(X,Z)$,
 any integer $n$ and any virtual 
 vector bundle $v$ over $Z$, the following maps are isomorphisms:
$$
H^{n,v}_Z(X,\T) \xleftarrow{d_1^*} H^{n,v}_{\AA^1_Z}(D_ZX,\T)
  \xrightarrow{d_0^*} H^{n,v}_Z(N_ZX).
$$
When $\base_0$ is the category of regular schemes in $\base$,
 we simply say $\T$ satisfies absolute purity.
\end{df}

\begin{ex} When $\base_0$ is the category of smooth schemes over
 some scheme $\Sigma$ in $\base$, $\T$ always satisfies
 the absolute purity property according to Morel and Voevodsky's 
 purity theorem \cite[Sec. 3, Th. 2.23]{MV}.
\end{ex}

\begin{num}\label{num:commutes_with_limits}
Let $S$ be a scheme and $\E$ be a $\T$-spectrum over $S$.
 Then for any $S$-scheme $X$ in $\base$, 
 we can define the $\E$-cohomology of $X$ as:
$$
\E^{n,i}(X)=\Hom_{\T(X)}(\un_X,f^*\E(i)[n])
$$
-- and similarly for the other indexes considered above.
We will say that \emph{$\E$-cohomology commutes with projective limits}
 if for any projective system of $S$-schemes $(X_i)_{i \in I}$
 in $\base$ which admits a limit $X$ in the category $\base$,
 the natural map 
$$
\left(\ilim_{i \in I} \E^{**}(X_i)\right) \rightarrow \E^{**}(X)
$$
of bigraded abelian groups is an isomorphism.
 When this property is verified for $\E$ the constant $\T$-spectrum $\un_S$,
 for any scheme $S$ in $\base$, we will also
 simply say that \emph{$\T$-cohomology commutes with projective
 limits}.

Extending slightly the terminology of
 \cite[4.3.2]{CD3},\footnote{which applies only to triangulated
 motivic categories which are the homotopy category associated
 with a model motivic category;} we say that
 \emph{$\T$ is continuous} if for any scheme $S$,
 and any $\T$-spectrum $\E$ over $S$,
 $\E$-cohomology commutes with projective limits.
 With this definition, all the results of \cite[Section 4.3]{CD3}
 work through and we will freely use them -- in fact,
 we will mainly use Proposition 4.3.4 of \emph{loc. cit.}

According to \cite[4.3.3]{CD3} (case 2,3,4,5), \cite{CD4}
 (case 6), and \cite[Example 2.6]{CD5} (case 1)
 all the motivic categories of Example \ref{ex:motivic_cat}
 are continuous. 
\end{num}

\begin{ex}\label{ex:abs_pur_equal}
Assume that $\T$-cohomology commutes with projective
 limits in the above sense and $\base$ is the category of all
 excellent noetherian finite dimensional schemes.
 Then $\T$ automatically satisfies the absolute purity with respect
 to the category of regular $F$-schemes for any prime field $F$
 (see \cite[Ex. 1.3.4(2)]{Deg12}). Again, this holds for all
 the motivic categories of Example \ref{ex:motivic_cat}.
\end{ex}

Recall one says a cartesian morphism of regular closed pairs
 $(Y,T) \rightarrow (X,Z)$ is \emph{transversal}
 if the induced map $N_T(Y) \rightarrow N_Z(X) \times_Z T$
 is an isomorphism. The following proposition is
 a straightforward generalization of \cite[A.2.8]{CD4}
 (see \cite{Deg12} for more details).
\begin{prop}
Let $\base_0$ be a subcategory of $\base$
 satisfying the following assumptions:
\begin{itemize}
\item for any regular closed $\base_0$-pair $(X,Z)$,
 the deformation space $D_ZX$ belongs to $\base_0$;
\item for any smooth morphism $X \rightarrow S$ of schemes,
 if $S$ belongs to $\base_0$ then $X$ belongs to $\base_0$.
\end{itemize}
 Then the following conditions are equivalent:
\begin{enumerate}
\item[(i)] $\T$ is $\base_0$-absolutely pure.
\item[(ii)] For any regular closed $\base_0$-pair $(X,Z)$,
 there exists a class $\eta_X(Z)$ in $H^{0,N_ZX}_Z(X,\T)$.
 The family of such classes satisfies the following properties:
\begin{enumerate}
\item For a vector bundle $E/Z$, and $(E,Z)$ being the closed pair
 corresponding to the $0$-section,
 the class $\eta_E(Z)$ corresponds to the identity of the Thom
 $\T$-spectrum $\Th(-E)$ through the identification:
$$
H^{0,N_ZE}_Z(E,\T) \simeq H^{0,E}_Z(E,\T)
=\Hom\big(\Th(-E),s^!(\un_Z)\big)
\simeq \Hom\big(\Th(-E),\Th(-E)\big).
$$
\item For any transversal morphism $f:(Y,T) \rightarrow (X,Z)$
 of regular $\base_0$-closed pairs, the following relation holds
 in $H^{0,N_ZX}_T(Y,\T)$:
$$
f^*\eta_X(Z)=\eta_Y(T).
$$
\item For any regular closed $\base_0$-pair $(X,Z)$,
 with closed immersion $i:Z \rightarrow X$,
$$
\eta_X(Z):\Th(-N_ZX) \rightarrow i^!(\un_X)
$$
is an isomorphism of $\T(Z)$.
\end{enumerate}
\end{enumerate}
Moreover, when these conditions hold, the properties (a) and (b)
 uniquely determines the family $\eta_X(Z)$ indexed by regular
 closed $\base_0$-pairs $(X,Z)$.
\end{prop}

\begin{cor}\label{cor:fund_class}
Assume $\T$ is $\base_0$-absolutely pure.

Then for any quasi-projective local complete intersection morphism
 $f:Y \rightarrow X$ such that $X$ and $Y$ belongs to $\base_0$,
 with virtual tangent bundle $\tau_f$, 
 we get an isomorphism in $\T(Y)$:
$$
\eta_f:\Th\big(\tau_f\big) \rightarrow f^!\big(\un_X\big).
$$
\end{cor}
This simply combines the preceding proposition with
 the relative purity isomorphism \eqref{eq:purity_iso}
 after choosing a factorisation of $f$ into a regular closed immersion
 followed by a smooth morphism.

\begin{rem}
 It can be shown that the classes $\eta_f$ are uniquely determined
 and in particular do not depend on the chosen factorization
 (see \cite{Deg12}) but we will not use this fact here.
\end{rem}

\section{The homotopy t-structure by generators}

\subsection{Definition}

\begin{df}\label{df:tstruct}
Let $(S,\delta)$ be a dimensional scheme (Definition \ref{df:dimensional}).

The \emph{$\delta$-homotopy $t$-structure over $S$},
 denoted by $t_\delta$, is the $t$-structure on $\T(S)$ generated by 
 $\T$-spectra of the form
 $\Mb(X/S)\gtw n[\delta(X)]$ (recall Definition \ref{df:BM})
 for any separated $S$-scheme $X$ and any integer
 $n \in \ZZ$.

Given a morphism $f:T \rightarrow S$ essentially
 of finite type, we will put $t_\delta=t_{\delta^f}$ as
 a $t$-structure on $\T(T)$. 
\end{df}
Note also that the $\delta$-homotopy $t$-structure
 is generated by $\T$-spectra of the form
 $$\Mb(X/S)\dtw{\delta(X)}\gtw n$$ for any separated $S$-scheme $X$
 and any integer $n \in \ZZ$.
 In the following, we will use this convention for the generators
 because it allows to treat the cases of the
 $\delta$-homotopy $t$-structure and the effective $\delta$-homotopy
 $t$-structure (cf. Definition \ref{df:eff_tstruct}) simultaneously.

\begin{rem}\label{rem:htp_t_nonnegdeg}
\begin{enumerate}
\item Given Lemma \ref{lm:non-neg-degenerated}
 and Remark \ref{rem:Mb_generators}, $t_\delta$ is left
 non-degenerate.
\item One readily deduces from the above definition that
 a spectrum $\E$ over $S$ is $t_\delta$-negative
 if and only if for any separated $S$-scheme $X$,
  the $\ZZ$-graded abelian group
$$
\E^{BM}_{p,\gtw *}(X/S)
$$
is zero as soon as $p \geq \delta(X)$.
\end{enumerate}
\end{rem} 

\begin{ex}
Assume $S$ is the spectrum of a field $k$ of 
 characteristic exponent $p$.
 We will show in the next section 
 (cf. Example \ref{ex:delta_htp-t-struct_field}) that 
 when $\T=\DM_\cdh[1/p]$,
 the $\delta_k$-homotopy $t$-structure
 coincides with the one defined in \cite{Deg8} on
 $\DM(k)[1/p]$.
 We will also show (\emph{ibid.}) that when $\T=\SH$
 and $k$ is a perfect field,
 the $\delta_k$-homotopy $t$-structure on $\SH(k)$
 coincides with that defined by Morel in \cite{Mor2}.
\end{ex}

The premotivic category $\T$ is additive: if $S=\sqcup_i S_i$,
 $\T(S)=\oplus_{i \in I} \T(S_i)$.
 In particular, given any family $\underline n=(n_i)_{i \in I}$ of
 integers,
 we can define the $n$-suspension functor $\Sigma_{\underline n}$ 
 as the sum over
 $i \in I$ of the $n_i$-suspension functor on $\T(S_i)$.
 The following lemma is obvious:
\begin{lm}\label{lm:indep_delta}
Let $(S_i)_{i \in I}$ be the family of connected components
 of $S$.
Let $\delta$ and $\delta'$ be two dimension functions
 on $S$. Put $\underline n=\delta'-\delta$ seen as an element of $\ZZ^I$
 according to Lemma \ref{lm:dim+locally_ct}.

Then the functor
 $\Sigma_{\underline n}:(\T(S),t_\delta) \rightarrow (\T(S),t_{\delta'})$
 is an equivalence of $t$-categories.
\end{lm}

\begin{rem}
In other words, changing the dimension function on $S$
 only changes the $\delta$-homotopy $t$-structure by some shift.
 However, we will keep the terminology ``$\delta$-homotopy $t$-structure''
 because it allows us to underline the difference with
 the other homotopy $t$-structures
 defined earlier (Voevodsky, Morel, D\'eglise, Ayoub).
 Then the symbol $\delta$ stands for ``dimensional''.

Beware also that it is sometimes important to be precise about the chosen
 dimension functions (in particular, about the $t$-exactness properties of 
functors).
\end{rem}

\begin{prop}\label{prop:basic_tstruct}
Let $S$ be a scheme with a dimension function $\delta$
 and $f:T \rightarrow S$ be a morphism essentially of finite type.
\begin{enumerate}
\item If $f$ is separated, 
 the pair $f_!:(\T(T),t_{\delta})
 \leftrightarrows (\T(S),t_\delta):f^!$
 is an adjunction of $t$-categories.
\item If $\delta \geq 0$, then the tensor product $\otimes_S$
 is right $t$-exact on $(\T(S),t_\delta)$.
\item If $\dim(f)\leq d$, then the pair of functors
 $f^*[d]:(\T(S),t_\delta) \rightarrow (\T(T),t_{\delta}):f_*[-d]$
 is an adjunction of $t$-categories.
\end{enumerate}
\end{prop}
\begin{proof}
In each case, one has to check that the left adjoint functor
 sends a generator of the relevant homotopy $t$-structure
 to a non-negative object.
For (1), let $Y/T$ be a separated scheme:
 recall from \S\ref{num:induced_delta}
 that $\delta^f(Y)=\delta(Y)$; thus
 $f_!(\Mb(Y/T)\dtw{\delta^f(Y)}=\Mb(Y/S)\dtw{\delta(Y)}$
 and this concludes.
Similarly, assertion (2) (resp. (3)) follows
 from part (1) (resp. (2)) of Proposition \ref{prop:delta-base_change}
 and property (BM5) (resp. (BM6)) of \ref{num:BM}.
\end{proof}

\begin{cor}
Adopt the assumptions of the previous proposition.
\begin{enumerate}
\item If $f$ is \'etale,
 $f^*:(\T(S),t_\delta) \rightarrow (\T(T),t_{\delta})$ is $t$-exact.
\item If $f$ is finite,
 $f_*:(\T(T),t_{\delta}) \rightarrow (\T(S),t_\delta)$ is $t$-exact.
\item[(2')] If $f$ is proper and $\dim(f) \leq d$,
 then $f_*$ has (homological) amplitude $[0,d]$
 (with respect to the $\delta$-homotopy $t$-structures).
\end{enumerate}
\end{cor}

\begin{num}
Consider a closed immersion $i:Z \rightarrow S$
 with complementary open immersion $j:U \rightarrow S$.
 Recall the localization property for $\T$
 says precisely that $\T(S)$ is glued from $\T(U)$ and
 $\T(Z)$ with respect to the six functors
 $i^*, i_*, i^!, j_!, j^*, j_*$.
 In particular, we are in the situation of \cite[I, 1.4.3]{BBD}.

From \emph{op. cit.} 1.4.10, if $t_U$ (resp. $t_Z$) is a $t$-structure
 on $\T(U)$ (resp. $\T(Z)$), there exists a unique $t$-structure $t_{gl}$
 whose positive objects $K$ are characterized by the conditions:
$$
i^*(K) \geq 0, j^*(K) \geq 0.
$$
The $t$-structure $t_{gl}$ is called the \emph{$t$-structure glued from
 $t_U$ and $t_Z$}; one also says the $t$-category
 $(\T(S),t_{gl})$ is \emph{glued from $(\T(U),t_U)$
 and $(\T(Z),t_Z)$}.
 The glued $t$-structure on $\T(S)$ is uniquely characterized
 by the fact that $j^*$ and $i_*$ are $t$-exact
 (see \emph{op. cit.}, 1.4.12). In particular,
 the previous corollary immediately yields:
\end{num}
\begin{cor}\label{delta-htp_glued}
Consider the assumption of the previous proposition.
 Let $i:Z \rightarrow S$ be a closed immersion
 with complementary open immersion $j:U \rightarrow S$.

 Then the $t$-category $(\T(S),t_\delta)$ is obtained by gluing
  the $t$-categories $(\T(Z),t_\delta)$ and $(\T(U),t_\delta)$.
\end{cor}

\begin{ex}\label{ex:motive_sm_positive}
let $S$ be a universally catenary integral scheme
 with dimension function $\delta=-\codim_S$
 (cf. Example \ref{ex:can_dim_fn}).
Then we deduce from the previous corollary that for
 any smooth $S$-scheme $X$, the $\T$-spectrum $M_S(X)$
 is $t_\delta$-non-negative.

Indeed, first we can assume that $X$ is connected.
 Let $d$ be its relative dimension over $S$.
 According to the preceding corollary, 
 this assertion is Zariski local on $S$. 
 In particular, we can assume
 that the tangent bundle of $X/S$ admits a 
 trivialization.
 Then according to \ref{num:BM}(BM4),
 $M_S(X)=\Mb(X/S)\dtw d$. It is now sufficient to
 remark that $d=\delta(X)$ -- because of our choice of $\delta$
 along with Proposition \ref{prop:delta-base_change}(3).

Actually, this example and the previous gluing property
 are the main reason for our choice of generators of
 the $\delta$-homotopy $t$-structure
 (see also Corollary \ref{cor:htp_t_over_fields}).
\end{ex}

\begin{cor}\label{cor:Thom&delta-homotopy}
Let $(S,\delta)$ be a dimensional scheme
 and $E/S$ be a vector bundle of rank $r$.
 Then the functor $\big(\MTh_S(E)[-r] \otimes -\big)$
 is $t_\delta$-exact. The same result holds if
 one replaces $E$ by a virtual vector bundle $v$ over $S$
 of virtual rank $r$
 (see \S\ref{num:virtual_vb}).
\end{cor}
It simply follows by noetherian induction on
 $S$ from the previous corollary and
 the fact a vector bundle of rank $r$
 is generically isomorphic to $\AA^r_S$.

\begin{prop}
Let $(S,\delta)$ be a dimensional scheme
 and $f:T \rightarrow S$ a smooth morphism.

Then $f^!$ is $t_\delta$-exact.
\end{prop}
This follows from Proposition \ref{prop:basic_tstruct},
 the previous corollary and the relative
 purity isomorphism \eqref{eq:purity_iso}.\footnote{Under
  stronger assumptions,
  the $t$-exactness of $f^!$ will be generalized in
  Corollary \ref{cor:nice_ppties_dhtp}.}

\begin{num}\label{num:pullback_change_dim}
One can reinterpret the preceding proposition in terms of
 the classical pullback functor $f^*$.

Let $(S,\delta)$ be a dimensional scheme and $f:T \rightarrow S$
 be a smooth morphism. Let us write $\delta^f$ for the dimension
 function on $T$ induced by $f$ (see \ref{num:induced_delta}).
 Note that since $f$ is smooth, the function $\dim(f)$ is
 Zariski locally constant on $T$. We consider a new dimension
 function on $T$:
\begin{equation}\label{eq:twisted_relative_dim}
\tilde \delta^f=\delta^f-\dim(f).
\end{equation}
As a corollary of the previous proposition, we get:
\end{num}
\begin{cor}
In the notation above, the adjunction 
$$
f^*:(\T(S),t_\delta) \leftrightarrows (\T(T),t_{\tilde \delta^f}):f_*
$$
of triangulated categories is an adjunction of $t$-categories
 such that $f^*$ is $t$-exact.
\end{cor}
\begin{proof}
We can assume that $T$ is connected in which case $\dim(f)$ is constant
 equal to an integer $d$. By definition, one has 
 $\tilde \delta^f-\delta^f=-d$.
Indeed, according to the previous proposition and Lemma \ref{lm:indep_delta},
 the following composite functor is $t$-exact:
$$
(\T(S),t_\delta) \xrightarrow{f^!} (\T(T),t_{\delta^f})
 \xrightarrow{\Sigma_{-d}}  (\T(T),t_{\tilde \delta^f}),
$$
where $\Sigma_{-d}(K)=K[-d]$.
According to \ref{num:BM}(BM4), $f^!=f^* \otimes \Th(\tau_f)$,
 where $\tau_f$ is the tangent bundle of $f$, which has pure rank $d$.
 Thus, according to the beginning of the proof,
 we get that $f^* \otimes \Th(\tau_f)[-d]=f^![-d]$ is $t$-exact
 with respect to $t_\delta$ on the source and $t_{\tilde \delta^f}$
 on the target. This concludes the proof by Corollary 
\ref{cor:Thom&delta-homotopy}.
\end{proof}

\begin{ex}
Let $f:T \rightarrow S$ be a smooth morphism such
 that $S$ and $T$ are equicodimensional (for example, $S$ integral over a 
field). 
 Let $\delta=\dim_S$ (resp. $\dim_T$)
 be the Krull dimension function restricted to $S$ (resp. $T$).
 Then one readily checks that $(\dim_S)^f=\dim_T+\dim(f)$, so that we have 
$\tilde \delta^f=\dim_T$
  in the notation of \eqref{eq:twisted_relative_dim}. Thus the preceding 
corollary yields
  the following  adjunction of $t$-categories:
$$
f^*:(\T(S),t_{\dim_S}) \leftrightarrows (\T(T),t_{\dim_T}):f_*
$$
such that $f^*$ is $t$-exact.
\end{ex}
  
Let us remark finally the following fact.
\begin{lm}\label{lm:inseparable_texact}
Assume that the motivic category $\T$ is
semi-separated.\footnote{\label{foot:semisep} This notion
 was first introduced by Ayoub. Recall it means that
 for any finite surjective radicial morphism $f:T \rightarrow S$,
 the functor $f^*:\T(S) \rightarrow \T(T)$
 is conservative. According to \cite[2.1.9]{CD3}, it implies that $f^*$
 is in fact an equivalence of categories.}
 Then for any finite surjective radicial morphism $f:T \rightarrow S$,
 the functor $f^*:\T(S) \rightarrow \T(T)$ is an equivalence of $t$-categories.
\end{lm}
This directly follows from Proposition \ref{prop:basic_tstruct}
 and \cite[2.1.9]{CD3}.

\begin{ex}\label{ex:DM_semi_sep}
The triangulated motivic category $\T=\DM_R$ in the conventions
 of point (1) or (2) of \S\ref{num:convention_DM} is semi-separated.

More generally, it can be shown that any triangulated motivic
 category $\T$ over a category of schemes $\base$ is semi-separated
 provided that the following two conditions hold:
\begin{itemize}
\item $\T$ is $\ZZ[N^{-1}]$-linear where $N$ is the set
 of characteristic exponent of all the residue fields of schemes in $\base$;
\item $\T$ is oriented.
\end{itemize}
This follows from the existence of the Gysin morphism associated
 with any finite local complete intersection morphism
 and from a local trace formula. This is the case
 in particular for the triangulated motivic category
 $\T=\MGLmod$ of $\MGL$-modules of Example \ref{ex:motivic_cat}(2).
\end{ex}

\subsection{The $\delta$-effective category}\label{sec:delta_effective}

\begin{df}
Let $(S,\delta)$ be a dimensional scheme.

We define the category of \emph{$\delta$-effective $\T$-spectra},
 denoted by $\T^{\delta-\eff}(S)$,
 as the localizing triangulated subcategory
 of $\T(S)$ generated by 
 objects of the form $\Mb(X/S)(n)$ for any
 separated $S$-scheme $X$ and any integer $n \geq \delta(X)$.

When $f:T \rightarrow S$ is a morphism essentially of finite type,
 we will put: $\T^{\delta-\eff}(T):=\T^{\delta^f-\eff}(T)$.
\end{df}

\begin{rem}\label{rem:Teff_indep_delta}
\begin{enumerate}
\item Arguing as in the proof of Lemma \ref{lm:indep_delta},
 we easily deduce that, up to a canonical triangulated equivalence,
 the category $\T^{\delta-\eff}(S)$ does not depend on $\delta$.
 This equivalence will preserve the monoidal structure if it is defined
 (see point (2) of the following proposition).
\item Our definition of $\delta$-effectivity is closely related
 to the classes of morphisms $B_n$ (and to the corresponding constructions)
 considered in \cite[\S2]{Pel13};
 yet it seems our use of dimension functions is new in this context.
\end{enumerate}
\end{rem}

\begin{ex}\label{ex:classical_eff}
\begin{enumerate}
\item Assume $\T$ is oriented, the base scheme $S$ in universally
 catenary and integral,
 and $\delta=-\codim_S$ (Example \ref{ex:can_dim_fn}).
 Then for any smooth $S$-scheme $X$, the $\T$-spectrum
 $M_S(X)$ is $\delta$-effective, as in the classical mixed
 motivic case --- this follows from \ref{num:BM}(BM4)
 and Proposition \ref{prop:delta-base_change}(3).
 In fact, we do not need the orientation assumption
 and a more general statement will be proved in Corollary
 \ref{cor:sm_motives_eff}.
\item Actually,
 as a consequence of the cancellation theorem of Voevodsky,
 we will see in Example \ref{ex:delta_effective_fields}
 that when $\T=\DM[1/p]$, $S$ is the spectrum of a field $k$
 of characteristic exponent $p$,
 $\delta$ the Krull dimension function on $k$,
 then $\T^{\delta-\eff}(k)$
 coincides with the category of Voevodsky's motivic (unbounded) complexes.
\end{enumerate}
\end{ex}

The following proposition summarizes basic facts about $\delta$-effective
 spectra.
\begin{prop}\label{prop:functors&eff}
Let $(S,\delta)$ be a dimensional scheme.
\begin{enumerate}
\item The inclusion functors
 $s:\T^{\delta-\eff}(S) \rightarrow \T(S)$
 and $s':\T^{\delta-\eff}(S)\rightarrow\T^{\delta-\eff}(S)(-1)$
 (that we consider as a subcategory of $\T(S)$)
 admit right adjoints $w:\T(S) \rightarrow \T^{\delta-\eff}(S)$
 and $w':\T^{\delta-\eff}(S)(-1) \rightarrow \T^{\delta-\eff}(S)$, respectively.
\item If $\delta \geq 0$, then $\T^{\delta-\eff}(S)$
 is stable under tensor products.
\item Let $f:T \rightarrow S$ be a morphism
 essentially of finite and $d$ an integer such that $\dim(f) \leq d$.
 Then the functor $f^*(d)$ sends $\T^{\delta-\eff}(S)$
 to $\T^{\delta-\eff}(T)$.
\item Let $f:T \rightarrow S$ be any separated morphism.
 Then the exceptional direct image functor $f_!$ sends
 $\T^{\delta-\eff}(T)$ to $\T^{\delta-\eff}(S)$.
\end{enumerate}
\end{prop}
\begin{proof}
Assertion (1) follows from \Mikhail{ Neeman's adjunction theorem, since} we have
 assumed $\T(S)$ is compactly generated.
Assertion (2) follows from Prop. \ref{prop:delta-base_change}(1)
 and \textsection\ref{num:BM}(BM5).
Finally, assertion (3) follows from \ref{prop:delta-base_change}(2) and 
 \ref{num:BM}(BM6) whereas assertion (4) is obvious.
\end{proof}

\begin{rem}
The couple of functors $(s,w)$ is analog to the couple of functors
 (infinite suspension, infinite loop space) of stable homotopy,
 though in our case, $s$ is fully faithful.
 In the case where $\T$ is the category of motives $\DM_R$,
 we refer the reader to Example \ref{ex:delta_effective_fields}
 for more precisions.
\end{rem}

\begin{cor}\label{cor:functoriality_delta-eff}
Consider the notations of the previous proposition.
\begin{enumerate}
\item If $\delta \geq 0$, $\T^{\delta-\eff}(S)$ has
 internal Hom: given any $\delta$-effective spectra
 $M$ and $N$ over $S$, one has:
$$
\uHom_{\T^{\delta-\eff}(S)}(M,N)
=w\uHom_{\T(S)}(M,N).
$$
\item Let $f:T \rightarrow S$ be a separated morphism.
 Then one has an adjunction of triangulated categories:
$$
f_!:\T^{\delta-\eff}(T) \leftrightarrows \T^{\delta-\eff}(S):
f^!_{\delta-\eff}:=w \circ f^!.
$$
\item Let $f:T \rightarrow S$ be a morphism
 essentially of finite type such that $\dim(f) \leq d$.
 Then one has an adjunction of triangulated categories:
$$
f^*(d):\T^{\delta-\eff}(T) \leftrightarrows
\T^{\delta-\eff}(S):w \circ \lbrack f_*(-d) \rbrack.
$$
\end{enumerate}
\end{cor}

\begin{rem}
When $f:T \rightarrow S$ is equidimensional
 (for example smooth, flat or universally open),
 one can improve point (3) using the same trick as that
 of \textsection \ref{num:pullback_change_dim}.
 If $\delta$ is a dimension function on $S$,
 we consider the following dimension function on $T$:
 $\tilde \delta^f=\delta^f-\dim(f)$.

Then, according to point (3) of the above corollary,
 we obtain a well defined adjunction of triangulated categories:
\begin{equation}\label{eq:pullback_eff_modified}
f^*:\T^{\delta-\eff}(S) \rightarrow
 \T^{\tilde \delta^f-\eff}(T):w \circ f_*.
\end{equation}
Note finally that if $S$ is equicodimensional
 and $\delta=\dim_S$ is the Krull dimension function on $S$,
 then $\tilde \delta^f=\dim_T$, the Krull dimension function
 on $T$ --- here we use the fact $f$ is equidimensional.
\end{rem}

\begin{ex}
Let $(S,\delta)$ be an arbitrary dimensional scheme.
According to point (1) of the preceding proposition
 we obtain  an adjunction of triangulated categories
$$
s:\T^{\delta-\eff}(S) \leftrightarrows \T(S):w
$$
such that $s$ is fully faithful. This formally
 implies that $\T^{\delta-\eff}(S)$ is the Verdier quotient
 of the category $\T(S)$ made by the full triangulated
 subcategory $\Ker(w)$ made of $\T$-spectra $K$ over $S$
 such that $w(K)=0$.
Moreover, according to our conventions, 
 the triangulated category $\T^{\delta-\eff}(S)$
 is compactly generated. This formally implies the right adjoint
 functor $w$ commutes with coproducts. So the category
 $\Ker(w)$ is stable by coproducts (\emph{i.e.} localizing).

It is difficult in general to describe concretely
 the category $\Ker(w)$. We now give examples in
 the particular case $\T=\DM_R$,
 following the conventions of  point (1) or (2) in
 Paragraph \ref{num:convention_DM}. We need some facts to justify
 them so we postpone this justification till
 Paragraph \ref{num:eq:compute_w}.

Assume $S$ is regular and $X/S$ is a smooth projective
 scheme such that $\dim(X/S)=d$ for a fixed integer $d$.
Then:
\begin{equation}\label{eq:compute_w}
w\big(M_S(X)(n)\big)=\begin{cases}
M_S(X)(n) & \text{if } n \geq \delta(X)-d, \\
0 & \text{if } n<\delta(S)-d.
\end{cases} 
\end{equation}
This computation allows us to compute the right adjoint appearing
 in point (3) in some particular cases.
 Let us give a very basic example. 
To fix ideas, assume $S$ is irreducible and $\delta(S)=0$.
 Of course, this implies $\un_S$ is $\delta$-effective.
Let $P$ be a projective bundle over $S$ of rank $d$,
 and $f:P \rightarrow S$ be the canonical projection.
 Then $\un_P(d)$ is $\delta$-effective and we get:
$$
w\big(f_*\big(\un_P(d)\big)(-d)\big)
 =\bigoplus_{i=0}^d w\big(\un_S(-i)[-2i]\big)
 =\un_S
$$
according to the projective bundle formula and the preceding
 computation. Many similar computations can be obtain from
 the previous formula. We let them as an exercise for the reader.
\end{ex}

\begin{num}
Let $\basex S^{\qf}$ be the sub-category of $\basex S$
 made of the same objects but whose morphisms $f$
 are quasi-finite (see also Example \ref{ex:quasi-finite}).

Then the previous proposition implies that 
 $\T^{\delta-\eff}$ is a $\sm$-fibred triangulated subcategory
 of $\T$ over $\basex S^{\qf}$
 (cf. \cite[\textsection 1]{CD3}).\footnote{In fact,
 point (3) of Corollary \ref{cor:functoriality_delta-eff}
 shows that given a dimensional scheme $S$,
 $\T^{\delta-\eff}$ is fibred over $S$-schemes essentially of finite
 type with respect to morphisms which are equidimensional
 (recall footnote \ref{foot:equidim}, p. \pageref{foot:equidim}).}
 Moreover, if $\delta \geq 0$, it is even a monoidal
 $\sm$-fibred triangulated sub-category of $\T$:
 for any quasi-finite $S$-scheme $T$, the tensor product
 $\otimes_T$ respects $\delta$-effective spectra
 and moreover the unit $\un_T$ is a $\delta$-effective spectra.
 In particular, the inclusion 
 $s:\T^{\delta-\eff}(T) \rightarrow \T(T)$
 is monoidal, for the induced monoidal structure on the left
 hand side.

Finally, it is easily seen using Proposition \ref{prop:functors&eff}
 that this premotivic category satisfies the localization property:
\end{num}
\begin{prop}\label{prop:localization_effective}
Consider the assumptions of the preceding proposition.
Then for any closed immersion $i:Z \rightarrow S$ with
 complementary open immersion $j$, and any $\delta$-effective
 spectrum $M$, the following triangle is a distinguished triangle
 in $\T^{\delta-\eff}(S)$:
$$
j_!j^*(K) \rightarrow K \rightarrow i_!i^*(K)
 \xrightarrow{+1}
$$
\end{prop}

\begin{cor}\label{cor:sm_motives_eff}
Let $(S,\delta)$ be any dimensional scheme.

Then for any smooth $S$-scheme $X$,
 the $\T$-spectrum $M_S(X)\big(\delta(S)\big)$
 is $\delta$-effective.
\end{cor}
\begin{proof}
By additvity of $M_S(X)$ in $X$, we can assume $X$ is connected.
 Using the preceding proposition, an easy Noetherian induction shows
 the assertion is Zariski local in $S$. In particular, we can assume the 
 tangent bundle of $X/S$ is trivial, say isomorphic to $\AA^d_X$.
 Then $f:X \rightarrow S$ has constant relative dimension, say $d$
 and we can apply Proposition~\ref{prop:delta-base_change}(3)
 to $f$. Thus one gets $\delta(X)=\delta(S)+d$.
 Then using \ref{num:BM}(BM4), the following spectrum is in 
$\T^{\delta-\eff}(S)$:
$$
\Mb(X/S)\dtw{\delta(X)}=M_S(X)\dtw{\delta(X)-d}=M_S(X)\dtw{\delta(S)},
$$
and this concludes. 
\end{proof}

\begin{ex}\label{ex:motive_sm_eff}
The preceding proposition applies in particular
 when $S$ is a universally catenary integral scheme
 with dimension function $\delta=-\codim_S$
 (cf. Example \ref{ex:can_dim_fn}). In this case,
  for any smooth $S$-scheme $X$,
 the $\T$-spectrum $M_S(X)$ is $\delta$-effective
  --- as expected when $\T=\SH$ or $\DAx R$ and $S$
 is the spectrum of a field.
\end{ex}

Using the same proof as the one of Corollary \ref{cor:Thom&delta-homotopy},
 we also deduce from the preceding proposition:
\begin{cor}\label{cor:delta_eff_thom_stable}
Let $(S,\delta)$ be a dimensional scheme
 and $E/S$ be a vector bundle of rank $r$.
 Then the functor $(\MTh(E)[-r] \otimes -)$
 preserves $\delta$-effective spectra.
\end{cor}

\begin{rem}\label{rem:functors&eff}
Here is a complete list of functors 
 that preserves $\delta$-effective spectra:
 $f_!$ for $f$ separated,
 $f_*$ for $f$ is proper,
 $f^*$ for $f$ quasi-finite, 
 $f^!$ for $f$ smooth.

All cases follow from \ref{prop:functors&eff}
 except the last one which also uses the previous corollary
 and the purity isomorphism  $f^!\simeq \MTh(T_f) \otimes f^*$.
\end{rem}

We can easily extend the definition of the $\delta$-homotopy
 $t$-structure to the effective case:
\begin{df}\label{df:eff_tstruct}
Let $(S,\delta)$ be a dimensional scheme.

The (effective) \emph{$\delta$-homotopy $t$-structure over $S$},
 denoted by $t_\delta^\eff$, or $t_\delta$ when no confusion can arise,
 is the $t$-structure
 on $\T^{\eff}(S)$ generated by spectra of the form
 $\Mb(X/S)\dtw{\delta(X)}\gtw n$ for any separated $S$-scheme $X$
 and any integer $n \geq 0$.

When $f:T \rightarrow S$ is essentially of finite type,
 we put $t_\delta=t_{\delta^f}$.
\end{df}
With these definitions, it is clear that the adjunction
 of triangulated categories (see \ref{prop:functors&eff}(1)):
\begin{equation}\label{eq:adj_t_delta_eff_st}
s:\T^{\delta-\eff}(S) \leftrightarrows \T(S):w
\end{equation}
is in fact an adjunction of $t$-categories.

\begin{rem}\label{rem:htp_t_nonnegdeg_eff}
 As in the non-effective case, we get:
\begin{enumerate}
\item an effective spectrum $\E$ over $S$ is $t_\delta$-negative
 if and only if for any separated $S$-scheme $X$ the $\ZZ$-graded
 abelian group $\E^{BM}_{p,\gtw *}(X/S)$ is zero
 in degree $* \geq \delta(X)$ and if $p \geq \delta(X)$;
\item (see Remark \ref{rem:htp_t_nonnegdeg}),
 the effective $\delta$-homotopy $t$-structure is left non-degenerate;
\item it does not depend on the choice of the dimension function
 $\delta$, up to a canonical equivalence of $t$-categories;
\item for any vector bundle $E/S$ of rank $r$,
 the endo-functor $\big(\MTh_S(E)[-r] \otimes -\big)$
 of $\T^{\delta-\eff}(S)$ (see Cor. \ref{cor:delta_eff_thom_stable})
 is $t_\delta$-exact.
\end{enumerate}
\end{rem}

\begin{num}\label{num:adj_tdelta_eff}
We can easily transport the results of Proposition \ref{prop:basic_tstruct}
 to the effective case as follows:
\begin{enumerate}
\item For any separated morphism $f:T \rightarrow S$, the pair
 of functors
$$
f_!:\Big(\T^{\delta-\eff}(T),t_{\delta}\Big)
 \rightarrow \Big(\T^{\delta-\eff}(S),t_\delta\Big):w \circ f^!
$$
is an adjunction of $t$-categories.
\item If $\delta \geq 0$, the tensor product $\otimes_S$
 is right $t$-exact on $\big(\T^{\delta-\eff}(S),t_\delta\big)$
\item For any morphism $f:T \rightarrow S$ such that
 $\dim(f)\leq d$, the pair of functors:
$$
f^*\dtw d:\Big(\T^{\delta-\eff}(S),t_\delta\Big)
 \rightarrow \Big(\T^{\delta-\eff}(S),t_{\delta}\Big):w \circ (f_*\dtw{-d})
$$
is an adjunction of $t$-categories.
\end{enumerate}
In particular, using the
 adjunction of Corollary \ref{cor:functoriality_delta-eff}
 together with Remark \ref{rem:functors&eff},
 we get that $f_!$ (resp. $f^*$) is $t$-exact
 when $f$ is finite (resp. \'etale).
 So we get the following corollary
 (as for Corollary \ref{delta-htp_glued}):
\end{num}

\begin{cor}\label{delta-htp_eff_glued}
Let $(S,\delta)$ be a dimensional scheme,
 $i:Z \rightarrow S$ be a closed immersion
 with complementary open immersion $j:U \rightarrow S$.
 Then the $t$-category $(\T^{\delta-\eff}(S),t_\delta)$
 is obtained by gluing
 of the $t$-categories $(\T^{\delta-\eff}(Z),t_\delta)$
 and $(\T^{\delta-\eff}(U),t_\delta)$.
\end{cor}

As in the non-effective case, we deduce that when $f$ is smooth,
 the functor $w \circ f^!=f^!$ is $t$-exact.
 Moreover, the functor $f^*$, with the conventions of
 \eqref{eq:pullback_eff_modified} for the dimension functions,
 is $t$-exact.
 Finally, 
 As in the non-effective case
 it is worth to remark the following easy lemma (use the proof of
 the analog Lemma \ref{lm:inseparable_texact}).
\begin{lm}\label{lm:inseparable_texact_eff}
Assume that the motivic category $\T$ is semi-separated.

Then for any finite surjective radicial morphism $f:T \rightarrow S$,
 the functor $f^*:\T^{\delta-\eff}(S) \rightarrow \T^{\delta-\eff}(T)$
 is an equivalence of $t$-categories.
\end{lm}

\subsection{Improved descriptions of generators}\label{s2.3}

In this section, we give several different descriptions of the generators
 of the $\delta$-homotopy $t$-structure and draw some
 corollaries for the $\delta$-homotopy $t$-structure.
\begin{prop}\label{prop:generators1}
Let $(S,\delta)$ be a dimensional scheme.

Then the $\delta$-homotopy $t$-structure on $\T(S)$
 (resp. $\T^{\delta-\eff}(S)$) admits the following
 three families of generators for positive objects:
\begin{enumerate}
\item spectra of the form
 $\Mb(X/S)\dtw{\delta(X)}\gtw n$
 for an integer $n \in \ZZ$ (resp. $n \in \NN$)
 and a proper $S$-scheme $X$;
\item spectra of the form
 $\Mb(X/S)\dtw{\delta(X)}\gtw n$
 for an integer $n \in \ZZ$ (resp. $n \in \NN$)
 and a regular $S$-scheme $X$;
\item assuming $S$ is regular
 and separated not necessarily of finite type over $\ZZ$:
 spectra of the form
 $\Mb(X/S)\dtw{\delta(X)}\gtw n$
 for an integer $n \in \ZZ$ (resp. $n \in \NN$)
 and a regular $S$-scheme $X$ such that there
 exists a closed $S$-immersion
 $X \rightarrow \AA^r_S$ with trivial normal bundle.
\end{enumerate}
\end{prop}
\begin{proof}
To treat all cases simultaneously,
 we denote by $I$ the set $\ZZ$ (resp. $\NN$)
 and put, using the notation of \ref{num:generated_aisle},
 $\C=\langle \mathcal G\rangle_+$
 where $\mathcal G$ is the family describe in point (1), (2) or (3).
 We have to prove that for any separated morphism
 $f:X \rightarrow S$, $\Mb(X/S)\dtw{d}$ belongs to
 $\C$ where $d=\delta(X)$.

\noindent \underline{Point (1)}: because $f$ is separated
 (of finite type), it admits a factorization
 $X \xrightarrow j \bar X \xrightarrow p S$
 such that $p$ is proper and $j$ is a dense open immersion.
 Let $i:Z \rightarrow \bar X$ be the complementary closed
 immersion of $j$, with $Z$ reduced.
 Then using \ref{num:BM}(BM3), we get a distinguished triangle:
$$
\Mb(Z/S)\dtw d[-1] \rightarrow \Mb(X/S)\dtw d
 \rightarrow \Mb(\bar X/S)\dtw d \xrightarrow{+1}
$$
Note that, since $j$ is dense, we have: $\delta(\bar X)=d$
 and $\delta(Z)<d$.
 In particular, by assumption on $\C$,
 the $\T$-spectra $\Mb(\bar X/S)\dtw d$ and $\delta(\bar X)=d$
 both belongs to $\C$.
 This implies $\Mb(X/S)\dtw d$ belongs to $\C$ and concludes. 

\noindent \underline{Point (2) and (3)}: We use noetherian induction
 on $X$, given that the result is obvious when $X$ is empty.
 Then it is sufficient to find a dense open subscheme $U \subset X$
 such that $\Mb(U/S)\dtw d$ belongs to $\C$. 
 Indeed, given that subscheme $U$
 we denote by $Z$ its complement in $X$ with its reduced structure
 of a subscheme of $X$. Then,
 applying again \ref{num:BM}(BM3),
 we get the following distinguished triangle:
$$
\Mb(U/S)\dtw d \rightarrow \Mb(X/S)\dtw d
 \rightarrow \Mb(Z/S)\dtw d \xrightarrow{+1}
$$
and the conclusion follows from the fact $\Mb(Z/S)\dtw d$
 belongs to $\C$ by noetherian induction using again the inequality
 $\delta(Z)<d$.

In any case, we can assume that $X$ is reduced.
Moreover, since $X$ is excellent, it admits a dense open subscheme $U$
 which is regular. This concludes in case of point (2). \\
For point (3) we remark that we can even assume $U$ is affine in addition
 to be a regular dense open subscheme of $X$.
 Then the structural morphism $p:U \rightarrow S$ is affine
 (because $S$ is separated over $\ZZ$, see \cite[1.6.3]{EGA2}).
 In addition, reducing $U$ again,
 we can assume the coherent $\mathcal O_S$-algebra $p_*(\mathcal O_U)$
 is generated by global sections. Thus there exists
 a closed immersion $i:U \rightarrow \AA^r_S$.
 This immersion is regular because $U$ is regular by construction
 and $\AA^r_S$ is regular by assumption on $S$.
 Thus the normal cone of $i$ is a vector bundle. In particular,
 it is generically trivial so that there exists an open subscheme
 $V$ in $\AA^r_S$ such that $i^{-1}(V)$ is dense in $U$ and trivialize
 the normal bundle of $i$. Thus, considering
 the dense open subscheme $i^{-1}(V) \subset X$,
 we are able to finish the proof of  assertion (3).
\end{proof}

As a corollary of assertion (1), we get the following statement.

\begin{cor}\label{cor:premotivic_adj_t-exactness}
Let $(S,\delta)$ be a dimensional scheme.

Then for any adjunction $\varphi^*:\T \leftrightarrows \T':\varphi_*$
 of triangulated motivic categories
 (cf. \cite[1.4.6]{CD3}) and any scheme $S$,
 the functor 
$\varphi^*:\T(S) \rightarrow \T'(S)$ is right $t_\delta$-exact
 and the functor
$\varphi_*:\T'(S) \rightarrow \T(S)$ is $t_\delta$-exact.
\end{cor}
\begin{proof}
The first assertion follows easily
 from the fact $\varphi^*$ commutes with functors $f_!$ 
 (cf. \cite[2.4.53]{CD3}).
 The second assertion comes from the previous
 proposition because $\varphi_*$ commutes with
 $f_*=f_!$ for $f$ proper.
\end{proof}

\begin{ex}\label{ex:adjunctions_heart}
Let $R$ be a coefficient ring, $(S,\delta)$ be a dimensional scheme
 and consider the premotivic adjunctions \eqref{eq:premotivic_adj}.
 According to the previous corollary, all left (resp. right) adjoints
 of these premotivic adjunctions are right $t_\delta$-exact
 (resp. $t_\delta$-exact).
 Then according to \eqref{eq:proto_t_adj},
  we get adjunctions of abelian categories between the homotopy hearts:
$$
\xymatrix@R=4pt@C=40pt{
 && \hrt{\MGLmod(S)}\restore\ar@<-4pt>_/-8pt/{\phi_*}[lld] \\
\hrt{\SH(S)}\ar@<0pt>_/8pt/{H_0\phi^*}[rru]\ar@<-2pt>_/6pt/{H_0\delta^*}[rd]
 & &  \\
 & \hrt{\DA(S,R)}\ar@<-2pt>_{H_0\gamma^*}[r]\ar@<-2pt>_/-12pt/{\delta_*}[lu]
 & \hrt{\DM(S,R)}.\ar@<-2pt>_{\gamma_*}[l] 
}
$$
Moreover, the right adjoints $\delta_*$, $\gamma_*$, $\rho_*$, and $\phi_*$,
 are all exact and conservative, thus faithful.
\end{ex}

As a corollary of point (3) of the previous proposition,
 we get the following fact:
\begin{cor}\label{cor:htp_t_over_fields}
Let $k$ be a perfect field.

Then the $\delta_k$-homotopy t-structure on $\T(k)$
 is generated by spectra of the form $M(X)\gtw n$,
 where $X/k$ is smooth and $n \in \ZZ$.
\end{cor}
Indeed, given $X/k$ as in point (3) of the previous proposition,
 we get: 
\begin{equation}\label{eq:compute_Mb4sm}
\Mb(X/k)\dtw{\delta_k(X)}
 \stackrel{(BM4)}\simeq \MTh_k(-T_{X/k})\dtw{\delta_k(X)}
 \simeq M(X)
\end{equation}
because, under the assumptions on $X$,
 the tangent bundle of $X/k$ is trivial of rank
 $\delta_k(X)$.

\begin{ex}\label{ex:delta_htp-t-struct_field}
 Let $k$ be a perfect field of characteristic exponent $p$,
 $\delta_k$ the canonical dimension function.
\begin{enumerate}
\item Let $\T=\DM_R$ be the triangulated motivic category of
 point (1) or (2) of Paragraph \ref{num:convention_DM}.
 In case (2), we assume $k$ contains $F$.
 Recall from \eqref{eq:DM&DM_cdh} that
 $\DM(k,R)$ can be described using Nisnevich topology instead
 of $\cdh$-topology, as was done in \cite[7.15]{CD1}.

Then according to the previous corollary and \cite[\textsection 5.7]{Deg9},
 the $\delta_k$-homotopy $t$-structure on $\DM(k,R)$ coincides
 with the \emph{homotopy $t$-structure} defined in \cite[Prop. 5.6]{Deg9}.
\item The case of the stable homotopy category $\T=\SH$.

Recall Morel has introduced in \cite[Th. 5.3.3]{Mor1}
 a $t$-structure on $\SH(k)$ that we will call the
 \emph{Morel homotopy $t$-structure}.
 One notable feature of this $t$-structure is that its heart
 is equivalent to the category of \emph{homotopy modules over $k$}
 (\cite[1.2.2]{Deg10}), or recall below).
 Moreover, as remarked in \cite[1.1.5]{Deg10},
 this $t$-structure is generated (in the sense of Definition
 \ref{df:gen_tstruct}) by spectra of the form:
 $\Sigma^\infty X_+ \wedge \GG^{\wedge n}$, where $X$ is a smooth $k$-scheme,
 $\Sigma^\infty$ is the infinite suspension $\PP^1$-spectrum functor,
 $\GG$ stands for the sheaf of sets represented
 by the scheme $\GG$ pointed by $1$, and $n$ is any integer.

Thus, it follows from the previous corollary
 that the $\delta_k$-homotopy $t$-structure on $\SH(k)$
 coincides with Morel's homotopy $t$-structure as defined in \cite{Mor1}.
\item The case of the stable $\AA^1$-derived category $\T=\DAx R$,
 $R$ any ring of coefficients.

It is explain in \cite[Remark 8 of Introduction]{dmtilde3}
 how one can build the analog of the Morel's homotopy $t$-structure
 recalled in the preceding point on the effective version of
 the category $\DA(k,R)$. 
 In fact, following the construction of the homotopy $t$-structure on $\SH(k)$,
 on can define Morel's homotopy $t$-structure on the stable category
 $\DA(k,R)$. Then it follows
 from the construction that the heart is again equivalent to
 the category of homotopy modules over $k$.
 Again, this $t$-structure is generated by $\DAx R$-spectra of
 the form $M_k(X)$, where $X$ is a smooth $k$-scheme.

 Thus, we also obtain that the $\delta_k$-homotopy $t$-structure
 is equivalent to Morel's homotopy $t$-structure on $\DA(k,R)$.
\end{enumerate}
\end{ex}

\begin{num}\label{num:concrete_htp_heart}
Let us underlined an interesting consequence of the previous
 example. Let again $k$ be a perfect field of characteristic
 exponent $p$, invertible in a fixed coefficient ring $R$.

Recall from \cite[5.2.4]{Mor1} and \cite[1.1.4]{Deg10}
 that a \emph{homotopy module}
 (resp. \emph{homotopy module with transfers}) over $k$
 is a sequence $(F_n,\epsilon_n)_{n \in \NN}$ where $F_n$
 is a Nisnevich sheaf over the category of smooth $k$-schemes
 whose cohomology is $\AA^1$-invariant
 (resp. which is $\AA^1$-invariant and admits transfers)
 and $\epsilon_n:F_n \rightarrow (F_{n+1})_{-1}$ is
 an isomorphism of sheaves where for any such a sheaf $G$,
 $G_{-1}(X)=G(\GG \times X)/G(X)$, induced by the unit section
 of $\GG$. We denote by $\hmod(k)$ (resp. $\hmtr(k)$)
 the corresponding category, morphisms being natural transformations
 compatible with the grading and with the given isomorphisms $\epsilon_*$.
 Recall it is a Grothendieck abelian closed monoidal category.

Then we get the following equivalences of abelian monoidal categories
 where the left hand sides are the heart of the relevant homotopy
 (or $\delta_k$-homotopy) $t$-structure:
\begin{equation}\label{eq:compute_hrt}
\hrt{\SH(k)}\simeq \hmod(k), \quad \hrt{\DA(k)}\simeq \hmod(k),
 \quad \hrt{\DM(k)}\simeq\hmtr(k).
\end{equation}
The first two equivalences follow from the construction of Morel
 and the third one by \cite[5.11]{Deg9}.

Recall also that Morel defines the \emph{Hopf map of a scheme $S$}
 as the morphism $\eta:\GG \rightarrow S^0$ in $\SH(S)$,
 --- or $\eta:\un_S\gtw 1 \rightarrow \un_S$ with the conventions
 of the present paper --- induced by the morphism of schemes
$$
\big(\AA^2_S-\{0\}\big) \rightarrow \PP^1_S, (x,y) \mapsto [x,y].
$$
We say that a spectrum $\E$ over $S$
 has \emph{trivial action of $\eta$}
 if the map $\eta \wedge 1:\E\gtw 1 \rightarrow \E$ is $0$.
 This is equivalent to the fact that $\eta$ has trivial action
 on the cohomology $\E^{n,m}(X)$ for any smooth $S$-scheme $X$
 and any couple of integers $(n,m) \in \ZZ^2$. The main
 theorem of \cite{Deg10} identifies $\hmtr(k)$
 with the full subcategory of $\hmod(k)$ made by
 homotopy modules with trivial action of $\eta$.

In particular, we get the following proposition.
\end{num}
\begin{prop}\label{prop:compute_heart}
We use the notations of Example \ref{ex:adjunctions_heart}.
 Let $k$ be a perfect field of characteristic exponent $p$
 and $R$ a ring such that $p \in R^\times$.
 Then the following assertions hold:
\begin{enumerate}
\item The adjunction of abelian categories
 between the $\delta_k$-homotopy hearts:
$$
H_0\delta^*:\hrt{\SH(k)} \leftrightarrows \hrt{\DA(k)}:\delta_*
$$
are mutually inverse equivalences of abelian monoidal categories.
\item The exact functor of abelian categories
 between the $\delta_k$-homotopy hearts:
$$
\gamma_*:\hrt{DM(k,R)} \rightarrow \hrt{\DA(k,R)}
$$
is fully faithful and its essential image is equivalent
 to the category of homotopy modules with trivial
 action of the Hopf map $\eta$.
\end{enumerate}
\end{prop}

\begin{rem}
Let us consider the notations of the preceding proposition.
 It follows from the construction of the 
 isomorphisms \eqref{eq:compute_hrt} that one gets commutative
 diagrams:
$$
\xymatrix@R=18pt@C=26pt{
\hrt{DM(k,R)}\ar^-{\gamma_*}[r]\ar@{}|{(1)}[rd]
& \hrt{\DA(k,R)}\ar^-{\delta_*}[r]\ar@{}|{(2)}[rd] & \hrt{\SH(k)} \\
\hmtr(k)\ar_{\gamma'_*}[r]\ar^\sim[u] & \hmod(k)\ar@{=}[r]\ar^\sim[u] & 
\hmod(k)\ar_\sim[u]
}
$$
where $\gamma'_*$ associates to a homotopy modules with transfers 
$(F_*,\epsilon_*)$
 the homotopy module 
 $$\left(\gamma_*(F_*),\gamma_*(\epsilon_*)\right)$$
 (see \cite[1.3.3]{Deg10}).
 As $H_0\delta^*$ is a quasi-inverse of the functor $\delta_*$ appearing in
 the commutative square (2), it induces the identity functor on $\hmod(k)$
 through the identifications \eqref{eq:compute_hrt}.
 There is no easy way to describe the functor
 $\hmod(k) \rightarrow \hmtr(k)$ induced by $H_0\gamma^*$, as the
 functor adding transfers does not preserve the property of being homotopy
 invariant. 
\end{rem}

\begin{num}
Our $t$-structure is very analogous to that
 defined earlier by Ayoub in \cite[\textsection 2.2.4, p. 365]{ayoub1},
 under the name \emph{perverse homotopy $t$-structure}.

To recall the definition, one needs to fix a base scheme $B$
 and assume the following property (see \cite[Hyp. 2.2.58]{ayoub1}):
\begin{enumerate}
\item[(a)] For any separated morphism $f:X \rightarrow B$,
 the functor $f^!$ preserves constructible $\T$-spectra.\footnote{Recall
 that under our assumptions the two properties constructible
 and compact coincide for $\T$-spectra.}
\end{enumerate}
Note this property is automatically fulfilled whenever
 $\base$ is the category of $\QQ$-schemes (cf. \cite[2.2.33]{ayoub1})
 or $\T$ is $\QQ$-linear, separated 
 and satisfies the absolute purity property
 (cf. \cite[Th. 4.2.29]{CD3}).

Let $S$ be a separated\footnote{In \emph{loc. cit.}
 one assumes further that $S$ is a quasi-projective $B$-scheme but
 this restriction
 only appears because the functors $f_!$ and $f^!$ are defined
 under the assumption $f$ is quasi-projective.
 This unnecessary assumption has been removed in \cite{CD3}.}
 $B$-scheme.
Then the \emph{perverse homotopy $t$-structure} (relative to $B$)
 on $\T(S)$ is the $t$-structure generated by spectra of the
 form $g_!q^!(\un_B)\gtw n$ for any separated morphism 
 $g:X \rightarrow S$, $q$ being the projection of $X/B$.
 Following Ayoub, we will denote it by $^pt$.
\end{num}
\begin{prop}
Consider the above notations along with assumption (a).
 Assume that $B$ is regular and let $\delta$ be a dimension function on $B$.

Then for any separated $B$-scheme $S$,
 the $t$-structure $^pt$ (resp. $t^\delta$) on $\T(S)$
 is generated by objects of the form $g_!q^!(\un_B)\gtw n$
 (resp. $\Mb(X/S)\dtw{\delta(X)}\gtw n$) for an integer $n \in \ZZ$,
 an affine connected regular $B$-scheme $X$ which admits a closed $B$-embedding
 $i:X \rightarrow \AA^r_B$ with trivial normal bundle,
 and $g:X \rightarrow S$ a separated $B$-morphism.
\end{prop}
\begin{proof}
The case of $t^\delta$ was already proved as point (3) of
 Proposition \ref{prop:generators1}.
 The proof in the case of $^pt$ is completely similar.
\end{proof}

\begin{cor}\label{cor:compare_Ayoub}
Consider the assumptions of the previous proposition
 together with the following ones:
\begin{enumerate}
\item[(b)] $S$ is regular with dimension
 function $\delta=-\codim_S$ (Example \ref{ex:can_dim_fn});
\item[(c)] $\T$ satisfies the absolute purity property 
 (Definition \ref{df:abs_purity}).
\end{enumerate}
Then for all separated $S$-scheme $X$,
 the perverse homotopy $t$-structure $^pt$ and the
 $\delta$-homotopy $t$-structure $t^\delta$ coincide
 on $\T(X)$.
\end{cor}
\begin{proof}
We will prove that the generators of the two $t$-structures
 are the same up to isomorphism.
 Using the preceding proposition,
 the generators of $^pt$ (resp. $t^\delta$) are of the form
 $g_!q^!(\un_B)\gtw n$ (resp. $\Mb(X/S)\dtw{\delta(X)}\gtw n$)
 for an integer $n \in \ZZ$, a $B$-scheme $X$ 
 and a $B$-morphism $g$ satisfying the assumptions of the preceding
 proposition.
 Then, by assumption, the virtual tangent bundle of $q:X \rightarrow B$
 is isomorphic to the trivial vector $B$-bundle of rank $d$,
 where $d$ is the relative dimension of $q$.
 Thus, according to the absolute purity property on $\T$
 and Corollary \ref{cor:fund_class}, one gets an isomorphism:
$$
q^!(\un_B) \simeq \un_S\dtw d
$$
and this concludes because $d=\delta(X)$ according to
 Proposition \ref{prop:delta-base_change}(3).
\end{proof}

We close this section with corollaries
 of Proposition \ref{prop:generators1} concerning $\delta$-effective
 spectra.
\begin{cor}
Let $k$ be a perfect field and $\delta=\dim$ be the Krull dimension
 function on $\spec(k)$.
\begin{enumerate}
\item The category $\T^{\delta-\eff}(k)$ is the 
 localizing triangulated subcategory of $\T(k)$
 generated by spectra of the form $M(X)$
 for a smooth $k$-scheme $X$.
\item The $\delta$-homotopy $t$-structure on $\T^{\delta-\eff}(k)$
 is generated by spectra of the form $M(X)$ for
 a smooth $k$-scheme $X$.
\end{enumerate}
\end{cor}
\begin{proof}
According to point (3) of Proposition \ref{prop:generators1},
 the homotopy $t$-structure on $\T^{\delta-\eff}(k)$ is generated, 
 for a smooth connected $k$-scheme $X$
 by spectra of the form $\Mb(X/k)\dtw{\delta(X)}\gtw n$
 for a smooth $k$-scheme $X$ (because $k$ is perfect)
 and an integer $n \geq 0$. According to Computation
 \eqref{eq:compute_Mb4sm}, this spectrum is isomorphic to $M(X)\gtw n$.
 On the other hand, the later one is a direct factor of $M(X \times \GG^n)$;
 so point (2) of the above statement follows.

By definition and according to Lemma \ref{lm:non-neg-degenerated},
 the homotopy $t$-structure on $\T^{\delta-\eff}(k)$
 is left non-degenerate. Applying again
 Lemma \ref{lm:non-neg-degenerated}, we obtain that the
 triangulated category $\T^{\delta-\eff}(k)$ is equal to its
 smallest triangulated subcategory stable by coproducts and
 containing objects of the form $M(X)$ for a smooth $k$-scheme $X$.
 This is precisely point (1).
\end{proof}

\begin{ex}\label{ex:delta_effective_fields}
\begin{enumerate}
\item Let $F$ be a prime field with characteristic exponent $p$
 and $R$ be a ring such that $p \in R^\times$.
 Then for any perfect field $k$ of characteristic exponent $p$,
 the equivalence \eqref{eq:DM&DM_cdh},
 the previous corollary and the cancellation theorem of Voeovdsky
 (cf. \cite{V4}) show that there exists a canonical
 equivalence of triangulated categories:
\begin{equation}
\label{eq:DM&DM_cdh_eff}
\DM^{\eff}(k,R) \rightarrow \DM_\cdh^{\delta-\eff}(k,R)
\end{equation}
 where the left hand side 
 is the triangulated category of Voevodsky's motivic complexes
 with coefficients in $R$ (see \cite[chap. 5]{FSV}, \cite{CD1}).
 Moreover, the corollary also shows that this functor
 is an equivalence of $t$-categories between Voevodsky's homotopy
 $t$-structure (see \cite[chap. 5]{FSV} or \cite[Cor. 5.2]{Deg9})
 and the $\delta$-homotopy $t$-structure.

Note finally that through these equivalences of categories,
 the adjunction of triangulated categories
$$
s:\DM_\cdh^{\delta-\eff}(k,R) \leftrightarrows \DM_\cdh(k,R):w
$$
corresponds to the adjunction:
$$
\Sigma^\infty:\DM^{\eff}(k,R) \leftrightarrows \DM(k,R):\Omega^\infty
$$
of \cite[Ex. 7.15]{CD1} using again Voevodsky's cancellation theorem.
\item In the case of the stable homotopy category $\T=\SH$,
 for any perfect field $k$, the previous corollary
 shows that the $\delta$-effective category
 $\SH^{\delta-\eff}(k)$ is equivalent to the essential image
 of the canonical functor going from $S^1$-spectra to
 $\PP^1$-spectra.
\end{enumerate}
\end{ex}

\begin{rem}\label{rem:Suslin}
Consider the assumptions of point (1) in the above example.
 Let moreover $k$ be a non perfect field.
 A. Suslin has proved the following facts (see \cite{Sus}):
\begin{itemize}
\item If $F$ is a $R$-linear homotopy invariant presheaf with transfers
 over $k$, then its associated sheaf is homotopy invariant
 (the fact the latter admits transfers was already known).
 \item If $F$ is a $R$-linear homotopy invariant sheaf with
 transfers over $k$,
 then its Nisnevich cohomology is $\AA^1$-invariant.
\end{itemize}
With all that in hands, we obtain (as in the case of a perfect field)
 that $\DM^{eff}(k,R)$ can be described has the full subcategory
 of $\Der(\mathrm{Sh}^{tr}(k,R))$
 made by complexes whose cohomology sheaves
 are $\AA^1$-invariant. Moreover, the cancellation theorem holds.

Therefore we can deduce from the above results of Suslin  that
 the equivalence of categories \eqref{eq:DM&DM_cdh_eff} holds
 even when $k$ is non perfect.
 Moreover, one can also extend Voevodsky's definition
 of the homotopy $t$-structure on $\DM^{eff}(k,R)$ when $k$
 is non perfect and the equivalence \eqref{eq:DM&DM_cdh_eff}
 is an equivalence of $t$-categories.

In particular, from the results of Suslin, one deduces
 that the $\delta$-homotopy heart of $\DM_\cdh^{\delta-\eff}(k,R)$
 is the category of $R$-linear homotopy invariant sheaves with 
 transfers over $k$, even when $k$ is non perfect.
\end{rem}

\begin{num}\label{num:eq:compute_w}
We are now in position to justify computation
 \eqref{eq:compute_w}.

Recall the assumptions of this formula:
 $(S,\delta)$ is an arbitrary dimensional scheme,
 $\T=\DM_R$ as in \ref{num:convention_DM}(1) or (2)
 and $X$ is a smooth projective $S$-scheme of pure dimension $d$.
According to \ref{num:BM}(BM4), we get:
$$
M_S(X)=\Mb(X/S)(d)[2d]
$$
In particular, $M_S(X)(n)$ is by definition $\delta$-effective
 if $n \geq \delta(X)-d$. We prove $w(M_S(X)(n))$ is zero when
 $n<\delta(S)-d$. This will use the following vanishing
 of motivic cohomology, true under our assumptions:
 for any regular scheme $X$, any couple of integers $(n,m)$,
\begin{equation}\label{eq:vanishing_neg_twist}
H^{n,m}(X,R)=0 \text{ if } m<0.
\end{equation}
In case of assumption \ref{num:convention_DM}(1), this follows
 from the isomorphism of the preceding group with the $m$-th $\gamma$-graded part
 of Quillen's K-group $K_{2m-n}(X)_\QQ$ (see \cite[14.2.14]{CD3})
 and under assumption \ref{num:convention_DM}(2), one reduces to
 the case of smooth $F$-schemes using Pospescu theorem and
 \cite[3.10]{CD5} where it follows from Voevodsky's cancellation
 theorem \cite{V4}.

According to point (3) of Proposition \ref{prop:generators1},
 we only need to check that for any regular quasi-projective $S$-scheme $Y$,
 the group
$$
\Hom_{\DM(S,R)}\big(\Mb(Y/S)(m)[i],M_S(X)(n)\big)
$$
vanishes if $m \geq \delta(Y)$, $n<\delta(S)-d$ and $i$ is 
any integer. A straightforward computation, using the absolute purity
 property of $\DM_R$, gives:
$$
\Hom_{\DM(S,R)}\big(\Mb(Y/S)(m)[i],M_S(X)(n)\big) 
=H^{2d'-i,d'+n-m}(X \times_S Y,R)
$$
where $d'$ is the relative dimension of the
 quasi-projective $S$-scheme $X \times_S Y$, which is lci because
 $X \times_S Y$ and $S$ are regular under our assumptions.
 Thus the required vanishing follows from \eqref{eq:vanishing_neg_twist},
 the assumptions on $n$ and $m$ and the fact:
$$
d'=d+\delta(Y)-\delta(S).
$$
\end{num}

\begin{rem}
In our knowledge of motivic homotopy theory,
 the vanishing \eqref{eq:vanishing_neg_twist} seems to be very
 specific to mixed motives. Indeed, it is false for 
 cobordism, homotopy invariant K-theory, 
 stable homotopy groups of spheres,
 $\ell$-adic (or torsion) \'etale cohomology.
\end{rem}

\subsection{The sharpest description of generators}\label{s2.4}

\begin{num}\label{num:resolution}
The next description of generators for
 the $\delta$-homotopy $t$-structure combines the advantages
 of generators in points (1) and (2) of the previous proposition. 
 But we need the following
 assumption on our motivic category:

\medskip

(Resol) One of the following assumptions on $\T$ and $\base$ is verified:
\begin{enumerate}
\item[(i)] Each integral scheme $X$ which is essentially 
of finite type over a scheme in $\base$ admits a desingularization, i.e.,
 there exists a proper birational morphism $X' \rightarrow X$ such that
 $X'$ is regular.
\item[(ii)]  $\base$ is made of essentially of finite type $S_0$-schemes,
 where $S_0$ is a noetherian excellent scheme
 with $\dim(S_0)\leq 3$,
 and $\T$ is the homotopy category associated
 with a premotivic model category which is moreover $\QQ$-linear
 and separated (cf. \cite[2.1.7]{CD3}).
\item[(iii)] $\base$ is made of $k$-schemes essentially of finite type
 where $k$ is a perfect field of characteristic exponent $p$,
 $\T$ is $\ZZ[1/p]$-linear and there exists a premotivic adjunction:
$$
\varphi^*:\SH \leftrightarrows \T.
$$
\end{enumerate}
\end{num}

\begin{rem}
The question whether any integral
 quasi-excellent scheme
 $X$ admits a desingularization as in (Resol)(i) was raised
 in \cite{EGA4}.
 It has been proved by Temkin in \cite{Tem1} for
 all integral quasi-excellent $\QQ$-schemes so that
 our assumption (Resol)(i) holds in our actual knowledge
 when $\base$ is a category made of (excellent) $\QQ$-schemes.
\end{rem}

\begin{thm}\label{thm:strong_generators}
Let $(S,\delta)$ be a dimensional scheme. Put $I=\ZZ$ (resp. $I=\NN$).
 Assume that condition (Resol) holds.

Then the homotopy $t$-structure
 on $\T(S)$ (resp. $\T^{\delta-\eff}(S)$)
 coincides with the $t$-structure generated by objects of the form
 $\Mb(X/S)\dtw{\delta(X)}\gtw n$, where $X/S$ is proper, $X$ is regular
 and $n \in I$.
Moreover, when (Resol)(ii) or (Resol)(iii) holds,
 one can restrict to schemes $X/S$ which are projective.
\end{thm}

\begin{rem}
This description of generators is closely related to the Chow weight structures 
(cf. \S\ref{sgws} below) $w_{Chow}(-)$ as constructed in \cite{Bon10b}, 
\cite{Bon11}, \cite{Heb11}, \cite{Bon14},  \cite{BI15}, and \cite{BL16} for 
various versions of $\DM(-)$
(since these generators are certain shifts of the so-called Chow motives over 
$S$ as defined in the latter three papers; one may consider motives that are 
constructible or not and $\delta$-effective or not here). 
In particular, note that several arguments in 
  \cite{BI15} are closely related to the ones used in the current paper. 

Moreover, in  \cite{Bon14},  \cite{BI15}, and \cite{BL16} 
weight structures where $\base$ does not satisfy any version of the assumptions 
(Resol) were also considered.
In this case our level of knowledge (on the resolution of singularities)
is not sufficient to prove that motives of the type $\Mb(X/S)\gtw n$ 
with  $X$ being regular and  proper over $S$, $n\in \ZZ$, form a  generating 
family for the corresponding $\DM(S)$. However,  \cite[Theorem 3.4.2]{BL16} 
appears to yield a way to prove  (under its assumptions) the corresponding 
version of our Theorem \ref{thm:hlg&htp_t} without relying on Theorem 
\ref{thm:strong_generators}; yet the corresponding argument is not written down 
yet.

Moreover, it seems a more careful study of the Chow weights of 
motives (based on the arguments from the proof of \cite[Lemma 6.2.7]{CD4}) 
would allow to weaken the assumption (Resol) in Theorem \ref{thm:hlg&htp_t}.
\end{rem}

\begin{proof}
To simplify the notation of this proof, 
 we denote by $\T$ the triangulated category $\T(S)$
  (resp. $\T^{\delta-\eff}(S)$) and we put $t=t_\delta$.
 We also denote by $t'$ the $t$-structure generated
 (Def. \ref{df:gen_tstruct})
 by the family $\cG$ made of the compact objects of $\T$ of the form
 $\Mb(X/S)\dtw{\delta(X)}\gtw n$ where $X/S$ is proper in case (Resol)(i),
  projective in case (Resol)(ii) and (Resol)(iii), $X$ is regular
	and $n \in I$.
 In this notation we only  have to  prove that
\begin{equation}\label{eq:strong_generators}
\T_{t \geq 0} \subset \T_{t' \geq 0}
\end{equation}
as the converse inclusion is clear. Recall from Definition
 \ref{df:gen_tstruct} that $\T_{t' \geq 0}$ is the smallest
 subcategory of $\T$ which contains $\cG$
 and is stable under extensions, positive suspensions and 
 coproducts.

We begin this proof by a preliminary reduction that occurs
 only for assumption (Resol)(iii) and which uses the localization
 techniques of \cite[Appendix B]{CD4}. Given any prime $l$,
 we let $\T_{(l)}$ be the Verdier quotient of $\T$
 by the thick triangulated subcategory generated by the
 cones of maps of the form $r.1_K:K \rightarrow K$ for any integer
 $r$ not divisible by $l$ and any $\T$-spectrum $K$.
\begin{lm} Consider the preceding notations.
\begin{enumerate}
\item The family of projection functors
  $\pi_l:\T \rightarrow \T_{(l)}$, indexed by prime integers $l$
	different from $p$, is conservative.
\item There exists a unique $t$-structure $t_l$
 (resp. $t'_l$)
 on $\T_{(l)}$ such that the functor
 $\pi_l:(\T,t) \rightarrow (\T_{(l)},t_l)$
 (resp. $\pi_l:(\T,t') \rightarrow (\T_{(l)},t'_l)$)
 is $t$-exact.
\end{enumerate}
\end{lm}
Points (1) and (2) of this lemma were proved 
 in \emph{loc. cit.}, B.1.7 and B.2.2 respectively.
 Note that \eqref{eq:strong_generators} is equivalent to the assumption that
 for any $\T$-spectrum $K$ in $\T_{t \geq 0}$,
 the $\T$-spectrum $\tau_{<0}^{t'}(K)$ is trivial.
 Thus it is sufficient to prove \eqref{eq:strong_generators}
 after localizing at a prime $l \neq p$.
 This means that we can assume that $\T$ is $\ZZ_{(l)}$-linear for
 a prime $l \neq p$.

\bigskip

Let us go back to the main part of the proof. In any case,
 it is sufficient to prove that for any separated $S$-scheme $X$,
 the following property holds:
\begin{equation}\label{eq:strong_generators2} \tag{$P(X)$}
\forall n \in I, \ 
\Mb(X/S)\dtw{\delta(X)}\gtw n \in \T_{t'\geq 0}
\end{equation}
We will show this by induction on
 the integer $d=\dim(X)$. The case $d=0$ is obvious.
 Thus we can assume that $(P(X'))$ holds for any scheme $X'$ 
 of dimension less than $d$. We will use the following lemma:
\begin{lm}
Consider the inductive assumption on $d$ as above.
\begin{enumerate}
\item For any dense open subscheme $U \subset X$,
 the property $(P(X))$ is equivalent to the property $(P(U))$.
\item Let $\phi:X' \rightarrow X$ be a proper morphism
 such that any irreducible component of $X'$ dominates
 an irreducible component of $X$.
 Given any open subscheme $U \subset X$,
 we denote by $\phi_U:\phi^{-1}(U) \rightarrow U$ the pullback
 of $\phi$ to $U$. It is proper, and induces the following morphism
 -- \ref{num:BM}(BM1) --
 of Borel-Moore $\T$-spectra over $U$:
$$
\phi_U^*:\un_U=\Mb(U/U) \rightarrow \Mb(V/U).
$$
We assume there exists a dense open subscheme $U \subset X$
 such that $\phi_U^*$ is a split monomorphism.

Then property $(P(X'))$ implies property $(P(X))$.
\end{enumerate}
\end{lm}
Let us prove point (1) of this lemma.
 We denote by $Z$ the reduced subscheme of $X$
 complementary to $U$. In $\T$, we get 
 from \ref{num:BM}(BM3) the following distinguished triangle:
$$
\Mb(Z/S)\dtw{\delta(X)}[-1] \rightarrow 
 \Mb(U/S)\dtw{\delta(X)} \rightarrow  \Mb(X/S)\dtw{\delta(X)}
 \rightarrow \Mb(Z/S)\dtw{\delta(X)}.
$$
Since $U$ is dense in $X$, one gets: $\delta(X)=\delta(U)$
 and $\delta(X)>\delta(Z)$, $\dim(Z)<\dim(X)$.
 Thus, by the inductive assumption, for any integer $n \in I$,
 the $\T$-spectrum
$$
\Mb(Z/S)\dtw{\delta(X)}\gtw n[-1]
=\Mb(Z/S)\dtw{\delta(Z)}\gtw{\delta(X)-\delta(Z)+n}[\delta(X)-\delta(Z)-1]
$$
belongs to $\T_{t' \geq 0}$. Thus point (1) of the lemma follows.

Then point (2) follows easily from point (1). Indeed, in the assumptions
 of point (2), we know from point (1) that $(P(X))$ is equivalent to
 $(P(U))$. Moreover, the open subscheme $V=\phi^{-1}(U)$ of $X'$ is dense.
 So we also know that $(P(X'))$ is equivalent to $(P(V))$.
 If we apply $\phi_{U!}$ to the morphism $\phi^*_U$ in $\T(U)$,
 we get a split monomorphism $\Mb(U/S) \rightarrow \Mb(V/S)$.
 Thus, since $\T_{t'\geq 0}$ is stable by direct factors,
 we obtain that $(P(V))$ implies $(P(U))$ and this concludes.

\bigskip

Let us go back to the proof of property $(P(X))$ by induction on the
 dimension $d$ of $X$.  Since $\Mb(X/S)=\Mb(X_{red}/S)$,
 we can assume that $X$ is reduced.
 Let $(X_\lambda)_{\lambda}$ be the irreducible components of $X$.
 The closed subset $Z=\cup_{\lambda \neq \mu} X_\lambda \cap X_\mu$
 in $X$ is rare. Applying point (1) of the preceding lemma, we can replace
 $X$ by $X-Z$. So we can assume that $X$ is integral.
 Because $X/S$ is separated of finite type,
 it admits an open embedding into a proper $S$-scheme $\bar X$
 that we can assume to be integral as $X$ is integral.
 According to the preceding lemma, we can replace $X$ by $\bar X$.
 In other words, we can assume $X$ is proper over $S$ and integral.

Let us treat the case of \underline{assumption (Resol)(i)}.
 Applying this assumption to $X$, we get a proper birational map
 $X' \rightarrow X$ such that $X'$ is regular. Because $X'/S$ is also proper,
 property $(P(X'))$ is obviously true -- by definition of $t'$.
 Thus point (2) of the previous lemma shows that
  $(P(X))$ is true as required.

\bigskip

Consider the cases (Resol)(ii) and (iii).
According to Chow lemma as stated in \cite[5.6.1]{EGA2}, there exists
 a birational map $X' \rightarrow X$ over $S$ such that $X'/S$ is projective.
 We deduce from point (2) of the preceding lemma
 that we can assume $X/S$ is projective.

We now treat the case of \underline{assumption (Resol)(ii)}.
 Under this assumption,
 it follows from \cite[1.2.4]{Tem2} that
 there exists a projective alteration\footnote{Recall
  that an \emph{alteration} is a proper surjective
  generically finite morphism.}
 $\phi:X' \rightarrow X$ such that $X'$ is regular connected.
 In particular, property $(P(X'))$ is obviously true. 
 Moreover, there exists a dense open subscheme $U \subset X$
 such that $U$ is regular and $\pi$ is finite on $V=\phi^{-1}(U)$.
 Then, according to \cite[3.3.40]{CD3} applied to the motivic category
 $\T$ and the finite morphism $\phi_U:V \rightarrow U$,
 the adjunction map $\un_U \rightarrow \phi_{U*}\phi_U^*(\un_U)$
 is a split monomorphism in $\T(U)$. So we can apply point (2)
 of the preceding lemma, and we deduce property $(P(X))$.

Let us finally treat the case of \underline{assumption (Resol)(iii)}.
 Using the reduction done at the beginning of the proof,
 we fix a prime $l \neq p$, and prove the assertion in
 $\T_{(l)}$.\footnote{Note indeed that the preceding lemma can 
 equally be applied to $\T_{(l)}$.}
 According to \cite[X, 3.5]{Gabber}, there exists a projective
 generically \'etale alteration $\phi:X' \rightarrow X$
 of generic degree $d$ prime to $l$ such that $X'$ is regular. 
 In particular, $X'$ is projective over $S$ and regular
 so property $(P(X'))$ holds.
 Let $\phi':\spec(L) \rightarrow \spec(K)$
  be the morphism induced by $\phi$ on generic points.
 Because $\phi'$ is finite \'etale, it induces a morphism in $\SH(K)$:
$$
\phi^{\prime*}:\un_K \xrightarrow{ad} \phi'_*\phi^{\prime*}(\un_K)
 =\phi'_!(\un_L).
$$
According to \cite[Lemmas B.3, B.4]{LYZR}, there exists
 a morphism $s:\phi'_!(\un_L) \rightarrow \un_K$ again in $\SH(K)$
 such that
$$
\phi^{\prime*} \circ s=d.Id+\alpha
$$
where $\alpha$ is a nilpotent endomorphism of $\un_K$ in $\SH(K)$.
 In particular, the preceding relation means that 
 $\phi^{\prime*} \circ \frac s d$ is an isomorphism
 in $\SH(K)_{(l)}$.

By the continuity property of $\SH$
 (\cite[4.3.2]{CD3}), there exists a dense open subscheme
 $U \subset X$, $V=\phi^{-1}(U)$,
 such that $s$ can be lifted to a map
 $s_U:\phi'_!(\un_V) \rightarrow \un_U$.
 Further, we can assume by shrinking $U$
 that $\phi^{*}_U \circ \frac {s_U} d$ is an isomorphism
 in $\SH(U)_{(l)}$.

Because $\varphi^*:\SH \rightarrow \T$ is a premotivic adjunction
 between motivic categories, it commutes with $f_!$
 (see \cite[2.4.53]{CD3}). Thus, we obtain that
 $\varphi^*\big(\frac {s_U} d\big)$ is a splitting of $\phi^*_U$ in
 $\T(U)_{(l)}$.
 So finally point (2) of the preceding lemma allows to conclude
  that $(P(X))$ holds.
\end{proof}

\begin{rem}\label{rbs}
\Mikhail{ The main application of the previous 
theorem will be given in the next section.
 Here we note that our result appears to be non-trivial already when 
$S=\operatorname{Spec}(k)$ (for $k$ being a perfect field of positive 
characteristic $p$). This case of the statement (along with the easier $p=0$ 
case)   has found quite interesting applications in \cite{BS16}, where it was 
used to relate the motivic homology theory (for motives) to the so-called 
Chow-weight homology ones.}
\end{rem}

Now we use our theorem to deduce certain results which are of independent 
interest.
The following corollary improves results of \cite{ayoub1}
 and \cite{CD3}.
\begin{cor}\label{cor:generators_T}
Suppose $\T$ satisfies assumption (Resol).
 Then for any scheme $S$, the triangulated category $\T_c(S)$
 (resp. $\T(S)$)
 is generated\footnote{See Conventions page \pageref{conventions}
  for the  definition of ``generated'' in the triangulated context.}
  by spectra (resp. arbitrary coproducts of spectra)
 of the form $f_*(\un_X)(n)$
 where $f:X \rightarrow S$ is a proper morphism
 (resp. projective under (Resol)(ii) or Resol(iii))),
 $X$ a regular scheme and $n \in \ZZ$ an integer.
\end{cor}
This follows from the preceding proposition,
 Lemma \ref{lm:non-neg-degenerated} and
 Remark \ref{rem:htp_t_nonnegdeg}.

Here is an application of the previous corollary,
 which improves significantly the theory of $\ZZ[1/p]$-linear
 triangulated motivic categories over the category of $\mathbb F_p$-schemes,
 on the model of \cite{ayoub1} or \cite[\textsection 4]{CD3}.
\begin{thm}\label{thm:duality}
Suppose $\T$ satisfies assumption (Resol)
 and the absolute purity property (Definition
 \ref{df:abs_purity}).
 Then the following properties hold:
\begin{enumerate}
\item The category of constructible $\T$-spectra is
 stable under the six functors if one restricts to morphisms
 $f$ of finite type for the functor $f_*$.
\item For any regular scheme $S$ and any separated morphism
 $f:X \rightarrow S$, the $\T$-spectrum
 $D_X=f^!(\un_S)$, which is constructible according to the first
 point, is dualizing for constructible $\T$-spectra over $X$:
 for any such $\T$-spectrum $M$, the natural map
\begin{equation}\label{eq:biduality_iso}
M \rightarrow \uHom\big(\uHom(M,D_X),D_X\big)
\end{equation}
is an isomorphism.
\end{enumerate}
\end{thm}
\begin{proof}
The proof of the first point follows from the following ingredients:
\begin{itemize}
\item Gabber's refinement of De Jong alteration theorem
 (cf. \cite[X, 3.5]{Gabber}).
\item The argument of Riou for the existence of trace maps
 for finite \'etale morphisms in $\SH$ satisfying
 the degree formula (cf. \cite[Lemmas B.3, B.4]{LYZR}).
\end{itemize}
Then the proof of Gabber (in the case of the derived category of
 prime to $p$ torsion \'etale sheaves) work through
 as explained in \cite[proof of Th. 6.4]{CD4}. The only supplementary
 observation is that one only needs local splitting of pullbacks
 by finite \'etale morphisms,
 which are given by the maps $\varphi^*\big(\frac {s_U} d\big)$
 as in the end of the proof of the previous theorem.

The second assertion follows essentially from the proof
 of \cite[7.3]{CD5}:
 using the previous corollary, it is sufficient to check that
 the map \eqref{eq:biduality_iso} is an isomorphism
 when $M=p_!(\un_Y)$ where $p:Y \rightarrow X$ is a projective
 morphism and $Y$ is regular. In this case, one can compute
 the right hand side of \eqref{eq:biduality_iso} as:
$$
p_!\uHom\big(\un_Y,\uHom(\un_Y,(fp)^!(\un_S))\big).
$$
so that, replacing $X$ by $Y$, 
 we can assume that $X$ is regular separated
 over $S$ and $M=\un_X$. The assertion is then Zariski local
 in $X$ so that we can assume in addition that $X/S$
 is quasi-projective.
 Because $\T$ satisfies absolute purity by assumption,
 we can apply Corollary \ref{cor:fund_class} which implies that
 $D_X=f^!(\un_S) \simeq \Th(\tau_f)$ where $\tau_f$ is the virtual tangent
 bundle of $f$. Because the $\T$-spectrum $\Th(\tau_f)$ is
 $\otimes$-invertible, the result follows trivially.
\end{proof}

\begin{ex}\label{ex:duality_SH}
The new application of the preceding proposition
 is given by the case of the stable homotopy category $\SH[1/p]$
 restricted to the category of $\F_p$-schemes for a prime $p>0$.

Thus constructible spectra with $\ZZ[1/p]$-coefficients are
 stable under the six operations if one restricts to excellent
 $\F_p$-schemes.
 Moreover, Grothendieck duality holds for these spectra.
 Note for example that given any field $k$ of characteristic $p$,
 and any quasi-projective regular $k$-scheme $X$ with virtual
 tangent bundle $\tau_X$, the Thom space $\Th(\tau_X)$ is
 naturally a dualizing object for constructible $\ZZ[1/p]$-linear
 spectra over $X$.
\end{ex}

Another interesting application of Corollary \ref{cor:generators_T}
 is the following result that extends previous theorems due to
 Cisinski and the second named author (\cite[4.4.25]{CD3}
 and \cite[9.5]{CD4}).
\begin{prop}\label{prop:adj&f_*}
Consider a premotivic adjunction
$$
\varphi^*:\T \leftrightarrows \T'
$$
of triangulated motivic categories over $\base$ which satisfies,
 the following assumptions:
\begin{enumerate}
\item[(a)] $\base$ is the category of $F$-schemes
 for a prime field $F$ of characteristic exponent $p$.
\item[(b)] $\T$ and $\T'$ are $\ZZ[1/p]$-linear,
 continuous (see Paragraph \ref{num:commutes_with_limits})
 and are the aim of a premotivic
 adjunction whose source is $\SH$.
\end{enumerate}
Then for any morphism $f:X \rightarrow S$ in $\base$,
 for any $\T$-spectrum $\E$ over $X$,
 the canonical exchange transformation:
\begin{equation}
\label{eq:prop:adj&f_*}
\varphi^*f_*(\E) \rightarrow f_*\varphi^*(\E)
\end{equation}
is an isomorphism.
\end{prop}
\begin{proof}
The proof follows that of \cite[4.4.25]{CD3}
 with some improvements.

We start by treating the case where $S=\spec(F)$
 is the spectrum of the prime field $F$.

We first show that one can assume $f$ is of finite type.
 Note that the functor $f_*$ commutes with arbitrary coproducts
 -- this follows formally from the fact its left adjoint
 preserves the generators which are assumed to be compact
 (see \cite[1.3.20]{CD3} for details).
 In particular, it is sufficient to prove that
 \eqref{eq:prop:adj&f_*} is an isomorphism when $\E$ is
 constructible.
 But $X$ can be written as a projective limit of
 a projective system of $F$-schemes of finite type
$$
X_i \xrightarrow{f_i} \spec(F), i \in I
$$
with transition maps $f_{ij}$.
Let us denote by $p_i:X \rightarrow X_i$ the canonical projection.
 Since $\T$ is continuous, we get
 an equivalence of triangulated categories (\cite[4.3.4]{CD3}):
\begin{equation}\label{eq:continuity_equivalence}
2-\ilim_{i \in I} \T_c(X_i) \rightarrow \T_c(X).
\end{equation}
Thus there exists a family $(\E_i)_{i \in I}$ such that
 for all $i \in I$, $\E_i$ is a constructible $\T$-spectrum
 over $X_i$
 with a given identification $f_i^*(\E_i) \simeq \E$
 for all index $i \in I$,
 and for all $i \leq j$, there is a map 
 $f_{ij}^*(\E_i) \rightarrow \E_j$ compatible with the previous
 identification.
 In particular, for all index $i \in I$, we get a canonical map:
$$
f_{i*}(\E_i) \rightarrow f_{i*}p_{i*}p_i^*(\E_i) \simeq f_*(\E).
$$
These maps are compatible with the transition maps $f_{ij}$
 by the following morphism:
$$
f_{i*}(\E_i) \rightarrow f_{i*}f_{ij*}f_{ij}^*(\E_i)
 \rightarrow f_{j*}(\E_j).
$$
So we get a canonical morphism of spectra over $F$:
$$
\psi:\hocolim_{i \in I}\big(f_{i*}(\E_i)\big) \rightarrow f_*(\E).
$$
A direct consequence of the continuity property for $\T$
 is the following lemma:
\begin{lm}
Under the preceding assumption,
 the morphism $\psi$ is an isomorphism.
\end{lm}
Indeed, it is sufficient to check $\psi$ is an isomorphism
 after applying the functor
 $\Hom_{\T(F)}(K,-)$ for any constructible $\T$-spectrum $K$.
 Then it follows from the equivalence \eqref{eq:continuity_equivalence}
 and from the fact $K$ is compact.

In particular, if we know the result for morphisms of finite
 type with target $\spec(F)$, we can conclude for the morphism $f$:
\begin{align*}
\varphi^*f_*(\E)
 & \stackrel{(1)}\simeq \varphi^*\hocolim_{i \in I}\big(f_{i*}(\E_i)\big) 
 \stackrel{(2)}\simeq \hocolim_{i \in I}\big(\varphi^*f_{i*}(\E_i)\big) \\
 & \stackrel{(3)}\simeq \hocolim_{i \in I} 
\left(f_{i*}\big(\varphi^*(\E_i)\big)\right) 
 \stackrel{(4)}\simeq f_*\varphi^*(\E_i)
\end{align*}
where (1) (resp. (4)) follows from the preceding lemma
 (the obvious analog of the preceding lemma for $\T'$),
 (2) is valid because $\varphi^*$ is a left adjoint thus commutes to homotopy
 colimits, and (3) is true by assumption.

So let us assume now that $f:X \rightarrow \spec(F)$ is of finite type.
 According to Corollary \ref{cor:generators_T},
 the category $\T(X)$ is generated by arbitrary coproducts
 of $\T$-spectra of the form $p_*(\un_Y)(n)$ for
 $p:Y \rightarrow X$ proper, $Y$ regular and $n$ any integer.
 Since $f_*$ commutes with coproducts as we have already seen,
  it is sufficient to show \eqref{eq:prop:adj&f_*} is an
  isomorphism when $\E$ is one of these particular $\T$-spectra.
 Thus, because $\varphi^*$ commutes with $p_*=p_!$ for $p$ proper
 (see \cite[2.4.53]{CD3}) and with twists, we only need to prove
 that \eqref{eq:prop:adj&f_*} is an isomorphism when $\E=\un_X$
 for $X$ a regular $F$-scheme of finite type. In particular,
 $f:X \rightarrow \spec(F)$ is smooth, because $F$ is perfect.
 From the six functors formalism, it now
 formally follows that the spectrum $f_*(\un_X)$ is rigid,
 in both the monoidal categories $\T(k)$ and in $\T'(k)$,
 with strong dual $f_\sharp(\Th(-\tau_f))$
 where $\tau_f$ is the tangent bundle of $f$
 (see for example \cite[2.4.31]{CD3}). 
 To conclude, it is sufficient now to recall that $\varphi^*$
 is monoidal, so it respects strong duals and it commutes
 with $f_\sharp$. This concludes the proof in the case where $S$ is the
 spectrum of a field.

The general case now formally follows,
 as in the proof of \cite[4.4.25]{CD3}.
\end{proof}

\begin{rem}
The result obtained here is stronger than its analogue in \cite{CD3}
 or \cite{CD4}, where we had to assume $f$ is of finite type,
 $\E$ is constructible.
 This is because we have used the continuity property as well as the
 compactly generated assumption on triangulated motivic categories.
 Note in particular that we do not need any semi-separatedness assumption
 on the triangulated motivic categories $\T$ and $\T'$.
\end{rem}

\begin{cor}\label{cor:prop:adj&f_*}
Under the assumption of the previous proposition,
 the functor $\varphi^*$ commutes with the 6 functors.
\end{cor}
This follows formally from the previous proposition
 (see the proof of \cite[4.4.25]{CD3} for details).

\begin{ex}
The corollary applies to all motivic adjunctions
 of the diagram \eqref{eq:premotivic_adj} provided
 we restrict to equal characteristics schemes and
 we invert their characteristic exponent.
\end{ex}

\section{Local description of homology sheaves}

\subsection{The $\delta$-niveau spectral sequence}

Below, we extend the considerations of \cite[\textsection 3]{BO},
 to the framework of motivic categories. Actually, the new point
 is to work in the absolute case (without a base field) and to
 use arbitrary dimension functions.
\begin{df}
Let $(X,\delta)$ be a dimensional scheme.
 A \emph{$\delta$-flag} on $X$ is a sequence $Z_*=(Z_p)_{p \in \ZZ}$
 of reduced closed subschemes of $X$ such that $Z_p \subset Z_{p+1}$
 and $\delta(Z_p) \leq p$.

We will denote by $\F_\delta(X)$ the set of $\delta$-flags ordered
 by term-wise inclusion.
\end{df}
Note that the ordered set $\F_\delta(X)$ is non-empty
 and cofiltered
 (\emph{i.e.} two elements admit a lower bound).

\begin{num}\label{num:exact_couple}
Let us recall the framework of exact couples
 to fix notations (following \cite[\textsection 1]{Deg11}).
 An \emph{exact couple} in
 an abelian (resp. triangulated, pro-triangulated\footnote{\emph{i.e.}
 the category of pro-objects of a triangulated category;}) category $\C$
 \emph{with homological conventions} is the data of bigraded objects
 $D$ and $E$ of $\C$ and homogeneous morphisms in $\C$ as pictured
 below:
\begin{equation}
\begin{split}
\label{hlg_exact_couple}
\xymatrix@R=20pt@C=32pt{
D\ar^{(1,-1)}_\alpha[rr]
 &&  D\ar^{(0,0)}_/8pt/\beta[ld] \\
 & E\ar^{(-1,0)}_/-6pt/\gamma[lu] &
}
\end{split}
\end{equation}
where the degree of each morphism is indicated,
 and such that the above diagram is a (long) exact sequence 
 (resp. a triangle, a pro-triangle) in each degree.

By contrast, an \emph{exact couple with cohomological conventions}
 is defined by inverting the arrows in the above diagram
 and taking the opposite degrees.

Consider now a dimensional scheme $(S,\delta)$
 and a separated $S$-scheme $X$. Denote again by $\delta$
 the dimension function on $X$ induced by that on $S$.
For any $\delta$-flag $Z_*$ on $X$ and any integer $p \in \ZZ$,
 we get according to Paragraph \ref{num:BM}(BM3),
 a distinguished triangle:
$$
\Mb(Z_p-Z_{p-1}/S) \rightarrow \Mb(Z_p/S)
 \rightarrow \Mb(Z_{p-1}/S)
 \rightarrow \Mb(Z_p-Z_{p-1}/S)[1]
$$
in $\T(S)$. Moreover, these triangles are contravariantly functorial
with respect to the order on $\delta$-flags. Therefore, taking
the (pseudo-)projective limit in the pro-triangulated category
$\pro{\T(X)}$, we get a distinguished pro-triangle:
\begin{equation}\label{eq:dniveau_triangle}
G_p\Mb(X/S) \xrightarrow{j_{p*}} F_p\Mb(X/S)
 \xrightarrow{i_p^*} F_{p-1}\Mb(X/S)
 \xrightarrow{\partial_p} G_p\Mb(X/S)[1]
\end{equation}
where we have put:
$$
F_p\Mb(X/S)=\pplim{Z_*} \Mb(Z_p/S),
 \qquad G_p\Mb(X/S)=\pplim{Z_*} \Mb(Z_p-Z_{p-1}/S).
$$
\end{num}
\begin{df}\label{df:niveau_exact_couple}
Consider the above assumptions and notations.

We define the \emph{$\T$-motivic $\delta$-niveau exact couple
 associated with $X/S$}
 as the following pro-exact couple with cohomological conventions
 in $\pro\T(S)$:
$$
D^{p,q}=F_p\Mb(X/S)[p+q], \qquad E^{p,q}=G_p\Mb(X/S)[p+q],
$$
and with morphisms in degree $(p,q)$ induced by the $(p+q)$-suspension
 of the exact triangle \eqref{eq:dniveau_triangle}. 
\end{df}

\begin{num}
Consider the notations of the above definition.
Assume we are given a Grothendieck abelian category $\A$
 and an exact functor\footnote{\emph{i.e.} a functor
 sending distinguished triangles to long exact sequences}:
$$
\mH:\T(S)^{op} \rightarrow \A.
$$
Then one can extend this functor to a pro-$\T$-spectrum
 $M_\bullet=(M_i)_{i \in I}$ over $S$ by the formula:
$$
\hmH(M_\bullet)=\ilim_{i \in I} \big(\mH(M_i)\big)
$$
where the transition morphisms of the colimit are induced
 by the transition morphisms of the pro-object $P_\bullet$.
 Since filtered colimits are exact in $\A$,
 the functor $\hmH$ sends distinguished pro-triangle to long
 exact sequences.
 In particular, if we apply $\hmH$ to the pro-exact couple
 of the previous definition, we get an exact couple in the
 abelian category $\A$ with homological conventions,
 and therefore a homological spectral sequence:
\begin{equation}\label{eq:general_dniveau_ssq}
E^1_{p,q}=\hmH_{p+q}\big(F_p\Mb(X/S)\big)
 \Rightarrow \mH_{p+q}\big(\Mb(X/S)\big).
\end{equation}
Note that because
 the dimension function $\delta$ is necessarily bounded,
 this spectral sequence always converges.
 This general definition will be especially useful in the particular
 case of the following definition.
\end{num}
\begin{df}\label{df:niveau_ssp}
Consider as above a dimensional scheme $(S,\delta)$ and a
 separated $S$-scheme $X$.
 Let $\E$ be a $\T$-spectrum.
 We define the \emph{$\delta$-niveau spectral
 sequence of $X/S$ with coefficients in $\E$} as the spectral
 sequence \eqref{eq:general_dniveau_ssq} associated with
 the functor $\mH=\Hom(-,\E)$.
\end{df}
It is useful to consider the $\ZZ$-graded version
 of this spectral sequence,
 replacing $\E$ by $\E(-n)$ for any integer $n \in \ZZ$.
 Then as in \cite[3.7]{BO},
 it is easy to check it admits the following form (up to a canonical
 isomorphism):
\begin{equation}\label{eq:dniveau_ssq}
^\delta E^1_{p,q}=\bigoplus_{x \in X_{(p)}} \hat \E^{BM}_{p+q,n}(x/S)
 \Rightarrow \E^{BM}_{p+q,n}(X/S)
\end{equation}
where $X_{(p)}=\{x \in X \mid \delta(x)=p\}$ and for any point $x \in X$,
 with reduced closure $Z(x)$ in $X$, we put:
\begin{equation}\label{eq:hlg_prelim}
\hat \E^{BM}_{p+q,n}(x/S)
:=\ilim_{U \subset Z(x)} \E^{BM}_{p+q,n}(U/S)
\end{equation}
where $U$ runs over non-empty open subschemes
 of $Z(x)$. Moreover, using the functoriality properties
 of Borel-Moore spectra (cf. \ref{num:BM}, (BM1) and (BM2)),
 one easily checks that this spectral sequence is covariant with
 respect to proper morphisms and contravariant with respect
 to \'etale morphisms.\footnote{Actually the opposite functoriality holds
 for the $\T$-motivic $\delta$-niveau exact couple.}

\begin{ex}\label{ex:niveau_ssp}
Let $S$ be a regular scheme with dimension function $\delta=-\codim_S$
 and $X$ be a separated $S$-morphism.

We fix a ring of coefficients $R$ flat over $\ZZ$ and consider
 the $R$-linear motivic category $\DM_R$ following
 the conventions of Paragraph \ref{num:convention_DM},
 point (1) or (2).

Suppose $\E=\un_X$ is the constant motive over $X$ in $\DM(X,R)$.
 Then the $\delta$-niveau spectral sequence has the following form:
$$
^\delta E^1_{p,q}=\bigoplus_{x \in X_{(p)}} \hat H^{BM}_{p+q,n}(x/S,R)
 \Rightarrow H^{BM}_{p+q,n}(X/S,R)
$$
where $H^{BM}_{**}$ stands for $R$-linear motivic Borel Moore homology.

Suppose $x \in X$ is a point such that $\delta(x)=p$.
 Let $s$ be its projection to $S$.
 Note that $\spec(\kappa(x)) \rightarrow S$ is a limit of lci morphisms
 whose relative dimension is:
\begin{equation}\label{eq:relative_dim}
\dtr(\kappa(x)/\kappa(s))-\codim_X(s)=\delta(x)=p.
\end{equation}
Moreover, because $S$ is regular, $\T$ satisfies absolute purity,
 and motivic cohomology commutes with projective limits
 (see \textsection \ref{num:commutes_with_limits}),
 we deduce from Corollary \ref{cor:fund_class} the following computation:
$$
\hat H^{BM}_{p+q,n}(x/S,R)=\HM^{p-q,p-n}(\kappa(x),R)
$$
where $\HM^{**}$ denotes $R$-linear motivic cohomology.
 Note that in each cases considered above, one gets:
$$
\HM^{p-q,p-n}(\kappa(x),R)=\begin{cases}
0 & \text{if } (p-q>p-n)\text{ or }(p-n<0), \\
K_{p-q}^M(\kappa(x)) \otimes_\ZZ R & \text{if } q=n,
\end{cases}
$$
where $K_*^M$ denotes Milnor K-theory.

In particular, the preceding spectral sequence is concentrated
 in the region $p\geq n$ and $q \geq n$. Moreover,
 as in \cite[2.7]{Deg11}, one can identify the differential
\begin{equation}\label{eq:divisor_class_map}
^\delta E^1_{n,n}=Z_{\delta=n}(X) \otimes R
 \xleftarrow{\quad d^1_{n+1,n}\quad }
 \bigoplus_{x \in X_{(n+1)}}  \big(\kappa(x)^\times\big) \otimes_\ZZ R
 = {}^\delta E^1_{n+1,n}
\end{equation}
with the classical divisor class map, where $Z_{\delta=n}(X)$
 is the abelian group of algebraic cycles in $X$ of $\delta$-dimension
 $n$ (\emph{i.e.} the free abelian group generated by $X_{(n)}$).

Thus we have proved the following proposition:
\end{ex}
\begin{prop}\label{prop:comput_BM_w_Chow}
Let $S$ be a regular scheme with dimension function $\delta=-\codim_S$,
 $X$ a separated $S$-scheme,
 and $n$ an integer. Let $R$ be a ring of coefficients.
 
Assume that one of the following conditions holds:
\begin{enumerate}
\item[(a)] $R=\QQ$,
\item[(b)] $S$ is a $\QQ$-scheme and $R=\ZZ$,
\item[(c)] $S$ is an $\F_p$-scheme and $p \in R^\times$.
\end{enumerate}
Then one has a canonical isomorphism:
$$
H_{2n,n}^{BM}(X/S,R) \simeq CH_{\delta=n}(X) \otimes R
$$
where the right hand side is the Chow group of $R$-linear
 algebraic cycles in $X$ of $\delta$-dimension $n$.
\end{prop}

\begin{rem}
\begin{enumerate}
\item The graduation on the Chow group of cycles by a
 dimension function $\delta$ first appeared, to our knowledge,
 in \cite[Chap. 41, Def. 9.1]{stack}.
\item Note of course that the case where $S$ is the spectrum
 of a field was already well known --- but with a 
 different definition of Borel-Moore motivic homology (though equivalent,
  see \cite{CD5}).
\item It is straightforward to show that the isomorphism of the proposition
 is functorial with respect to smooth pullbacks and proper
 push-forwards. It can also be proved it is functorial
 with respect to lci pullbacks. For all this, see the method
 of \cite{Deg10}, in particular Corollary 3.12
 and Proposition 3.16.
\end{enumerate}
\end{rem}

The following formulation of the preceding result
 is a kind of duality statement.
\begin{cor}\label{cor:chow_codim}
Consider the assumptions of the previous proposition
 and assume $X/S$ is equidimensional of dimension $d$
 (see \cite[13.3.2]{EGA4}).

Then one has a canonical isomorphism:
$$
H_{2n,n}^{BM}(X/S,R) \simeq CH^{d-n}(X) \otimes R
$$
where the right hand side is the Chow group of $R$-linear
 algebraic cycles in $X$ of codimension $d-n$.
\end{cor}
\begin{proof}
This follows from the preceding proposition and the following equality
 for a point $x \in X$ with image $s \in  S$
 (see \cite[13.3.4]{EGA4}, as $\mathcal O_{S,s}$ is universally catenary because
 it is a regular local ring):
$$
\codim_X(x)=\codim_S(s)+d-\dtr(\kappa(x)/\kappa(s))=d-\delta(x).
$$
\end{proof}

\begin{num}\label{num:pro&essentially}
Let us go back to the case of an abstract motivic triangulated
 category $\T$ satisfying our general conventions.
 Consider a scheme $S$, noetherian according to our assumptions.

Let us consider the categories $\Setx S$ (resp. $\bSetx S$)
 made by $S$-schemes which are separated (of finite type)
 (resp. essentially separated) and whose morphisms are \'etale
 $S$-morphisms.
 We let $\Petx S$ be the category of pro-objects of $\Setx S$
 which are essentially affine.\footnote{\emph{i.e.} the transition
 morphisms are affine up to a large enough index.}
 Any pro-object $X_\bullet$ of $\Petx S$ admits
 a projective limit $X$ in the category of $S$-schemes
 which is separated according to \cite[8.10.5]{EGA4}, 
 and essentially of finite type by definition
  (see the conventions of this paper, p. \pageref{conventions}).
 Thus we get a functor:
$$
L:\Petx S \rightarrow \bSetx S, \
 X_\bullet \mapsto (\plim X_\bullet).
$$
\end{num}
\begin{lm}\label{lm:pro_schemes_essentially}
In the notation above  the functor $L$ is an equivalence of categories.
\end{lm}
\begin{proof}
The essential surjectivity of $L$ follows from the definition
 of essentially of finite type and the reference \cite[8.10.5(v)]{EGA4}
 (resp. \cite[17.7.8]{EGA4})
 for the fact that the property of being separated (resp. \'etale)
 is compatible with projective limits.
 Finally, we get that $L$ is fully faithful by applying \cite[8.13.2]{EGA4}.
\end{proof}

\begin{num}
The preceding lemma allows to canonically extend
 some of the functors considered previously.
 First, we can define the following functor:
$$
\hMb:\Petx S \rightarrow \pro{\T(S)}, \
 (X_i)_{i \in I} \mapsto \pplim{i \in I} \Mb(X_i/S).
$$
The preceding lemma tells us that it uniquely corresponds
 to a functor defined on $\bSetx S$, and it obviously extends
 the functor $\Mb$ from separated $S$-schemes
 to essentially separated $S$-schemes.

Moreover, given a $\T$-spectrum $\E$, we can also define:
$$
\hat \E^{BM}_{**}:\Petx S \rightarrow \ab^{\ZZ^2}, \
 (X_i)_{i \in I} \mapsto \ilim_{i \in I} \E^{MB}_{**}(X_i/S).
$$
\end{num}
\begin{df}\label{df:BM_essentially}
Given any essentially separated $S$-scheme $X$,
 we will denote by $\hMb(X/S)$ (resp. $\hat \E^{BM}_{**}(X/S)$)
 the unique pro-spectrum (resp. bi-graded abelian group)
 obtained by applying the corresponding functor defined above
 to any pro-scheme $X_\bullet$ in $\Petx S$ such that $L(X_\bullet)=X$.
\end{df}
The pro-spectrum $\hMb(X/S)$ (resp. abelian group $\hat \E^{BM}_{**}(X/S)$)
 is functorially contravariant (resp. covariant) in $X/S$
 with respect to \'etale morphisms.

\begin{rem}\label{rwhen}
It follows from the previous lemma that the notation of this definition
 coincides with that of formula \eqref{eq:hlg_prelim} when,
 given a point $x \in X$, one denotes abusively by $x/S$ the 
 essential separated $S$-scheme with structural morphism:
 $\spec(\kappa(x)) \rightarrow X \rightarrow S$.
\end{rem}

\begin{num}
Consider a (noetherian) scheme $S$
 and $i:Z \rightarrow X$ a closed immersion between
 essentially separated $S$-schemes.
 Let us fix pro-schemes $(X_s)_{s \in I}$ and $(Z_t)_{t \in J}$
 in $\Petx S$ such that $L(X_\bullet)=X$ and $L(Z_\bullet)=Z$.
 
Because $X$ is noetherian, the ideal of $Z$ in $X$
 is locally finitely generated and we can find 
 indexes $s \in I$ and $t \in J$ such that the closed immersion $i$
 can be lifted to a closed immersion $i_s:Z_t \rightarrow X_s$.
 Thus, by reducing $I$ and replacing $J$ by $I$,
 one can find a morphism
 of pro-objects $i_\bullet:Z_\bullet \rightarrow X_\bullet$
 such that for any morphism $s \rightarrow s'$ in $I$, the
 following commutative diagram is cartesian:
$$
\xymatrix@=12pt{
Z_s\ar^{i_t}[r]\ar[d] & X_s\ar[d] \\
Z_{s'}\ar^{i_{t'}}[r] & X_{s'},
}
$$
where the vertical maps corresponds to the transition morphisms
 of the pro-objects $Z_\bullet$ and $X_\bullet$,
 and the horizontal maps are closed immersions.

Note that $U_\bullet=(X_s-Z_s)_{s \in I}$ is a pro-scheme in $\Petx S$
 and we have $L(U_\bullet)=U$.
 Then, we deduce from \ref{num:BM}(BM3) a pro-distinguished triangle:
$$
\hMb(U/S) \xrightarrow{j_*} \hMb(X/S)
 \xrightarrow{i^*} \hMb(Z/S)
 \xrightarrow{\partial_i} \hMb(U/S)[1].
$$
Using again \cite[8.13.2]{EGA4}, we deduce that this triangle
 does not depend on the lifting of $i$ constructed above.

Given a spectrum $\E$, 
 we also deduce a canonical long exact sequence of abelian groups:
$$
\cdots \rightarrow \hat \E_{n+1,m}^{BM}(U/S) \xrightarrow{\partial_i}
 \hat \E_{n,m}^{BM}(Z/S) \xrightarrow{i_*}
 \hat \E_{n,m}^{BM}(X/S) \xrightarrow{j^*}
 \hat \E_{n,m}^{BM}(U/S) \rightarrow \cdots
$$
Assuming we are given a dimension function $\delta$ on $S$,
 it is now straightforward to generalize Definitions
 \ref{df:niveau_exact_couple} and \ref{df:niveau_ssp}
 to the case of an essentially separated $S$-scheme $X$.
 Therefore we get:
\end{num}
\begin{prop}\label{prop:df_niveau_ssp_essntially}
Let $(S,\delta)$ be a dimensional scheme
 and $X$ be an essentially separated $S$-scheme.
 Then there is a canonical spectral sequence
 (with homological conventions)
$$
^\delta E^1_{p,q}=\bigoplus_{x \in X_{(p)}} \hat \E^{BM}_{p+q,\gtw n}(x/S)
 \Rightarrow \hat \E^{BM}_{p+q,\gtw n}(X/S)
$$
where $X_{(p)}$ stands for the set of points $x \in X$
 such that $\delta(x)=p$.
\end{prop}

\subsection{Fiber homology}

Let $x \in S(E)$ be a point.
 As a morphism $x:\spec(E) \rightarrow S$,
 it is according to our conventions an essentially
 separated morphism.
 Thus, it follows from Definition \ref{df:BM_essentially}
 that the bi-graded abelian group 
 $\hat \E^{BM}_{**}(x)$ is well defined.
 However, in the proof of the forthcoming theorem,
 we will have to be more precise. This motivates 
 the following definition. 
\begin{df}\label{df:Smodels}
Let $S$ be a scheme and $x \in S(E)$ be a point.

An \emph{$S$-model} of $x$ will be an affine regular $S$-scheme
 $X=\spec(A)$ of finite type
 such that $A$ is a sub-ring of $E$ whose fraction field is equal to $E$
 and such that $x$ is equal to the composite of the natural map
 $\spec(E) \rightarrow \spec(A)$ and the structural map
  of $X/S$.

We let $\M(x)$ be the essentially small category 
 whose objects are $S$-models of $x$ and morphisms
 are open immersions.
\end{df}

\begin{num}\label{num:fiber_hlg}
Consider a scheme $S$ and a point $x \in S(E)$.
Because $S$ is assumed to be excellent according to our conventions,
 $\M(x)$ is non-empty. Moreover, it is easy to check it is a right filtering
 category.
 Thus, with the notations of Lemma \ref{lm:pro_schemes_essentially},
 we have: $L\big(\pplimN_{X \in \M(x)} X\big)=x$

Therefore, for any spectrum $\E$ over $S$,
 one has according to Definition \ref{df:BM_essentially}:
$$
\hat \E^{BM}_{p,\gtw n}(x)
=\ilim_{X \in \M(x)^{op}} \E^{BM}_{p,\gtw n}(X/S).
$$
We will denote by $\pts(S)$ the class of points of $S$,
 seen as a discrete category.
\end{num}
\begin{df}\label{df:fiber_hlg}
Let $(S,\delta)$ be a dimensional scheme,
 $\E$ be a $\T$-spectrum over $S$ and $p$ be an integer.
Consider the preceding notations.
One defines the \emph{fiber $\delta$-homology of $\E$ in degree $p$}
 (or simply fiber homology)
 as the following functor:
$$
\rH^\delta_p(\E):\pts(S) \times  \ZZ \rightarrow \ab,
(x,n) \mapsto \hat \E^{BM}_{\delta(x)+p,\gtw{\delta(x)-n}}(x).
$$
One also defines the  \emph{effective fiber $\delta$-homology of $\E$
 in degree $p$} as the restriction of $\rH^\delta_p(\E)$ to
 the discrete category $\pts(S) \times \ZZ^-$,
 where $\ZZ^-$ is the set of non-positive integers.
 We will denote it by $\rH^{\delta-\eff}_p(\E)$.
\end{df}

Here are a few obvious facts about this definition:
\begin{lm}\label{lm:rost_hlg} Let $\E$ be
 a spectrum over $S$ and $p \in \ZZ$ be an integer.
\begin{enumerate}
\item One has the relation $\rH^\delta_p(\E[1])=\rH^\delta_{p-1}(\E)$.
\item The presheaf $\rH^\delta_p(\E)$ is covariantly functorial
 in $\E$.
\item The covariant functor $\rH^\delta_p$ commutes with coproducts.
 It also commutes with twists in the following sense:
 $\rH^\delta_p(\E\gtw 1)(x,n)=\rH^\delta_p(\E)(x,n+1)$.
\item Given any distinguished triangle
$
\E' \xrightarrow a \E \xrightarrow b \E'' \xrightarrow{+1}
$ in $\T(S)$, we deduce a long exact sequence of presheaf
 of $\ZZ$-graded abelian groups 
\begin{equation}\label{eq:long_exact_hlg}
\hdots \rightarrow H_p^\delta(\E') \xrightarrow{a_*}
 \rH^\delta_p(\E) \xrightarrow{b_*}
 \rH^\delta_p(\E'') \rightarrow
 \rH^\delta_{p-1}(\E') \rightarrow \hdots
\end{equation}
\end{enumerate}
Moreover, the same assertions hold when $\T$-spectra 
 are replaced by $\delta$-effective $\T$-spectra and $\rH^\delta_p$
 is replaced by $\rH^{\delta-\eff}_p$.
\end{lm}

\begin{rem}\label{rem:pullback&fiber_hlg}
Let $(S,\delta)$ be a dimensional scheme, 
 $f:T \rightarrow S$ be a separated morphism
 and $\E$ be a $\T$-spectrum  over $S$.

Then for any point $x \in S(E)$, corresponding to a morphism
 $x:\spec(E) \rightarrow S$, the following relation directly
 follows from the definition:
\begin{equation}\label{eq:exceptional_pullback&fiber_hlg}
\rH^\delta_p(f^!\E)(x,n)=\rH^\delta(\E)(x \circ f,n),
\end{equation}
where $\delta$ abusively denotes the dimension function $\delta^f$
 on $T$ induced by $\delta$ (see Paragraph \ref{num:induced_delta}).
 Suppose that $f$ is in addition smooth. Then it is relevant
 (see Paragraph \ref{num:pullback_change_dim}) to introduce
 a new dimension function on $T$ by the formula:
 $\tilde \delta^f=\delta^f-\dim(f)$.
\begin{equation}\label{eq:smooth_pullback&fiber_hlg}
\rH^{\tilde \delta^f}_p(f^*\E)(x,n)=\rH^\delta(\E)(x \circ f,n).
\end{equation}
This formula can also be extended to the following cases:
\begin{itemize}
\item $\T$ is continuous and $f$ is essentially smooth;
\item $\T$ is oriented, satisfies absolute purity
 and $f$ is a quasi-projective morphism between regular schemes
 (use Corollary \ref{cor:fund_class});
\item $\T$ is continuous, oriented, satisfies absolute purity
 and $f$ is an essentially quasi-projective morphism between
 regular schemes.
\end{itemize}
\end{rem}

\begin{num}
Let $S$ be a regular scheme and $x \in S(E)$ be a point.
Note that the corresponding morphism $x:\spec(E) \rightarrow S$
 is then a localization of a quasi-projective lci morphism,
 given by any $S$-model $X \rightarrow S$ of $x$
 (Definition \ref{df:Smodels}).

In particular, it admits a virtual tangent bundle
 $\tau_{x}$ in the category $\underline K(E)$ of virtual
 $E$-vector spaces (see \textsection \ref{num:virtual_vb}).
 This virtual bundle can be computed
 using the cotangent complex $L_{x}$ (cf. \cite{Ill})
 of the morphism $x$:
$$
\tau_{x}=\sum_i (-1)^i[H_i(L_{x})]=[\Omega^1_{x}]-[H_1(L_{x})].
$$
In particular, if $s$ denotes the image of $x$ in $S$,
 and $\kappa_s$ is the residue field of $s$ in $S$, we get:
$$
\tau_{x}=[\Omega^1_{E/\kappa_s}]
 -[\Gamma_{E/\kappa_s/F}]-[N_s \otimes_{\kappa_s} E]
$$
where $N_s=\mathcal M_s/\mathcal M_s^2$ is the normal bundle of $s$ in
 $\spec(\mathcal O_{S,s})$, and $\Gamma_{E/\kappa_s/F}$ is
 the imperfection module (\cite[\textsection 26]{Mat})
 of the extension $E/\kappa_s$ over the prime field $F$ contained
 in $\kappa_s$
 (thus the latter is trivial if the extension is separable).

Recall moreover from \cite{Del}
 that the category of virtual $E$-vector spaces is equivalent
 to the Picard category
 $\underline{Pic}(E)$ of graded line bundles over $E$ 
 through the determinant functor:
$$
\det:\underline K(E) \rightarrow \underline{Pic}(E),
 [V] \mapsto \big(\Lambda^{\max} V,\mathrm{rk}(E)\big)
$$
where $\Lambda^{\max}$ denotes the maximal exterior power
 and $\mathrm{rk}$ the virtual rank.
\end{num}
\begin{prop}\label{prop:comput_fiber_hlg}
Assume the following conditions hold:
\begin{enumerate}
\item[(a)] $\T$-cohomology commutes with projective limits
 (see \textsection \ref{num:commutes_with_limits});
\item[(b)] $\T$ satisfies absolute purity
 (see Definition \ref{df:abs_purity}).
\end{enumerate}
Let $S$ be a regular scheme
  with dimension function $\delta=-\codim_S$.
Then for any point $x \in S(E)$
 and any integers $(i,n) \in \ZZ$,
 there exists a canonical isomorphism
\begin{equation}\label{eq:compute_hatH_0}
\rH^\delta_i(\un_S)(x,n) \simeq H^{-i,\dtw{\tau_x-r},\gtw{n}}(E,\T)
\end{equation}
where $\tau_{x}$ is the virtual tangent bundle of
 the essential lci morphism $x:\spec{E} \rightarrow S$
 and $r$ is its virtual rank --
 or equivalently, $r=\delta(x)$ (see \eqref{eq:relative_dim}).

Moreover we can associate to any trivialization $\psi$ of 
 the 1-dimensional vector space $\det(\tau_{x})$ a
 canonical isomorphism:
$$
\rH^\delta_i(\un_S)(x,n) \simeq H^{-i,\gtw{n}}(E,\T).
$$
If moreover $\T$ is oriented, this isomorphism
 is independent of the choice of $\psi$.
\end{prop}
\begin{proof}
Let us choose an $S$-model $f:X \rightarrow S$ of the point $x$.
 As $X$ is affine, $f$ is quasi-projective.
 Because $X$ and $S$ are regular, $f$ is also lci.
 We denote by $\tau_{X/S}$ its virtual tangent bundle over $X$.
 Thus, from the absolute purity property of $\T$
 and Corollary \ref{cor:fund_class}, we get an isomorphism
 for any integers $p$, $s$:
\begin{align*}
H^{BM}_{p,\gtw{s}}(X/S)
 &=\Hom_{\T(X)}(f_!(\un_X)\gtw s[p],\un_S)
 \simeq \Hom_{\T(X)}(\un_X\gtw s[p],f^!\un_S) \\
 &\xrightarrow{\eta_f^{-1}} \Hom_{\T(X)}(\un_X\gtw s[p],\Th(\tau_{X/S}))
 =H^{i,\dtw{\tau_{X/S}},\gtw s}(X,\T).
\end{align*}
The class $\eta_f$ being compatible with restriction along an open
 immersion $j:U \rightarrow X$, we get the same isomorphism
 after reducing $X$ to $U$ and the corresponding isomorphism is compatible
 with the pullback map $j^*$. Therefore we get an isomorphism:
$$
\rH^\delta_i(\un_S)(x,n)
 =\hat H^{BM}_{r+i,\gtw{r-n}}(x/S)
 =\ilim_{U \subset X} H^{BM}_{r+i,\gtw{r-n}}(U/S)
 \simeq \ilim_{U \subset X} H^{-i,\dtw{\tau_{U/S}-r},\gtw{n}}(U,\T).
$$
Now, for any open immersion $j_{VU}:V \rightarrow U$ of open subschemes
 of $X$, we have $j_{VU}^*\tau_{U/S}=\tau_{V/S}$.
 Therefore assumption (a)
 gives the isomorphism \eqref{eq:compute_hatH_0}.

The remark about the trivialization then follows from
 the fact $\det$ is an equivalence of categories
 and the last assertion follows from Remark \ref{rem:Thom&orient}.
\end{proof}

\begin{ex} \label{ex:fiber_hlg&Kth}
Let $S$ be a regular scheme and $\delta=-\codim_S$.
\begin{enumerate}
\item Assume $\T=\DM_R$ under the hypothesis of
 points (1) or (2) of \S\ref{num:convention_DM}.

Then $\DM_R$ satisfies assumptions (a) and (b) of the previous
 proposition and is oriented.
 Therefore one gets for any point $x \in S(E)$ the following computation:
$$
\rH^\delta_{p}(\un_S)(x,n)=
\begin{cases}
0 & \text{if } p<0, \\
K_n^M(E) \otimes R & \text{if } p=0
\end{cases}
$$
where $K_*^M$ stands for the Milnor K-theory of the field $E$,
 or equivalently motivic cohomology of $E$ in degree $(n,n)$.
\item Assume $\T=\SH$ and $S$ is in addition an $F$-scheme
 for a prime field $F$.

Again, $\SH$ satisfies assumptions (a) and (b)
 of the preceding proposition. Thus one gets
 the following non-canonical isomorphism:
$$
\rH^\delta_{p}(\bS^0_S)(x,n)=
\begin{cases}
0 & \text{if } p<0, \\
K_n^{MW}(E) & \text{if } p=0.
\end{cases}
$$
where $K_n^{MW}(E)$ denotes the Milnor-Witt K-theory of 
 the field $E$
 (cf. \cite[Introduction, Def. 21]{Mor2}).
 The case $p<0$ follows from Morel's
 $\AA^1$-connectivity theorem (cf. \cite[Introduction, Th. 18]{Mor2})
 and the case $p=0$ is the computation of the $0$-th stable
 $\AA^1$-homotopy group of $\bS^0_S$
 (cf. \cite[Introduction, Cor. 24]{Mor2}).

To get a canonical identification, we must introduce
 a twisted version of Milnor-Witt K-theory as in \cite{Mor2},
 bottom of page 139. An element $\tau$ of
 $\underline K(E) \simeq \underline{\Pic(E)}$ can be 
 canonically represented by a 1-dimensional vector space
 $\Lambda$ -- its determinant. We put:
$$
K_n^{MW}(E,\tau)
 =K_n^{MW}(E) \otimes_{E^\times} \ZZ[\Lambda^\times]
$$
where $\Lambda^\times$ is the set of non zero elements of $\Lambda$,
 with its canonical action of $E^\times$, and the action
 of an element $u \in E^\times$ is induced by functoriality 
 through the multiplication map  by $u$. With this notation,
 one gets a canonical isomorphism:
\begin{equation}
\rH^\delta_{0}(\bS^0_S)(x,n) \simeq K_n^{MW}\big(E,\tau_x\big)
\end{equation}
where $\tau_x$ is the virtual tangent bundle of 
 the essentially lci morphism $x:\spec(E) \rightarrow S$.
\end{enumerate}
\end{ex}

\begin{num}\label{num:ssp_niveau&hlg}
It is useful to rewrite the $\delta$-niveau spectral sequence
 using fiber homology.
 Let $\E$ be a spectrum, $(S,\delta)$ a dimensional scheme
 and $X$ be an essentially separated $S$-scheme.
 Then the $\delta$-niveau spectral sequence of $X/S$
  with coefficients in $\E$,
	defined in Proposition \ref{prop:df_niveau_ssp_essntially},
  has the following form:
\begin{equation}\label{eq:niveau_ssp_homology0}
^\delta E^1_{p,q}=\bigoplus_{x \in X_{(p)}} \rH^\delta_q(\E)(x,p-n)
 \Rightarrow \E^{BM}_{p+q,\gtw n}(X/S).
\end{equation}
Note that the spectral sequence can actually be written as
\begin{equation}\label{eq:niveau_ssp_homology}
^\delta E^1_{p,q}=\bigoplus_{x \in X_{(p)}} \rH^\delta_q(\E)(x)
 \Rightarrow \E^{BM}_{p+q,\gtw *}(X/S),
\end{equation}
where it is assumed to take values in graded abelian groups;
 then the differentials in the $E^1$-term are homogeneous
 of degree $(-1)$.\footnote{We are following the convention
  and homological notations of Rost in \cite{Ros}.}

Assume moreover that $\E$ is $\delta$-effective.
 Considering the first form of the spectral sequence $^\delta E^1_{p,q}$,
 we see that when $n\geq \delta(X)$, then $p-n \leq 0$ or
 $X_{(p)}=\varnothing$. Moreover the abutment of the spectral
 sequence can be computed as a group morphisms of the category
 $\T^{\delta-\eff}$.
 Thus the effective version of the preceding spectral
 sequence as the form:
\begin{equation}\label{eq:niveau_ssp_homology_eff}
^\delta E^1_{p,q}=\bigoplus_{x \in X_{(p)}} \rH^{\delta-\eff}_q(\E)(x)
 \Rightarrow \E^{BM}_{p+q,\gtw *}(X/S),
\end{equation}
where we have restricted the gradings on the abutment to $*\geq \delta(X)$.  So it takes its values in $\NN$-graded abelian groups.
\end{num}

Recall that, because we work with noetherian finite dimensional
 schemes, the preceding spectral sequences are all convergent.
 Thus, a corollary of their existence is the following result.
\begin{prop}\label{prop:hlg_conservative}
For any dimensional scheme $(S,\delta)$,
 the family of functors:
\begin{align*}
\ & \rH^\delta_p:\T(S) \rightarrow \mathrm{PSh}\big(\pts(S) \times \ZZ\big),
 p \in \ZZ, \\
\text{resp. }
& \rH^{\delta-\eff}_p:\T^{\delta-\eff}(S)
 \rightarrow \mathrm{PSh}\big(\pts(S) \times \ZZ^-\big),
 p \in \ZZ,
\end{align*}
is conservative.
\end{prop}

The following vanishing conditions will be a key point
 for our main theorem \ref{thm:hlg&htp_t}.
\begin{prop}\label{prop:hom_compatible}
The following conditions are equivalent:
\begin{enumerate}
\item[(i)]
 For any regular dimensional scheme $(S,\delta)$,
  and any $i<-\delta(S)$, 
  $\rH^\delta_i(\un_S)=0$.
\item[(i')] For any regular scheme $S$,
 with dimension function $\delta=-\codim_S$,
 and any $i<0$, $\rH^\delta_i(\un_S)=0$.
\item[(ii)] For any regular dimensional scheme $(S,\delta)$,
  any essentially separated $S$-scheme $X$,
  and any $p<\delta_-(X)-\delta(S)$
	(see Remark \ref{rem:bound_dim_not} for the notation)
  one has $\hat H ^{BM}_{p,\gtw *}(X/S)=0$.
\end{enumerate}
\end{prop}
\begin{proof}
 (i) is obviously equivalent to (i'),
since any dimension function $\delta$ on a connected regular scheme
 satisfies the relation $\delta(x)=\delta(S)-\codim_S(x)$ for $x \in S$.

The fact (ii) implies (i) follows directly
 from the definition of $\rH^\delta_i$.
Let us prove the converse implication. In the situation of (ii),
 we look at the $\delta$-niveau spectral sequence
 \eqref{eq:niveau_ssp_homology} in the case $\E=\un_S$.
 Using assumption (i), we get
 that $^\delta E^1_{p,q}=0$ whenever $q<-\delta(S)$ or $p<\delta_-(X)$.
 In particular, $^\delta E^1_{p,q}$ is zero if $p+q<\delta_-(X)-\delta(S)$
 and this concludes.
\end{proof}

The equivalent properties of the previous proposition will be crucial
 for the main theorem of this section so that we introduce the
 following definition.
\begin{df}\label{df:htp_compatible}
We say that the motivic category $\T$ is \emph{homotopically compatible}
 if the equivalent conditions of the preceding proposition
 are satisfied.
\end{df}

As an immediate corollary of Proposition \ref{prop:comput_fiber_hlg},
 we get the following useful criterion for this property.
\begin{prop}
Assume the following conditions hold:
\begin{enumerate}
\item[(a)] $\T$-cohomology commutes with projective limits
 (see \textsection \ref{num:commutes_with_limits});
\item[(b)] $\T$ satisfies absolute purity 
 (see Definition \ref{df:abs_purity});
\item[(c)] For any field $E$ such that $\spec(E)$ is in our category of schemes
 $\base$, $H^{n,m}(\spec(E),\T)=0$ if $n>m$.
\end{enumerate}
Then $\T$ is homotopically compatible.
\end{prop}

\begin{ex} Here are our main examples of homotopically compatible
 categories:
\begin{enumerate}
\item Assume $\base$ is included in the category of
 schemes over a prime field $F$. Then all the triangulated motivic
 categories of Example \ref{ex:motivic_cat}, except possibly
 that of point (5) if the characteristic of $F$ is positive,
 satisfies the assumptions of the previous corollary:
\begin{itemize}
\item assumption (a): see \cite[4.3.3]{CD3} (the case of $\SH$
 is treated as the case of $\DA$, and the case of modules then directly
 follows);
\item assumption (b) is then automatic (as recalled in
 Example \ref{ex:abs_pur_equal});
\item assumption (c): in the case of $\SH$,
 $\MGLmod$ and $\DAx R$, is a consequence of Morel's stable
 $\AA^1$-connectivity theorem \cite[Th 3]{dmtilde3};
 in the case of $\DM_R$, under the conventions of Paragraph
 \ref{num:convention_DM} (so that the characteristic exponent
 of $F$ is invertible in $R$), this follows from the basic property 
 of motivic cohomology of perfect fields and the semi-separation
 property  of $\DM_R$ (see Lemma \ref{lm:inseparable_texact}).
\end{itemize}
\item Assume $\base$ is any category of schemes satisfying
 our conventions. Then for any $\QQ$-algebra $R$,
 $\DM_R$ is homotopically compatible.
\end{enumerate}
\end{ex}

\subsection{Main theorem}\label{sec:main_thm}

\begin{thm}\label{thm:hlg&htp_t}
Assume the following conditions hold:
\begin{enumerate}
\item[(a)] $\T$ is homotopically compatible
 (see Definition \ref{df:htp_compatible}).
\item[(b)] Assumption (Resol) of \textsection \ref{num:resolution}
 is satisfied.
\end{enumerate}
Then for any dimensional scheme $(S,\delta)$ and
 any $\T$-spectrum (resp. $\delta$-effective $\T$-spectrum) $\E$ over $S$,
 the following conditions are equivalent:
\begin{enumerate}
\item[(i)] $\E$ is $t_\delta$-positive
 (resp. $t_\delta$-negative).
\item[(ii)] For any integer $p \leq 0$
 (resp. $p\geq 0$), $\rH_p^\delta(\E)=0$
  (Definition \ref{df:fiber_hlg}).
\end{enumerate}
Moreover, when $\E$ is $\delta$-effective,
 these conditions remain equivalent after
 replacing $t_\delta$ by $t_\delta^\eff$
 and $\rH_p^\delta$ by $\rH_p^{\delta-\eff}$.
\end{thm}
\begin{proof}
The proof is valid in the general and $\delta$-effective case
 except for a change of indexes.
 We use unified notations to treat both cases in a row:
\begin{itemize}
\item in the general case, we put $I=\ZZ$, $t=t_\delta$,
 $\C=\T(S)$, $\rH_p=\rH_p^\delta$;
\item in the effective case, 
 we put $I=\NN$, $t=t_\delta$, $\C=\T^{\delta-\eff}(S)$,
 $\rH_p=\rH_p^{\delta-\eff}$. Below, this case will also be referred
 to as the resp. case.
\end{itemize}

Let us write $\C_{\geq0}$ (resp. $\C_{<0}$)
 for the subcategory of non-$t$-negative
 (resp. $t$-negative) spectra in $\C$
 and $\C_{H\geq 0}$ (resp. $\C_{H<0}$)
 for the subcategory of $\C$ made of spectra $\E$ such that
 $\rH_p(\E)=0$ if $p<0$ (resp. $p\geq 0$).

Thus we have to prove $\C_{<0}=\C_{H<0}$ and $\C_{\geq 0}=\C_{H\geq0}$.

The fact $\C_{<0} \subset \C_{H<0}$ follows from definitions
 -- see in particular Remark \ref{rem:htp_t_nonnegdeg}
 (resp. Remark \ref{rem:htp_t_nonnegdeg_eff}).
Let us prove the converse inclusion.
 Take a spectrum $\E$ in $\C_{H<0}$. 
 According to the remark previously cited,
 we have to prove that for any separated $S$-scheme $X$,
 $\E^{BM}_{\delta(X)+p,\delta(X)-n}(X/S)=0$ if $p>0$ and $n \in I$.
 Let us rewrite the spectral sequence \eqref{eq:niveau_ssp_homology}
 (resp. \eqref{eq:niveau_ssp_homology_eff}) for $\E$ and $X/S$
 in our notations, and for the grading $*=\delta(X)-n$,
 $n \in I$:
$$
^\delta E^1_{p,q}=\bigoplus_{x \in X_{(p)}} \rH_q(\E)(x,p-\delta(X)+n)
 \Rightarrow \E^{BM}_{p+q,\gtw{\delta(X)-n}}(X/S).
$$
The $E^1_{p,q}$-term is concentrated in the region $p\leq \delta(X)$
 by construction and in the region $q < 0$ by assumption on $\E$.
 In particular, $^\delta E^1_{p,q}=0$ if $p+q>\delta(X)$
 and this concludes.

We now prove that $\C_{\geq 0} \subset \C_{H\geq 0}$.
 According to Lemma \ref{lm:rost_hlg},
 the subcategory $\C_{H\geq 0}$ of $\C$ is stable
 by positive suspensions, coproducts and extensions.
 Thus it is sufficient to prove that the generators
 of the $t$-structure $t$ belongs to $\C_{H\geq 0}$.
 According to assumption (b), we can apply
 Theorem \ref{thm:strong_generators} to $\T$.
 Thus, we have to show that for a proper regular $S$-scheme
 $Y$ and an integer $m \in I$,
 $\Mb(Y/S)\dtw{\delta(Y)}\gtw m$ belongs to $\C_{H\geq 0}$.
 Given a point $x$ of $S$, we show that
 the abelian group
 $H_i^\delta\big(\Mb(Y/S)\dtw{\delta(Y)}\gtw m\big)(x,n)$
 is zero for $i<0$ and $n \in I$.
 Note that, as $Y/S$ is proper,
 we get the following computation for any model $X \in \M(x)$:
\begin{align*}
\Hom_{\T(S)}\big(\Mb(X/S)\gtw s[j],\Mb(Y/S)\big)
 &=\Hom_{\T(Y)}\big(\Mb(X \times_S Y/Y)\gtw s[j],\un_S\big) \\
 &=\E^{BM}_{j,\gtw s}(X \times_S Y/Y).
\end{align*}
Thus, with $j=\delta(X)-\delta(Y)+i$ and $s=\delta(X)-\delta(Y)+n-m$,
 we deduce:
$$
H_i^\delta\big(\Mb(Y/S)\dtw{\delta(Y)}\gtw m\big)(x,n)
=\hat H^{BM}_{\delta(x)-\delta(Y)+i,\gtw{\delta(x)-\delta(Y)+n-m}}(Y_x/Y)
$$
where $Y_x$ is the fiber of $Y$ at the point $x$ of $S$ --
 which is in fact essentially separated over $Y$.
 According to assumption (a),
 we can apply assertion (ii) of Proposition \ref{prop:hom_compatible}.
 This concludes the proof because
 $\delta(x)-\delta(Y)+i<\delta_-(Y_x)-\delta(Y)$ whenever $i<0$.

Let us finally prove that $\C_{H\geq 0} \subset \C_{\geq 0}$.
 Let $\E$ be an object in $\C_{H\geq 0}$.
 We can consider the distinguished triangles associated
 with $E$ and the $t$-structure $t$:
$$
E_{\geq 0} \rightarrow \E \rightarrow E_{<0} \xrightarrow{+1}
$$
Applying the functor $H_p$ and Lemma \ref{lm:rost_hlg}(4),
 we get a sequence of presheaves of abelian groups on $\pts(S) \times \ZZ$
 (resp. $\pts(S) \times \ZZ^-$):
$$
\hdots \rightarrow \rH_p(\E_{\geq 0}) \rightarrow
 \rH_p(\E) \rightarrow
 \rH_p(\E_{<0}) \rightarrow
 \rH_{p-1}(\E_{\geq 0}) \rightarrow \hdots
$$
According to the inclusion $\C_{<0} \subset \C_{H<0}$
 already proved,
 we get that $\rH_p(\E_{<0})=0$ if $p\geq 0$.
 Let $p<0$. By assumption, $\rH_p(\E)=0$. According 
to the inclusion $\C_{\geq 0} \subset \C_{H\geq 0}$
 proved just above, we also have $\rH_{p-1}(\E_{\geq 0})=0$.
 Thus the preceding long exact sequence implies that
 $\rH_p(\E_{<0})=0$.
 According to Proposition \ref{prop:hlg_conservative},
 we deduce that $E_{<0}=0$ and this concludes. 
\end{proof}

This theorem has many nice consequences. 
 Let us start with computations.
\begin{ex}\label{ex:compute_hlg_constant1}
Assume $\T=\DM_\QQ$ following conventions of
 \ref{num:convention_DM}(1).

Let $S=\spec(\QQ)$ equipped with
 the Krull dimension function $\delta$.
 By definition, $\un_\QQ$ is $t_\delta$-non-negative.
 Moreover, we can see it is {\bf not} positively bounded
 for the $t_\delta$-homotopy $t$-structure. Actually,
 for any integer $n\geq 0$,
\begin{equation}\label{eq:constant_unbounded_DM}
H_n^\delta(\un_\QQ) \neq 0.
\end{equation}
Indeed, let $K/\QQ$ be the field extension generated
 by the group of primitive $d$-th roots of unity $\mu_d^0$.
 Note that according
 to our conventions, $K$ is a point of $\spec(\QQ)$.
 According to Proposition \ref{prop:comput_fiber_hlg}, we get
 for any integer $n\geq0$:
$$
\rH^\delta_{n}(\un_\QQ)(K,n+1) \simeq \HB^{1,n+1}(K).
$$
But this group is non-zero, since Beilinson's
 construction of polylogarithms yields a canonical
 map
$$
\epsilon_{n+1}:\mu_d^0 \rightarrow \HB^{1,n+1}(K)
$$
whose composition with the regulator map is the classical
 polylogarithm which is non-zero. Here
 we refer the reader to \cite[Cor. 9.6]{HW}.
 Then relation \eqref{eq:constant_unbounded_DM}
  follows from the preceding theorem.

According to \cite{Deg5}, given any extension field $L/\QQ$
 the pullback map $\HB^{i,n+1}(\QQ) \rightarrow \HB^{i,n+1}(L)$
 is a (split) monomorphism. Thus relation \eqref{eq:constant_unbounded_DM}
 is true if one replaces $\QQ$ by any characteristic $0$ field,
 or even any regular $\QQ$-scheme $S$ with dimension function
 $\delta=-\codim_S$
 (given again the computation of Proposition \ref{prop:comput_fiber_hlg}).

Note in particular that the $\delta$-homology of a constructible
 motive will not be positively bounded in general. This makes it different
 from its cousin, the perverse $t$-structure on torsion or $\ell$-adic
 \'etale sheaves.
\end{ex}

\begin{ex}\label{ex:unit_eff_heart}
Assume $\T=\DM_R$ following conventions of
 \ref{num:convention_DM}(1) or (2).
Let $S$ be a regular scheme with dimension function $\delta=-\codim_S$.

We consider the $\delta$-homotopy $t$-structure on
 $\DM^{\delta-\eff}(S,R)$. This means in particular that we
 restrict to the $\ZZ^-$-graded part of fiber $\delta$-homology.
 Applying Proposition \ref{prop:comput_fiber_hlg},
 for any couple of integer $(i,n) \in \ZZ \times \ZZ^-$, we get:
$$
\rH^{\delta-\eff}_i(\un_S)(x,n) \simeq H^{-i,\gtw{n}}_M(E,R)=
\begin{cases}
R & \text{if }i=0, n=0 \\
0 & \text{otherwise.}
\end{cases}
$$
In other words, $\rH^{\delta-\eff}_0(\un_S)=R$ as a graded abelian group
 concentrated in degree $0$,
 and if $i \neq 0$, $\rH^\delta_i(\un_S)=0$ .

Thus, the constant motive $\un_S$ is concentrated in degree
 $0$ for the effective $\delta$-homotopy $t$-structure
 -- the case of a (perfect) field was already well known thanks to
 Example \ref{ex:delta_effective_fields}.

This indicates that the $\delta$-homotopy $t$-structure
 is better behaved (with respect to bounds) on $\delta$-effective
 motives over nonsingular schemes. Actually, we expect
 that the $\delta$-homology of any constructible
 (\emph{i.e.} compact) $\delta$-effective motive over $S$ is bounded
 (see in particular Prop. \ref{prop:compute_hlg_curves})
 but this is a deep conjecture. Indeed, already in the case
 where $S$ is the spectrum of a perfect field $k$,
 the Suslin complex $C_*^{sing}\left(\ZZ^{tr}(X)\right)$ of a smooth $k$-scheme $X$
 is not known to be bounded though it is believed its homology sheaves
 vanish in degree greater than $2\dim(X)$ --- see in particular a
 conjecture of Morel \cite[Conjecture 11]{dmtilde3} without transfers.\footnote{Beware
 however the conjecture of Morel does not imply the analogous conjecture for sheaves
 with transfers as the functor ``adding transfers'' $\gamma^*$  is not
 right exact for the homotopy $t$-structure.}
\end{ex}

\begin{rem}\label{rem:compute_hlg_constant2}
The preceding example cannot be generalized to the singular case.
 Let us consider its notations and assumptions.

Let us take a field $k$, $\delta$ being the Krull dimension
 function on $\spec(k)$.
 We look at the example where $S$ is an algebraic $k$-scheme
 which is the union two copies of the affine line
 $D_1$ and $D_2$ crossing in a single rational point $s$.
 Consider the canonical closed immersions:
$$
\xymatrix@=16pt@C=24pt{
s\ar^{k_1}[r]\ar_{k_2}[d]\ar|k[rd] & D_1\ar^{i_1}[d] \\
D_2\ar|{i_2}[r] & S.
}
$$
Then, using cdh-descent for $\DM(S,R)$, we get a distinguished triangle
 (apply \cite[3.3.10]{CD3} to the preceding cartesian square):
$$
k_*(\un_s)[-1] \rightarrow \un_S
 \rightarrow i_{1*}(\un_{D_1}) \oplus i_{2*}(\un_{D_2})
 \rightarrow k_*(\un_s)
$$
where the first map is a boundary, while the other two are obtained
 by considering the relevant adjunctions.
Applying $k^!$ to this triangle,
 together with the base change formula and the absolute purity
 (Corollary \ref{cor:fund_class})
 of $\DM_R$ with respect to $k_1$ and $k_2$, we get the following 
 distinguished triangle:
$$
\un_s[-1] \rightarrow  k^!(\un_S)
 \rightarrow \un_s\dtw{-1} \oplus \un_s\dtw{-1}
 \xrightarrow{(1)} \un_s.
$$
The map labelled $(1)$ corresponds to a cohomology class in
 $H^{2,1}_M(s,R) \oplus H^{2,1}_M(s,R)$ which vanishes since $x$
 is a point. Therefore, the above distinguished triangle splits
 and gives:
$$
k^!(\un_S)=\un_s[-1] \oplus \un_s\gtw{-1}[-1] \oplus \un_s\gtw{-1}[-1].
$$
Taking into account formula \eqref{eq:exceptional_pullback&fiber_hlg},
 and Lemma \ref{lm:rost_hlg}(1)(3), we obtain for any couple of integers
 $(i,n) \in \ZZ \times \ZZ^-$:
$$
H_i^{\delta-\eff}(\un_S)(s,n)=\begin{cases}
\ZZ & \text{if } (i,n)=(-1,0), \\
0 & \text{otherwise.}
\end{cases}
$$
The computation of the fiber $\delta$-homology of the other points,
 the closed nonsingular points $S'_{(0)}$ and the two generic points
 $S^{(0)}$,
 follows from the previous example, as it can be reduced to the
 case of a regular base. In the end we obtain the following computation:
$$
H_i^{\delta-\eff}(\un_S)(x,n)=\begin{cases}
\ZZ & \text{if } \big(x=s \text{ and } (i,n)=(-1,0)\big)
 \text{ or } \big(x \in S^{(0)} \text{ and } (i,n)=(0,0)\big) \\
& \qquad\qquad
 \text{ or } \big(x \in S'_{(0)} \text{ and } (i,n)=(1,-1)\big), \\
\kappa(x)^\times & \text{if } x \in S'_{(0)} \text{ and } (i,n)=(1,0), \\
0 & \text{otherwise.}
\end{cases}
$$
Thus the motive $\un_S$ has $t_\delta^\eff$-amplitude $[-1,1]$
 exactly.

However, it is easy using a stratification by regular locus
 to show that for
 any noetherian excellent scheme, $\un_S$
 has finite $t_\delta^\eff$-homological
 amplitude. In particular, it is reasonable to expect that the
 conjectural boundedness stated in the end of the preceding example
 happens also in the singular case.
\end{rem}

Let us go back to corollaries of the preceding theorem.
\begin{cor}\label{cor:texact}
Suppose that the hypothesis of the preceding theorem
 are fulfilled and let $(S,\delta)$ be a dimensional scheme.
 Then for any $\T$-spectrum $\E$ over $S$,
 we get the following equivalent conditions where the
 left hand side conditions refer to the $\delta$-homotopy
 $t$-structure:
\begin{enumerate}
\item $\E \geq 0$ if and only if for all separated $S$-scheme $X$,
 $\E^{BM}_{p,\gtw *}(X/S)=0$ when $p<\delta_-(X)$.
\item $\E \leq 0$ if and only if for all separated $S$-scheme $X$,
 $\E^{BM}_{p,\gtw *}(X/S)=0$ when $p>\delta_+(X)$.
\end{enumerate}
Moreover, these equivalent conditions remain true
 when one replaces separated $S$-schemes by essentially separated
 $S$-schemes.
\end{cor}
\begin{proof}
This is a straightforward consequence of the preceding theorem together
 with the $\delta$-niveau spectral sequence under the
 form \eqref{eq:niveau_ssp_homology}.
\end{proof}

In the $\delta$-effective case, we need to be more precise about
 the $\GG$-gradings.
\begin{cor}
Suppose that the hypothesis of the preceding theorem
 are fulfilled and let $(S,\delta)$ be a dimensional scheme.
 Then for any $\delta$-effective $\T$-spectrum $\E$ over $S$,
 we get the following equivalent conditions where the
 left hand side conditions refer to the $\delta$-homotopy
 $t$-structure on $\T^{\delta-\eff}(S)$:
\begin{enumerate}
\item $\E \geq 0$ $\Leftrightarrow$
 for all separated $S$-scheme $X$, 
 $\E^{BM}_{p,\gtw n}(X/S)=0$ when $p<\delta_-(X)$ and $n>\delta_+(X)$.
\item $\E \leq 0$ $\Leftrightarrow$ for all separated $S$-scheme $X$,
 $\E^{BM}_{p,\gtw n}(X/S)=0$
 when $p>\delta_+(X)$ and $n<\delta_-(X)$.
\end{enumerate}
Moreover, these equivalent conditions remain true
 when one replaces separated $S$-schemes by essentially separated
 $S$-schemes.
\end{cor}
Again, use the preceding theorem together with the $\delta$-niveau
 spectral sequence but under the form \eqref{eq:niveau_ssp_homology0}.

\begin{cor}\label{cor:nice_ppties_dhtp}
Suppose that the hypothesis of the preceding theorem
 are fulfilled and let $(S,\delta)$ be a dimensional scheme.
\begin{enumerate}
\item The $t$-structure $t_\delta$ (resp. $t_\delta^\eff$)
 is non-degenerate.
\item The functor $w:\T(S) \rightarrow \T^{\delta-\eff}(S)$
 of Proposition \ref{prop:functors&eff}(1)
 is $t_\delta$-exact.
\item The composite functor
 $(-_{-1}):\T^{\delta-\eff}(S)
 \stackrel{\{-1\}}{ \rightarrow} \T^{\delta-\eff}(S)(-1)
 \stackrel{w'}{ \rightarrow} \T^{\delta-\eff}(S) $ (see the aforementioned proposition)
  is $t_\delta^\eff$-exact.
\item For any separated morphism $f:T \rightarrow S$,
 the functor $f^!$ is $t_\delta$-exact.
\item For any morphism $f:X \rightarrow S$
 such that $\dim(f) \leq d$,
 the functor $f_*$ has $t_\delta$-amplitude $[0,d]$.
\end{enumerate}
\end{cor}
\begin{proof}
Assertion (1) follows from the previous theorem combined
 with Proposition \ref{prop:hlg_conservative}.

Assertion (2) follows from the previous theorem
 given that for any $\T$-spectrum $\E$ over $S$ an easy computation gives:
$$
\rH^{\delta-\eff}_i\big(w(\E)\big)
 =\rH^{\delta}_i\big(\E\big)|_{\pts(S) \times \ZZ^-}.
$$
Similarly, for any effective $\E$ over $S$ if we denote
 (essentially following \cite[\S3.3.1]{FSV}) $-_{-1}(\E)$ by $\E_{-1}$,
 then we have 
$$(\E_{-1})^{BM}_{\delta(x)+p,\dtw{\delta(x)-n}}(x/S)
 = E^{BM}_{\delta(x)+p,\dtw{\delta(x)-n+1}}(x/S)
$$
for any $(x,n)\in \pts(S) \times \ZZ^-$. This easily yields assertion (3).

Assertion (4) obviously follows from the previous theorem 
 along with relation \eqref{eq:exceptional_pullback&fiber_hlg}.

Let us prove assertion (5). 
 In view of Proposition \ref{prop:basic_tstruct}(3),
 one needs only to prove that $f_*$ preserves
 $t_\delta$-positive spectra.
 Let $\E$ be such a spectrum over $X$.
 A direct computation gives, for any point $s \in S(E)$:
$$
\rH^\delta_i(f_*\E)(s,*)=\hat \E^{BM}_{\delta(s)+i,\gtw{\delta(s)-*}}(X_s/X)
$$
where $X_s$ denotes the fiber of $f$ at the point $s$ -- \emph{i.e.}
 the pullback of the $S$-scheme $X$ along the morphism $s$ --
 seen as an essentially separated scheme over $X$. Then the required 
 vanishing easily follows from the preceding corollary.
\end{proof}


\begin{num}\label{num:i^!_essentially}
To state the next result, we need the following definition taken
 from \cite[2.2.12]{BBD}. Let $S$ be a scheme and $x \in S$
 a set-theoretic point of $S$, with inclusion $i_x$.
 Then $i_x$ can be factorized as $x \xrightarrow j \bar x \xrightarrow i S$,
 where $\bar x$ is the reduced closure of $x$ in $S$
 and we put:
$$
i_x^!=j^*i^!.
$$
\end{num}
\begin{cor}\label{cor:restriction_fields}
Suppose that the assumptions of the previous theorem are fulfilled
 and that the motivic triangulated category $\T$ is continuous
 (see Paragraph \ref{num:commutes_with_limits}). Let $(S,\delta)$ be a
 dimensional scheme.
\begin{enumerate}
\item For any set-theoretic point $x \in S$, the functor
 $i_x^!$ defined above is $t_\delta$-exact.
 Moreover, for any $\T$-spectrum $\E$ over $S$
  and any finitely generated field extension $K$ of the residue field of $x$,
  with corresponding point $x_K \in S(K)$, one has:
\begin{equation}\label{eq::restriction_fields}
\rH^\delta_*\big(i_x^!\E\big)(K,*)=\rH^\delta_*(\E)(x_K,*).
\end{equation}
\item The family of $t_\delta$-exact functors $i_x^!:\T(X) \rightarrow \T(x)$
 indexed by set-theoretic points $x \in S$ respects and detects $t_\delta$-positive
 and $t_\delta$-negative spectra.
\end{enumerate}
\end{cor}
\begin{proof}
Given the previous theorem, the only thing to prove is relation
 \eqref{eq::restriction_fields}.
 Obviously, we can restrict to $\rH^\delta_0(-,0)$.
 By definition of $i_x^!$ and point (4) of the previous corollary,
 we can assume
 $S$ is irreducible and $x$ is the generic point of $S$. Let us denote by
 $k$ the residue field of $x$.
 With the notations of the paragraph preceding the proof,
 we thus have $i_x^!=j^*$.
 We want to prove there exists an isomorphism of the form:
$$
\psi:\Hom_{\T(x)}(\hMb(K/k),j^*\E) \simeq \Hom_ {\T(S)}(\hMb(x_K),\E).
$$
Given any $S$-model $X$ of $x_K$ (cf. Definition \ref{df:Smodels}),
 it is clear that $X_k:=X \times_S k$ is a $k$-model of $K$.
 Moreover, it is clear that the corresponding functor
 $\M(x_K) \rightarrow \M(K/k), X \mapsto X_k$ is final. Therefore,
 one gets the following computation of the left hand side of $\psi$:
$$
\Hom_{\T(x)}(\hMb(K/k),h^*\E)
 =\ilim_{X \in \M(x_K)} \Hom_{\T(x)}(\Mb(X_k/k),h^*\E).
$$
On the other hand, given any $S$-model $X$ of $x_K$, we get by using the
 continuity property of $\T$ and \cite[Prop. 4.3.4]{CD3}:
\begin{align*}
\Hom_{\T(x)}(\Mb(X_k/k),j^*\E)
 &=\ilim_{U \subset S} \Hom_{\T(x)}(\Mb(X_U/U),j_U^*\E) \\
 &=\ilim_{U \subset S} \Hom_{\T(x)}(\Mb(X_U/S),\E)
\end{align*}
where $U$ runs over the non-empty open subschemes of $S$,
 $X_U=X \times_S U$ and $j_U:U \rightarrow S$ is the obvious open immersion.

The result follows from this last computation, and from
 the definition of $\hMb(x_K)$
 (in particular Lemma \ref{lm:pro_schemes_essentially}).
\end{proof}

Let us finish with the following easy but useful corollary.
\begin{cor}\label{cor:generic_motives}
Let us assume that the hypothesis of the preceding theorem
 are fulfilled. Let $(S,\delta)$ be a dimensional scheme.

Then for any $\T$-spectrum $\E$ over $S$, and any point $x$ of $S$,
 the truncation functors for the $\delta$-homotopy $t$-structure
  induce a canonical isomorphism:
$$
\Hom_{\T(S)}(\hMb(x)\dtw{\delta(x)},\E)
 \simeq \Hom_{\T(S)}(\hMb(x)\dtw{\delta(x)},H_0(\E)).
$$
\end{cor}
This is another way of saying that
 $\rH_0^\delta(\E) \simeq \rH_0^\delta\big(H_0(\E)\big)$
 which follows immediately from the preceding theorem.

\subsection{On Gersten-type weight structures}\label{sgws}
In this section we only indicate the relation of homotopy $t$-structures to the so-called Gersten weight structures (avoiding giving precise definitions).

So, recall that Gersten weight structures were constructed in \cite{Bon10a} for motives over a countable perfect field, and in 
\cite{Bon15} for "arbitrary" motivic categories over an arbitrary perfect field (see \S6 of ibid. and note 
 that in these two papers "opposite conventions on signs on weights" were used; see Remark 2.1.2(3) of ibid.). Now, weight structures are (very roughly) "orthogonals"  to $t$-structures, 
and so the main property of the aforementioned  Gersten weight structures is that they are orthogonal to the corresponding homotopy $t$-structures (see Definition 2.4.1, Proposition 4.4.1, and Theorem 6.1.2(II.6) of ibid., noting that in ibid.  the "cohomological convention" for $t$-structures was used). 

Another (more or less, equivalent) form of the main property of  Gersten weight structures (both for the ones from \cite{Bon15} and for the "relative" ones that we will discuss here) is that their hearts should be generated by 
 $\Mb(x/S)\gtw n[\delta(x)]$ for $x$ running through the (spectra of) fields over the corresponding base scheme $S$. The problem is that we  don't have 
any "reasonable" $\Mb(x/S)$ inside $\C$ 
whenever $x$ is not of finite type over $S$; so that one has to consider a certain triangulated "homotopy completion" of $\C$. Thus one has to treat some model for  $\C$, and construct a certain triangulated category $\C'$ of \emph{pro-spectra} using this model. Next,  there is a functor $\C'\to \operatorname{Pro}-\C$, but it is 
 far from being a full embedding; so one has to overcome certain difficulties when computing $\C'$-morphism groups. Note here that  one requires 
 the morphism group $\Hom_{\C'}(\Mb(x/S)\gtw n[\delta(x)-\delta(x')+i], \Mb(x'/S)
 )$ to be zero for any $x,x'\in \pts(S)$ and $i< 0$ (this is a certain Gersten weight structure substitute of 
the vanishing of $H_i^\delta\big(\Mb(Y/S)\dtw{\delta(Y)}\gtw m\big)(x,n)$
 whenever 
 $n \in I$,  $x\in \pts(S)$, and $Y$ is a proper regular $S$-scheme, cf. the proof of Theorem \ref{thm:hlg&htp_t}).

There currently exist two distinct methods for proving this $\C'$-orthogonality result. The first one (applied in \cite{Bon10a}) uses the properties  
of countable homotopy limits. It can be carried over to the "relative context" more or less easily; so one obtains the existence and  basic properties 
of Gersten weight structures whenever the corresponding $S$ possesses a Zariski cover by spectra of (at most) countable rings. Yet one can probably get rid of this restriction by using the methods of \cite{Bon15} instead.

The existence of such a weight structure would allow to define a certain "generalized ($\delta$-co)niveau spectral sequence" for the $H$-cohomology of any constructible object of $\C$ (where $H$ is a cohomological functor from $\C$ into an AB5-abelian category) and to prove the $\C$-functoriality of these spectral sequences (as well as the corresponding filtrations) starting from $E^2$ (cf. \cite[\S4.3]{Bon15}). Besides, if $H$ is $\C$-representable then this spectral sequence (for the $H$-cohomology of any constructible $c$) is naturally isomorphic to the one coming from the $t_{\delta}$-truncations of $c$ (cf. Corollary 4.4.3 of ibid.).

An interesting problem related to the "relative" Gersten weight structures is the construction of connecting functors between the corresponding 
$\T'(X)$ for $X$ running through $\base$. The $t_{\delta}$-exactness statements of Corollary \ref{cor:texact} (parts 4 and 5) 
should correspond to the "dual"  Gersten-weight-exactness properties of the "pro-spectral versions" of $f^*$ and $f_!$.



\section{On $\delta$-homotopy hearts}

\begin{num}\label{num:good}
Theorem \ref{thm:hlg&htp_t} is the central result in the applications
 of the $\delta$-homotopy $t$-structure.
 So we will isolate its assumptions
 and call $(good)$ the conjunction of
 the assumptions (a) (vanishing of certain Borel-Moore homology groups)
 and (b) (suitable resolution of singularities).
\end{num}

\subsection{Definition and functoriality properties}

We will adopt the following notations.
\begin{df}\label{df:heart}
Let $\T$ be a motivic triangulated category
 satisfying our general conventions
 and  $(S,\delta)$ be a dimensional scheme.

We will denote by $\sh(S,\T)$ (resp. $\she(S,\T)$)
 the heart of the $\delta$-homotopy $t$-structure on $\T(S)$
 (resp. $\T^{\delta-\eff}(S)$). Objects of $\sh(S,T)$
 (resp. $\she(S,\T)$)
 will be called
 \emph{$\delta$-homotopy modules}
 (resp. \emph{effective $\delta$-homotopy modules})
 with coefficients in $\T$.

When $R$ is a ring of coefficients using the conventions
 of Paragraph \ref{num:convention_DM}, we simply put:
\begin{align*}
\sh(S,R):=\sh(S,\DM_R), \quad & \she(S,R):=\she(S,\DM_R), \\
\sht(S,R):=\sh\big(S,\D_{\AA^1,R}\big), \quad
 & \shte(S,R):=\she\big(S,\D_{\AA^1,R}\big).
\end{align*}
\end{df}
In the followings, we will simply denote by $H_i$
 the homology functor $H_i^\delta$ (resp. $H_i^{\delta-\eff}$).

\begin{ex}\label{ex:concrete_htp_heart}
Let $S=\spec(k)$ be the spectrum of a perfect field with characteristic exponent $p$,
 $\delta$ be the canonical dimension function on $S$
 and $R$ a ring of coefficients such that $p \in R^\times$.
 According to Examples \ref{ex:delta_htp-t-struct_field}
 and \ref{ex:delta_effective_fields}, we get the following
 description of the above categories:
\begin{itemize}
\item the category $\she(k,R)$ is equivalent to the category of homotopy
 invariant sheaves with transfers and coefficients in $R$ defined
 by Voevodsky in \cite[chap. 5]{FSV} -- and denoted by $HI(k) \otimes_\ZZ R$;
\item the category $\sh(k,R)$ is equivalent to the category of \emph{homotopy
 modules (with transfers)} and coefficients in $R$ of \cite[1.17]{Deg9};
\item the category $\sht(k,R)$ is equivalent to the category of
 \emph{(generalized) homotopy modules} with coefficients in $R$ defined
 by Morel (see \cite[1.2.2]{Deg10}).
\end{itemize}
Besides, using results of Suslin as explained in Remark \ref{rem:Suslin},
 the first two points are valid even when $k$ is not perfect.
\end{ex}

\begin{num}\textit{Basic properties}.\label{num:hrt_basic}
Recall that the category of $\delta$-homotopy modules and its effective variant
 do not depend on the choice of the dimension
 function $\delta$ on $S$ up to a canonical equivalence of categories
 (see \ref{lm:indep_delta} and \ref{rem:htp_t_nonnegdeg_eff}).

Moreover, using the adjunction of $t$-categories
 \eqref{eq:adj_t_delta_eff_st},
 we obtain an adjunction of abelian categories:
$$
H_0(s):\she(S,\T) \leftrightarrows \sh(S,\T):H_0(w).
$$
Let us assume that (good) is satisfied.
 Then according to point (2) of Corollary \ref{cor:nice_ppties_dhtp},
 the functor $H_0(w)=w$ is exact.
 Moreover, because the functor
 $s:\T^{\delta-eff}(S) \rightarrow \T(S)$ is fully faithful
 and its right adjoint $w$ is $t_\delta$-exact, we get that
 $H_0(s)$ is fully faithful. Thus $\she(S,\T)$ is a full
 subcategory of $\sh(S,\T)$; beware however it is not stable
 by kernel.

The categories $\sh(S,\T)$ and $\she(S,\T)$ are 
 abelian categories with exact small coproducts (see Remark \ref{rem:exist_colimits}).
 In fact,
 we will prove in Theorem \ref{thm:finally_Grothendieck} that under
 assumption (good), they are Grothendieck abelian categories and
 exhibit an explicit family of generators.
 According to Remark \ref{rem:htp_t_nonnegdeg_eff}(4)
 (resp. Cor. \ref{cor:Thom&delta-homotopy}),
 the category $\sh(S,\T)$ (resp. $\she(S,\T)$)
 is stable by the functor
 $\big(\MTh_S(v)[-r] \otimes -\big)$
 (resp. $\big(\MTh_S(E)[-r] \otimes -\big)$)
 for $v$ a virtual vector bundle (resp. $E$ a vector bundle) 
 over $S$ of rank $r$.
\end{num}
\begin{df}
For any $\delta$-homotopy module $\F$ of $\sh(S,\T)$
 (resp. of $\she(S,\T)$) and any virtual vector bundle $v$ (resp. vector space $E$)
 of rank $r$ over $S$, we put: 
$$
\F\gtw v:=\F \otimes \MTh_S(v)[-r]
 \quad \text{resp. }
 \F\gtw E:=\F \otimes \MTh_S(E)[-r].
$$
\end{df}
Note that if  $v$ (resp. $E$) is trivial of rank $r$,
 we get $\F\gtw v=\F\gtw r$ (resp. $\F\gtw E=\F\gtw r$)
 and this new notation is coherent with the notation for twists
 in the present paper.
 Note also that the endo-functor $(-\gtw v)$ of $\sh(S,\T)$
 (resp. $(-\gtw E)$ of $\she(S,\T)$) is an equivalence of categories
 (resp. right exact and fully faithful).
%
%
%
%

\begin{num} \label{num:heart_functoriality} \textit{Functoriality}. 
 Let $f:T \rightarrow S$
 be a morphism such that $\dim(f) \leq d$. Then according to
 point (3) of Proposition \ref{prop:basic_tstruct},
 we get an adjunction of abelian categories:
$$
H_{-d}f^*:\sh(S,\T) \leftrightarrows \sh(T,\T):H_df_*
$$
and according to Paragraph \ref{num:adj_tdelta_eff},
 the right exact functor $H_{-d}f^*\gtw{-d}$ preserves
 effective objects
 (images of effective $\delta$-homotopy modules).
In particular, when $f$ is quasi-finite, we get an adjunction:
$$
H_0f^*:\sh(S,\T) \leftrightarrows \sh(T,\T):H_0f_*,
$$
and the functor $H_0f^*$ preserves effective objects.
Moreover, we get better properties in the following two cases,
 using the notation of the previous definition:
\begin{itemize}
\item when $f$ is \'etale, the functor $H_0f^*=f^*$ is exact and
 acts on the homology of Borel-Moore motives as: \\
 $f^*H_i\big(\Mb(X/S)\big)\gtw n=H_i\big(\Mb(X \times_S T/T)\big)\gtw n$;
\item when $f$ is finite, the functor $H_0f_*=f_*$ is exact
 and acts the homology of Borel-Moore motives as: \\
  $f_*H_i\big(\Mb(Y/T)\big)\gtw n=H_i\big(\Mb(Y/S)\big)\gtw n$.
\end{itemize}

Assume now that $f$ is separated. Then applying
 point (1) of Proposition \ref{prop:basic_tstruct},
 we get an adjunction of abelian categories:
$$
H_0f_!:\sh(T,\T) \leftrightarrows \sh(S,\T):H_0f^!
$$
and according to Paragraph \ref{num:adj_tdelta_eff},
 the functor $H_0f_!$ preserves effective objects.
 Moreover, we get the following additional informations:
\begin{itemize}
\item when $f$ is smooth (resp. under assumption (good)),
 the functor $H_0f^!=f^!$ is exact and acts on the homology of Borel-Moore motives
 (resp. assuming $X/S$ is proper)
 as:
 \\ $f^!H_i\big(\Mb(X/S)\big)\gtw n=H_i\big(\Mb(X \times_S T/T)\big)\gtw n$.
\end{itemize}

One gets a shadow of the 6 functors formalism.
 For example, if $f$ is finite, $H_0f_!=H_0f_*$.
 Moreover one easily gets the following result
 (from the 6 functors formalism satisfied by $\T$). 
\end{num}

\begin{prop}
Consider a cartesian square of schemes:
$$
\xymatrix@=10pt{
Y\ar^g[r]\ar_q[d] & X\ar^p[d] \\
T\ar^f[r] & S
}
$$
such that $p$ is separated and $\dim(f) \leq d$.
 Then $\dim(g) \leq d$ and we get a canonical isomorphism
 of functors:
\begin{align*}
H_{-d}f^* \circ H_0p_! &\xrightarrow{\ \sim\ } H_0q_! \circ H_{-d}g^* \\
H_dp_* \circ H_0f^! & \xrightarrow{\ \sim\ } H_0g^! \circ H_dq_*.
\end{align*}
\end{prop}
\begin{proof}
According to the 6 functors formalism (cf. \cite[2.4.50]{CD3}),
 one gets for any object $\F$ of $\sh(S,;\T)$ the isomorphisms 
 $f^*g_!(\F)[d] \xrightarrow{\sim} q_!g^*(\F)[d]$.
 The first isomorphism of the proposition is obtained by applying
 the functor $\tau_{\leq 0}$ to the previous isomorphisms and by using 
 the fact the functors $f^*[d]$, $g_!$ are all
 right $t_\delta$-exact. The second isomorphism of the statement
 follows by adjunction.
\end{proof}

\begin{rem}
\begin{enumerate}
\item Considering a cartesian square as in the proposition, 
 one can derive several relations analogous to the previous ones.
 For example:
\begin{itemize}
\item if $f$ is \'etale, then for any integer $i \in \ZZ$,
 we get an isomorphism:
 $f^* \circ H_ip_! \simeq H_iq_! \circ g^*$;
\item If $f$ is smooth or under assumption (good),
 then for any integer $i \in \ZZ$,
 we get and isomorphism of functors:
 $H_ip_* \circ f^! \simeq g^! \circ H_iq_*$.
\end{itemize}
\item All these isomorphisms satisfy the usual compatibilities
 with respect to compositions (see \cite[Rem. 1.1.7]{CD3} for example).
\end{enumerate}
\end{rem}

\begin{num} \textit{Monoidal structures}. Assume
 the dimension function $\delta$ on $S$ is non-negative.

Then the tensor product $\otimes$ is right $t_\delta$-exact
 (Proposition \ref{prop:basic_tstruct}) and
 thus induces a closed monoidal structure $\otimes^\delta$
 on $\sh(S,\T)$ such that for any $\F, \G$,
 one has:
$$
\F \otimes^\delta_S \G:=H_0^\delta(\F \otimes_S \G).
$$
It follows from Paragraph \ref{num:adj_tdelta_eff}
 that this tensor product preserves effective objects.

We will also denote by $\uHom^\delta$ the internal Hom with
 respect to this tensor structure. It is given by
 the following formula:
$$
\uHom^\delta_S(\F,\G)=H_0^\delta\uHom_S(\F,\G).
$$

Recall that changing $\delta$ only changes $\sh(S,\T)$
 up to a canonical equivalence of categories.
 But the equivalence of categories involved here will
 not be monoidal. In particular, the monoidal structure
 on $\sh(S,\T)$ depends on the choice of $\delta \geq 0$.

Using again the six functors formalism, one gets the following
 formulas.
\end{num}
\begin{prop}
Assume $(S,\delta)$ is a dimensional scheme such that 
 $\delta\geq 0$.
Let $f:T \rightarrow S$ be a quasi-finite and separated morphism.
Then for $\delta$-homotopy modules $\F \in \sh(T,\T)$, $\G \in \sh(S,\T)$
 one has a functorial isomorphism:
$$
H_0f_!\big(\F \otimes_S^\delta H_0f^*(\G)\big)
 \xrightarrow \sim H_0f_!(\F) \otimes_T^\delta \G
$$
\end{prop}
The proof is the same as the proof of the previous proposition.

\begin{rem}
If $f$ is not quasi-finite, but we have an integer $d>0$
 such that $\dim(f) \leq d$, we obtain that
 $H_0f_!(\F \otimes_S^\delta H_{-d}f^*(\G))=0$
 because for any $\delta$-homotopy module $\F$ over $S$,
 $H_{-d}f_!(\F)=0$.
\end{rem}

The most important fact about the 6 functors formalism
 and $\delta$-homotopy modules comes from the gluing property
 of the $\delta$-homotopy $t$-structure (cf. Corollaries
 \ref{delta-htp_glued} and \ref{delta-htp_eff_glued}).
\begin{prop}
Let $(S,\delta)$ be a dimensional scheme and
 $i:Z \rightarrow S$ be a closed immersion
 with complementary open immersion $j:U \rightarrow S$.
 We denote by $\delta$ the induced dimension function on $Z$ and $U$.

Then $\delta$-homotopy modules satisfy the gluing
 formalism:
$$
\xymatrix@=30pt{
\sh(U,\T)\ar@<5pt>^{H_0j_!}[r]\ar@<-4pt>_{H_0j_*}[r]
 & \sh(X,\T)\ar|{j^*}[l]\ar@<5pt>^{H_0i^*}[r]\ar@<-4pt>_{H_0i^!}[r]
 & \sh(Z,\T)\ar|{i_*}[l]
}
$$
such that $\sh(Z,T)$ is a quotient of the abelian category
 $\sh(X,\T)$ by the thick abelian subcategory $\sh(U,\T)$.
 Under (good), one has an isomorphism of functors $H_0i^!=i^!$,
 which both are exact. Moreover all the above 6 functors
 preserve effective objects.

For any $\delta$-homotopy module $\F$ over $X$,
 one has exact sequences:
\begin{align*}
0 \longrightarrow H_0j_!j^*(\F) \longrightarrow &\F
 \longrightarrow i_*H_0i^*(\F) \longrightarrow 0, \\
0 \longrightarrow i_*H_0i^!(\F) \longrightarrow &\F
 \longrightarrow H_0j_*j^*(\F) \longrightarrow 0.
\end{align*}
\end{prop}

In particular, $\delta$-homotopy modules over a scheme
 are determined by their restrictions to a closed subscheme
 and to its open complement. From Corollary \ref{cor:restriction_fields},
 we get the following stronger form.
\begin{prop}\label{prop:test_on_fields}
Suppose assumpion (good) is fulfilled and $\T$ is continuous
 (\cite[4.3.2]{CD3}). Let $S$ be a scheme.
\begin{enumerate}
\item For any set-theoretic point $x \in S$,
 the functor $i_x^!$ of Paragraph \ref{num:i^!_essentially}
 induces an exact functor
$$
i_x^!:\sh(S) \rightarrow \sh(x)
$$
which commutes with colimits,
 preserves $\delta$-effective modules and acts on the homology of Borel-Moore motives,
 for $X/S$ proper, as: \\
$i_x^!H_i\big(\Mb(X/S)\big)\{n\}=H_i\big(\Mb(X_x/x)\big)\{n\}$.
\item The family of functors $i_x^!:\sh(S) \rightarrow \sh(x)$
 for $x \in S$ is conservative.
\end{enumerate}
\end{prop}
\begin{proof}
The only assertion which requires an argument at this point
 is the fact $i_x^!$ commutes with colimits. As it is exact,
 it is enough to show it commutes with coproducts. Going
 back to the definition (Paragraph \ref{num:i^!_essentially}),
 it is enough to remark that for any closed immersion,
 the functor $i^!$ commutes with coproducts. This is in fact
 a consequence of the assumption that $\T$ is compactly generated
 by its Tate twists and from the localization triangle.
\end{proof}

We end up this section with the following additional proposition,
 which follows directly from \ref{lm:inseparable_texact}.
\begin{prop}\label{prop:pre_rost}
Assume the triangulated motivic category is semi-separated
 (see \cite[2.1.7]{CD3} or footnote number \ref{foot:semisep}
  page \pageref{foot:semisep}).

Then for any finite surjective radicial morphism $f:T \rightarrow S$
 the functor $f^*:\T(S) \rightarrow \T(T)$ induces an equivalence
 of abelian categories:
$$
f^*:\sh(S,\T) \rightarrow \sh(T,\T).
$$
Moreover,
 the same is true for effective $\delta$-homotopy modules.
\end{prop}
Recall in particular from
 Example \ref{ex:DM_semi_sep} that this can be applied to the
 categories $\sh(-,R)$ and $\she(-,R)$ when $R$ is a ring
 and we consider conventions (1) or (2) of
 Paragraph \ref{num:convention_DM}.

%
%
%

\subsection{Fiber functors and generators}

The preceding proposition can also be restated using the following
 definition, which mimick the notion of fiber functors for sheaves:
\begin{df}\label{df:Rost_transform}
Consider the notation of Definition \ref{df:heart}.
Let $\F$ be an object of $\sh(S,\T)$
 (resp. $\she(S,\T)$). For any point $x \in S(E)$
 and any integer $n \in \ZZ$ (resp. $n \leq 0$),
 we define the fiber of $\F$ at the point $(x,n)$ 
 as the following abelian group:
$$
\hat \F^\delta_n(x):=\rH_0^\delta(\F)(x,n)
$$
using the notation of Definition \ref{df:fiber_hlg}.
 When $\delta$ is clear, we will simply put: $\hat F_*=\hat F_*^\delta$.
\end{df}
In particular one can see $\hat \F_*$ as
 a graded functor on the discrete category $\pts(S)$
 of points of $S$. As a consequence of the main theorem
 \ref{thm:hlg&htp_t}, together with Lemma \ref{lm:rost_hlg},
 we get the following important result (compare with \cite[chap. 3, 4.20]{FSV}).
\begin{prop}\label{prop:conservativity_htpm}
Assume that (good) is satisfied  (Par. \ref{num:good}).
 Then the functors:
$$
\left.
\begin{array}{r}
\sh(S,\T) \longrightarrow \widehat{\pts(S) \times \ZZ} \\
\she(S,\T) \longrightarrow \widehat{\pts(S) \times \ZZ^-}
\end{array}\right\}, \F \mapsto \hat \F_*
$$ 
are conservative, exact and commute with colimits.
\end{prop}
In other words, the objects of $\pts(S) \times \ZZ$
 (resp. $\pts(S) \times \ZZ^-$) defines a conservative
 family of points for $\delta$-homotopy modules
 (resp. effective $\delta$-homotopy modules).

\begin{rem}\label{rem:fiber_hlg&good}
Consider the assumptions of the preceding proposition.
\begin{enumerate}
\item According to relation \eqref{eq:exceptional_pullback&fiber_hlg},
 we obtain for any separated morphism $f:T \rightarrow S$,
 any $\delta$-homotopy module $\F$ over $S$ and any point $x$
 in $T(E)$:
$$
\widehat{f^!\F}_*(x)=\hat \F_*(f \circ x).
$$
Moreover, when $f$ is smooth,
 if one defines $\tilde \delta^f=\delta^f-\dim(f)$, we deduce
 from \eqref{eq:smooth_pullback&fiber_hlg} the following relation:
$$
\widehat{f^*\F}^{\tilde \delta^f}_*(x)=\hat \F^\delta_*(f \circ x).
$$
Note that in the case $\T=\DM_R$, this relation is also valid
 for $f$ essentially quasi-projective between regular schemes
 (see Remark \ref{rem:pullback&fiber_hlg}).
\item In a work in preparation, we will show that,
 when $\T=\DM_R$ or $\T=\MGLmod$,
 $\hat \F_*$ can be equipped with a rich functoriality,
 that of a cycle module over $S$ following the definition
 of Rost.
 It will be called the \emph{Rost transform of $\F$}
 along the lines of \cite{Deg9}.
\end{enumerate}
\end{rem}

Recall from Remark \ref{rem:exist_colimits} that
 the abelian categories $\sh(S,\T)$ and $\she(S,\T)$ admit colimits.
 A nice application of the previous theorem is the following corollary.
\begin{cor}\label{cor:finally_Grothendieck}
Under assumption (good) of Paragraph \ref{num:good},
 filtered colimits are exact in the abelian categories $\sh(S,\T)$, $\she(S,\T)$.
\end{cor}
\begin{proof}
According to the preceding theorem, we reduce to prove the corresponding
 fact for categories of presheaves of abelian groups, respectively over
 $\pts(S) \times \ZZ$ and $\pts(S) \times \ZZ^-$. This later fact is well known.
\end{proof}

Let us introduce some useful definition.
\begin{df}\label{df:0-Suslin_homology}
Let $\T$ be a motivic triangulated category
 and  $(S,\delta)$ be a dimensional scheme.

Let $X$ be a separated $S$-scheme,
 and $(X_i)_{i \in I}$ the family of its connected
 components.
 We put: $\hrep(X/S)
=\oplus_{i \in I} H_0\big((\Mb(X_i/S)\dtw{\delta(X_i)}\big)$. 
\end{df}

\begin{rem}\label{rem:compute_fiber}
Note that it follows from this definition and \ref{df:fiber_hlg}
 that one can compute the fiber of a $\delta$-homotopy module $\F$ at a point
 $(x,n)$ in $\pts(S) \times \ZZ$ as follows:
$$
\hat \F^\delta_n(x)
 =\ilim_{X \in \M(x)^{op}} \Hom_{\sh(S,\T)}\big(\hrep(X/S)\gtw{-n},\F\big).
$$
Of course the same assertion holds in the effective case,
 assuming $n \geq 0$.
\end{rem}

\begin{ex}
Assume $\T$ is oriented, $S$ is universally catenary and integral,
 and $\delta=-\codim_S$ (Ex. \ref{ex:can_dim_fn}).
Then for any smooth $S$-scheme $X$, $\hrep(X/S)=H_0(M_S(X))$
 -- apply \ref{num:BM}(BM4) as in Example \ref{ex:motive_sm_positive}.
\end{ex}

\begin{num} \textit{Functoriality properties}.
Because of the functoriality properties of Borel-Moore
 $\T$-spectra (BM1) and (BM2), we obtain that the $\delta$-homotopy module
 $\hrep(X/S)$ is covariant with respect to proper morphisms,
 and contravariant with respect to smooth morphisms,
 of separated $S$-schemes. Note in particular that,
 because of our choice of definition and Proposition
 \ref{prop:delta-base_change}(3), there is no twist in the 
 contravariant functoriality.

These functoriality properties correspond to the following
 formulas with respect to the functoriality of the category
 $\sh(S,\T)$ (see Paragraph \ref{num:heart_functoriality}):
\begin{itemize}
\item For any equidimensional morphism $f:T \rightarrow S$
 of relative dimension $d$ and any proper $S$-scheme $X$,
 we get:
$H_{-d}f^*\big(\hrep(X/S)\gtw n\big)=\hrep(X \times_S T/T)\gtw{n+d}$.
\item For any proper morphism $f:T \rightarrow S$
 and any separated $T$-scheme $Y$, we get: \\
 $H_0f_*\big(\hrep(Y/T)\big)=\hrep(Y/S)$.
\end{itemize}
Note the first assertion follows from the base change formula
 and point (3) of Proposition \ref{prop:delta-base_change}.
 The second assertion is obvious.
\end{num}

\begin{ex}
The definition above has been chosen to meet that of \cite[(1.18.a)]{Deg8}.
Consider the motivic category $\T=\DM_R$ under
 the conventions of point (1) or point (2) of
 Paragraph \ref{num:convention_DM}.
\begin{enumerate}
\item When $S=\spec(k)$ is the spectrum of a perfect field equipped with the
 obvious dimension function, then through the equivalence
  of categories of Example \ref{ex:delta_htp-t-struct_field}(1),
  the $\delta$-homotopy module $\hrep(X/k)$ for $X$ a smooth $k$-scheme
  corresponds to the homotopy module $h_{0,*}(X)$ of \cite[(1.18.a)]{Deg8}.
\item Let $X$ be a regular and proper $S$-scheme.
 Then for any point $x$ of $S$, one gets the following computation:
\begin{align*}
\widehat{\hrep(X/S)}(x,0)
&=\Hom_{\DM(S,R)}\big(\hMb(x)\dtw{\delta(x)},H_0\Mb(X/S)\dtw{\delta(X)}\big) \\
& \stackrel{(1)}=\Hom_{\DM(S,R)}\big(\hMb(x)\dtw{\delta(x)},\Mb(X/S)\dtw{\delta(X)}\big) \\
& \stackrel{(2)}=\Hom_{\DM(S,R)}\big(\hMb(X_x/X)\dtw{\delta(x)-\delta(X)},\un_X\big) \\
& \stackrel{(3)}=CH_{\dim=0}(X_x) \otimes_\ZZ R,
\end{align*}
where $\dim$ is the Krull dimension function on $X_x$.
Indeed, (1) follows from Corollary \ref{cor:generic_motives},
 (2) from the base change formula and because $X/S$ is proper,
 (3) from Proposition \ref{prop:comput_BM_w_Chow} because $X$ is regular.
\end{enumerate}
Actually, (2) is a generalization of \cite[\textsection 3.8]{Deg9}.
 It also enlightens the functoriality properties of $\hrep(X/S)$.
\end{ex}

Recall that in any category $\mathscr C$, an object $X$
 is called compact if the functor $\Hom_\C(X,-)$ commutes
 with coproducts.
\begin{prop}\label{prop:strong_generators_htpm}
Consider the notations of the previous 
 definition and suppose (good) is satisfied.
\begin{enumerate}
\item For any separated $S$-scheme $X$ and any integer $n \in \ZZ$,
 the $\delta$-homotopy module $\hrep(X/S)\gtw n$ is compact.
\item Under assumption (good), the family of $\delta$-homotopy
 modules $\hrep(X/S)\gtw n$
 for $X/S$ of finite type, $X$ an affine regular scheme and $n \in \ZZ$
 (resp. $n=0$)
 generates the abelian category $\sh(S,\T)$ (resp. $\she(S,\T)$).
\end{enumerate}
\end{prop}
\begin{proof}
Assertion (1) follows easily from the fact $\Mb(X/S)\dtw{\delta(X)}\gtw n$
 is $t_\delta$-non-negative and compact in the category $\T(S)$.

Consider assertion (2). We first remark that for any separated $S$-scheme $X$
 and any integer $n \geq 0$, one has:
$$
\hrep(X/S)\gtw n=\hrep\big(X \times \GG^n/S\big)/
 \oplus_{i=1}^n \hrep\big(X \times \GG^{n-1}/S\big).
$$
Thus, in the $\delta$-effective case, we can consider all integers $n \geq 0$
 in the description of the generators.

We have to prove that for any non zero map $h:M \rightarrow N$ in $\hrt{\T}$,
 there exists a map $f:\hrep(X/S)\gtw n \rightarrow M$ for $n \in \ZZ$ (resp. $n \geq 0$)
 such that $h \circ f$ is non zero.
 According to Proposition \ref{prop:conservativity_htpm}, there exist a point $x \in S(E)$
 and an integer $m \in \ZZ$ (resp. $m\leq 0$) such that the induced functor
$$
\hat M^\delta_m(x) \xrightarrow{h_*} \hat N^\delta_m(x)
$$
is non zero. According to the computation of Remark \ref{rem:compute_fiber},
 one deduces that there exists an $S$-model $X$ of the point $x$
 (recall Definition \ref{df:Smodels}) such that the following map
 is non zero:
$$
\Hom\big(\hrep(X/S)\gtw{-m},M\big) \xrightarrow{h_*} \Hom\big(\hrep(X/S)\gtw{-m},N\big).
$$
This concludes.
\end{proof}

\begin{rem}
Apart from being compact, one can introduce the following
 conditions of finiteness on a $\delta$-homotopy module $M$ over $S$
 (compare with \cite[6.6, 6.7]{Deg9}):
\begin{itemize}
\item \emph{finitely generated} if any sum of subobjects
 of $M$ must be a finite sum;\footnote{this is standard in abelian categories;
 it amounts to ask that any map $\oplus_{i \in I} N_i \rightarrow M$
 factors as $\oplus_{i \in I_0} N_i \rightarrow M$, where $I_0$ is a finite subset of $I$;}
\item \emph{pseudo-finitely generated} if it is a quotient
 of a $\delta$-homotopy module of the form $\hrep(X/S)\gtw n$
 for $X/S$ of finite type $X$, $X$ affine and regular and $n \in \ZZ$;
\item \emph{$t_\delta$-constructible}
 (resp. \emph{strongly $t_\delta$-constructible})
 if it belongs to the smallest abelian thick subcategory of $\sh(S,\T)$
 which contains objects of the form $\hrep(X/S)\gtw n$
 for $X/S$ of finite type $X$, $X$ affine and regular and $n \in \ZZ$
 (resp. it is of the form $H_0(\E)$ for a constructible $\T$-spectrum
 $\E$ over $S$).
\end{itemize}
Given the preceding proposition, it is an exercise in abelian categories
 to show the following implications:
\begin{align*}
\text{finitely generated}
 \Rightarrow \text{pseudo-finitely generated}
 &\Rightarrow \text{compact and $t_\delta$-constructible,}
\end{align*}
but other implications are unclear.
\end{rem}

Let us state explicitly the following important result obtained by
 putting together the preceding proposition and the above corollary.
\begin{thm}\label{thm:finally_Grothendieck}
Suppose that assumption (good) of Paragraph \ref{num:good} is satisfied.

Then the abelian category $\sh(S,\T)$ (resp. $\she(S,\T)$)
 is a Grothedieck abelian category with compact generators
 $\hrep(X/S)\gtw n$ (see Definition \ref{df:0-Suslin_homology})
 for $X/S$ of finite type, $X$ affine and regular and $n \in \ZZ$ (resp. $n=0$).
\end{thm}

\subsection{Comparison between $\delta$-homotopy hearts}

As a generalization of Example \ref{ex:delta_htp-t-struct_field},
 we get the following result:
\begin{thm}\label{thm:compute_heart}
Let $R$ be a ring and assume that one of the two following conditions
 hold:
\begin{enumerate}
\item[(a)] $R$ is a $\QQ$-algebra.
\item[(b)] $\base$ is the category of $\QQ$-schemes.
\end{enumerate}
Use the notation of Example \ref{ex:adjunctions_heart}.
Then the following assertions hold for any scheme $S$ in $\base$:
\begin{enumerate}
\item assume $R=\QQ$ (resp. $R=\ZZ$) under assumption (a) (resp. assumption (b)).
 Then the adjunction of abelian categories:
$$
H_0\delta^*:\hrt{(\SH(S) \otimes_\ZZ R)} \leftrightarrows \sht(S,R):\delta_*
$$
is an equivalence, compatible with the monoidal structure
 if $\delta \geq 0$.
\item The exact functor
\begin{align*}
\gamma_*:\sh(S,R) & \rightarrow \sht(S,R)
\end{align*}
is fully faithful and its essential image is equivalent
 to the category of generalized $\delta$-homotopy modules
 with trivial action of $\eta$
 (see Paragraph \ref{num:concrete_htp_heart}).
\end{enumerate}
\end{thm}
\begin{proof}
When $R$ is a $\QQ$-algebra, each point is a consequence
 of the stronger statement that it is already true for the
 full triangulated motivic categories involved:
 see \cite[5.3.35]{CD3} for point (1) and \cite[16.2.13]{CD3}
 for point (2).

So we restrict to the case of assumption (b). For each
 point, we will use Proposition \ref{prop:test_on_fields}.

First, let us remark that the functor $\phi=\delta_*$ 
 (resp. $\phi=\gamma_*$)
 commutes with functors $i_x^!$ for any set-theoretic point
 $x \in S$. By definition of $i_x^!$, this boils down to the fact
 $\phi$ commutes with $j^*$ where $j$ is a pro-open immersion.
 One can check
 easily this follows from the continuity property of
 the triangulated motivic categories involved and the fact it is
 true when $j$ is open.

Secondly, Corollary \ref{cor:prop:adj&f_*} shows that the
 functors $H_0\delta^*$ and $H_0\gamma^*$ commutes with $i_x^!$.

Then we can prove point (1). We need to prove that
 the two adjunction maps $H_0\delta^*\delta_* \rightarrow Id$
 and $Id \rightarrow \delta_*H_0\delta^*$ are isomorphisms
 over the base scheme $S$.
 But using Proposition \ref{prop:test_on_fields},
 we only need to check that after applying the functor $i_x^!$
 for any point $x \in S$. Thus we are restricted to the
 case where $S$ is the spectrum of a field of characteristic
 $0$ which follows from Proposition \ref{prop:compute_heart}. 

\bigskip

Finally, we remark that the fact $\eta$ acts trivially
 on a homotopy module $\E$ is preserved by any functor $i_x^!$
 and detected by the family of functors
 $i_x^!$ for a set-theoretic point $x$ in $S$. In fact,
 the functor $i_x^!$ commutes with $\GG$-twists;
 this follows from the fact $\GG$ is $\otimes$-invertible in $\SH(S)$
 and from the formula: 
$$
i_x^!\uHom(\GG,E) \simeq \uHom(i_x^*\GG,i_x^!E)=\uHom(\GG,i_x^!E);
$$
see \cite[2.4.50]{CD3}. Then one has only to remark that
 the following diagram is commutative in $\SH(\kappa(x))$:
$$
\xymatrix@R=12pt@C=30pt{
i_x^!(\GG \otimes \E)\ar^-{i_x^!\big(\gamma_\eta^\E\big)}[r] & i_x^!(\E)\ar@{=}[d] \\
\GG \otimes i_x^!(\E)\ar^-{\gamma_\eta^{i_x^!\E}}[r]\ar^\sim[u] & i_x^!(\E)
}
$$
where $\gamma_\eta^?$ is the map representing multiplication by $\eta$

Let us denote by $\sht(S,R)/\eta$ the full subcategory
 of $\sht(S,R)$ consisting of those objects $\E$ such that $\eta$ acts
 trivially on $\E$ (see
 Paragraph \ref{num:concrete_htp_heart}).
 According to the preceding subsection and
 Proposition \ref{prop:test_on_fields}, the family
 of functors $i_x^!$ for a point $x \in X$ 
 induces a conservative family of functors
 $\sht(S,R)/\eta \rightarrow \sht(x,R)/\eta$.

Since any object $\E$ of $\DM(S,R)$ defines an orientable
 cohomology theory, the Hopf map $\eta$ acts trivially 
 on $\E$. This implies that the image of $\sh(S,R)$
 by the map $\gamma_*$ lies in $\sht(S,R)/\eta$ and
 therefore we get an adjunction of abelian categories:
$$
H_0\gamma^*:\sht(S,R)/\eta \leftrightarrows \sh(S,R):\gamma_*.
$$
We need to prove that these are equivalences of categories.
 Now, applying the conservative family of functors $i_x^!$
 for any point $x \in S$, we can assume that $S$
 is the spectrum of a field of characteristic $0$.
 This is Proposition \ref{prop:compute_heart}.
\end{proof}

\begin{rem}
What is missing to deal with the case of fields of positive characteristic $p$
 in the preceding theorem 
 is the fact that the triangulated motivic category $\SH[1/p]$
 is semi-separated
 (see footnote \ref{foot:semisep} page \pageref{foot:semisep}
 for the definition). Indeed, this fact will immediately
 imply that Proposition \ref{prop:compute_heart} can be generalized
 to arbitrary fields, up to inverting $p$.
\end{rem}

\begin{rem}\label{rem:hrt_MGL}
We can also describe the heart of the triangulated motivic category
 $\MGLmod[1/p]$ of $\MGL$-modules when $\base$ is the category
 of $F$-schemes for a prime field $F$ of characteristic exponent $p$.
 Here are the main steps: first, using the premotivic adjunction of remark \ref{rem:MGL->DM},
 we obtain for any dimensional scheme $(S,\delta)$ in $\base$
 an adjunction of abelian categories:
$$
H_0\lambda^*:\hrt{\MGLmod}
 \leftrightarrows \sh(S,\ZZ[1/p]):\lambda_*
$$
such that $\lambda_*$ is exact
 (cf. Corollary \ref{cor:premotivic_adj_t-exactness}).
 We will prove this adjunction is in fact an equivalence of categories.

Using Proposition \ref{prop:test_on_fields} as in the preceding proof,
 we restrict to the case where $S$ if the spectrum of a field $k$.
 From Example \ref{ex:DM_semi_sep}, we get that the triangulated motivic
 categories $\DM_{\ZZ[1/p]}$
 and $\MGLmod[1/p]$ are semi-separated over $\base$.
 Using Lemma \ref{lm:inseparable_texact}, we restrict
 to the case where $k$ is perfect.

This last case now follows as we can prove that the category on
 the left hand side is equivalent to Rost category
 of cycle modules using the arguments of \cite{Deg11}
 applied to the category $\MGLmod$ instead of $\DM$. Note in particular
 that this is possible because the $0$-th stable homotopy sheaf of
 $\MGL$ (that is the $0$-th homology group of $\MGL$ with
 respect to Morel's homotopy $t$-structure on $\SH(k)$)
 can be computed as follows:
$$
\underline \pi_0(\MGL)(k)=K_*^M(k);
$$
see \cite[6.4.5]{Mor3}.
\end{rem}

\subsection{Examples and computations}

\begin{num}\label{num:i_x^!}
Assume that $\T$ is absolutely pure (cf. Definition \ref{df:abs_purity}).

Let $S$ be a regular scheme and $x \in S$ a set-theoretic point.
 Consider the notation of paragraph \ref{num:i^!_essentially}:
 $\bar x$ is the reduced closure of $x$ in $S$.
 According to our general conventions, $S$ is excellent so that $\bar x$
 is also excellent; thus there exists an open neighbourhood
 $U$ of $x$ in $S$ such that $\bar x \cap U$ is regular.
 Let us consider the following immersions:
 $\bar x \xrightarrow{j_U} \bar x \cap U \xrightarrow{i_U} U$.
 Using the absolute purity property
 for the closed immersion $i_U$, we get a fundamental class:
$$
\eta_{x,U}:\MTh(-N_U) \rightarrow i_U^!(\un_U)
$$
where $N_U$ is the normal bundle of $i_U$.
Applying the functor $j_U^*$, we get:
$$
\eta_x:\MTh(-N_x) \rightarrow j_U^*i_U^!(\un_U)=i_x^!(\un_X)
$$
where $N_x$ is the normal bundle of $x$ in $\spec(\mathcal O_{X,x})$.
Since fundamental classes are compatible with pullbacks
 along open immersions, this map does not depend on the choice of
 $U \subset S$.

Recall also that, from the six functor formalism, we get for
 any $\T$-spectra $\E$, $\F$ over $S$ a canonical map
\begin{equation}\label{eq:pairing!*}
i^*(\E) \otimes i^!(\F) \rightarrow i^!(\E \otimes \F)
\end{equation}
by adjunction from the canonical one:
$$
i_!\big(i^*(\E) \otimes i^!(\F)) \xrightarrow \sim
  \E \otimes i_!i^!(\F) \xrightarrow{ad(i_!,i^!)} \E \otimes \F.
$$
\end{num}
\begin{df}
Consider the notations and assumptions above.
Let $\E$ be a $\T$-spectrum over $S$.

For any point $x \in S$, we will say that $\E$ is \emph{punctually pure at $x$}
 if the following canonical map:
$$
i_x^*(\E) \otimes \MTh(-N_x)
 \xrightarrow{\eta_x} i_x^*(\E) \otimes i_x^!(\un_S)
 \xrightarrow{\eqref{eq:pairing!*}} i_x^!(\E)
$$
is an isomorphism.

On says $\E$ is punctually pure if
 it is punctually pure at all points of $S$.
We will denote by $\T^{\pure,x}(S)$ (resp. $\T^{\pure}(S)$)
 the full subcategory of $\T(S)$ made by $\T$-spectra which are
 pure at $x \in S$ (resp. pure).
\end{df}

\begin{rem}
Typical examples of non punctually pure $\T$-spectra
 are objects in the image of $j_*$ for an open immersion
 $j:U \rightarrow S$.
 For example, in the case $\T=\DM_\QQ$,
 if $x$ is the closed point of a scheme $S$, of codimension $c$,
 $U=S-\{x\}$,
 then $i^*_xj_*(\un_U)=\un_x \oplus \un_x(-c)[1-2c]$ while
 $i^!_xj_*(\un_U)=0$.
\end{rem}

The following proposition is clear:
\begin{prop}\label{prop:elementary_pure}
Consider the assumptions and notations of the preceding definition.
\begin{enumerate}
\item The subcategories $\T^{\pure,x}(S)$ and $\T^{\pure}(S)$
 are stable under extensions, suspensions, direct factors, and arbitrary coproducts.
 They are also stable under tensor product by an invertible object.
\item The property of being punctually pure is local for the
 Nisnevich topology on $S$. If $\T$ is separated for the \'etale
 topology\footnote{\emph{i.e.} for any \'etale cover $f:T \rightarrow S$,
 the functor $f^*$ is conservative; see \cite[2.1.5]{CD3};}
 the same is true for the \'etale topology on $S$.
\item For any smooth proper $S$-scheme $X$, the $\T$-spectra
 $\Mb(X/S)$ and $M_S(X)$ are punctually pure.
\end{enumerate}
\end{prop}
\begin{proof}
All the assertions in point (1) are clear, except possibly the assertion
 about coproducts; it follows from the fact $i^!$ commutes with coproducts
 since the functor $i_!$ respects compact (\emph{i.e.} constructible
 under our assumptions) $\T$-spectra according to \cite[4.2.9]{CD3}.

Point (2) follows from the fact that fundamental classes are compatible
 with pullbacks along \'etale morphisms, and from the fact
 $f^*$ is conservative for a Nisnevich cover $f$ (see \cite[2.3.8]{CD3}).

For point (3), in view of \ref{num:BM}(BM4), it is sufficient
 to consider the case of $\Mb(X/S)$. It follows from the following two facts:
\begin{itemize}
\item if one denotes by $f$ the structural morphism of $X/S$,
 the functor $f_!=f_*$ commutes with $i_x^!$ and $i^*$;
\item the immersion $X_x \rightarrow X$ is an essentially closed immersion
 of regular schemes whose fundamental class is the pullback of
 that of $i_x$ along the smooth morphism $f$;
 this follows from \cite[2.4.4]{Deg12} applied to the base change
 along $f$ of the closed immersion $i_U$ which appears in the definition
 of $i_x^!$ in Paragraph \ref{num:i_x^!}.
\end{itemize}
\end{proof}

\begin{num}\label{num:AHP}
We now recall the main constructions of \cite{AHP}.
 We consider the triangulated motivic category $\DM_R$
 as in the convention of point (1) of \ref{num:convention_DM}.
 In particular, $R$ is a $\QQ$-algebra.\footnote{The constructions
 of \emph{loc. cit.} are usually done only in the case 
 where the coefficient ring $R=\QQ$; 
 yet they can be carried over to the case
 of an arbitrary $\QQ$-algebra $R$.}

Let $\grp_S$ be the category of smooth commutative
 group schemes over $S$ and $\shet(S,R)$ the category
 of $R$-linear \'etale sheaves.
 Given such a group scheme $G$,
 we denote by $\underline{G/S}_R$ (see \cite[2.1]{AHP})
 the \'etale sheaf of $R$-vector spaces on $\sm_S$ represented by $G$:
 for any smooth $S$-scheme $X$,
 $\Gamma(X,\hgrp G)=\Hom_S(X,G) \otimes_\ZZ R$.

Let us consider the following composite functor:
$$
\mathcal M:\Der(\shet(S,R))
 \xrightarrow \pi D_{\AA^1,\et}^{\eff}(S,R)
 \xrightarrow{\Sigma^\infty} D_{\AA^1,\et}(S,R)
 \simeq \DM(S,R)
$$
where the first map is the $\AA^1$-localization functor
 to the effective \'etale $\AA^1$-derived category over $S$
 (see \cite{ayoub1}, or \cite{CD3}),
 $\Sigma^\infty$ is the infinite suspension functor
 and the last equivalence is Morel's theorem as proved
 in \cite[16.2.18]{CD3}.
 Note that according to the theory developed in \cite{CD3},
 $\mathcal M$ is in fact the left adjoint of a premotivic 
 adjunction of triangulated categories.

We will denote abusively by $\hgrp G$ the image of
 $\underline{G/S}_R$, seen as a complex in degree $0$,
 by this canonical functor
 --- this is denoted by $M_1^{\eff}(G/S)$ in \cite[2.3]{AHP}.
 If we assume further that $S$ is a regular scheme
 with dimension function $\delta=-\codim_S$, then it
 follows from this definition that $\underline G$ is
 a $\delta$-effective motive over $S$.

Let us summarize the basic properties of this construction
 (see \cite[section 2]{AHP}).
\end{num}
\begin{prop}[Ancona, Pepin Lehalleur, Huber]\label{prop:AHP1}
Consider the preceding notations. For any regular scheme $S$,
 there exists a canonical functor:
$$
\grp_S \rightarrow \DM(S,R), \ 
 G \mapsto \hgrp G
$$
which is additive and sends exact sequences to distinguished triangles.
\end{prop}

\begin{num}
We will also use the central construction 
 of \cite[D.1]{AHP}. Let us consider the preceding notations
 and describe the construction of \emph{loc. cit.}
 in the particular case where we will use it.
 For any commutative group scheme 
 $G$ over $S$, there is a homologically non-negative complex
 of rational \'etale sheaves $A(G,R)$ together with a natural
 quasi-isomorphism of complexes of \'etale sheaves:
$$
r_G:A(G,R) \rightarrow \underline{G/S}_R
$$
such that for any index $i$, the $i$-th term $A(G,R)_i$
 is of the form $R_S(H_i(G))$ where $H_i(G)$ is a finite coproduct
 of certain powers of $G$, seen as a smooth $S$-scheme $Y$,
 and $R_S(Y)$ denote as usual the $R$-linear \'etale sheaf freely
 represented by $Y$.

Moreover, for any morphism of schemes $f:T \rightarrow S$,
 one has the relation:
\begin{equation}\label{eq:AHP_pullback}
H_i(G \times_S T)=H_i(G) \times_S T.
\end{equation}
We can summarize this construction
 -- a kind of \emph{cofibration resolution lemma} for $G$ --
 using our slightly different notations as follows:
\end{num}
\begin{prop}[Ancona-Pepin Lehalleur-Huber]
Consider the above notations.

Then one has a canonical isomorphism in $\DM(S,R)$:
$$
\rho_G:\hgrp G=\mathcal M(\hgrp G)
 \xrightarrow{r_G^{-1}} \mathcal M(A(G,R))
 = \hocolim_{i \in \NN} M_S(H_i(G))
$$
where the homotopy colimit runs over the category
 associated with the ordered set $\NN$.
\end{prop}
Recall from \cite[2.7]{AHP} that the main corollary of this proposition
 is the following commutativity with base change along 
 a morphism of schemes $f:T \rightarrow S$:\footnote{apply the proposition
 together with relation \eqref{eq:AHP_pullback}}
\begin{equation}\label{eq:AHP_basechg}
f^*(\hgrp G)=\hgrp{G \times_S T}.
\end{equation}
\begin{proof}
Let us consider the \'etale descent model category structure
 on $\Comp(\shet(S,R))$ (see \cite{CD1}) whose homotopy
 category is $\Der(\shet(S,R))$. 
 Since $A(G,R)$ is a bounded below complex of cofibrant objects,
 we get the following relation in $\Der(\shet(S,R))$:
$$
A(G,R)=\hocolim_{i \in \NN} \big(A(G,R)_i\big)
 =\hocolim_{i \in \NN} \big(R_S(H_i(G))\big).
$$
Then it is sufficient to apply the functor $\mathcal M$,
 which commutes with homotopy colimits as a left adjoint,
 to this isomorphism to get the above statement.
\end{proof}

One deduces from the preceding proposition the following result:
\begin{prop}\label{prop:semi-ab_pure}
Consider the above notations.
Then for any regular scheme $S$ and any semi-abelian scheme $G$
 over $S$, the motive $\hgrp G$ of $\DM(S,R)$ is punctually pure.
\end{prop}
\begin{proof}
According to Propositions \ref{prop:AHP1}
 and \ref{prop:elementary_pure}(1), one needs only to consider
 the case where $G$ is an abelian variety or a torus.
 The case of an abelian variety follows from the preceding proposition
 and points (1) and (3) of \ref{prop:elementary_pure}.
 The case of a torus follows from the \'etale separation property
 of $\DM_R$ (recall it is even separated \cite[14.3.3]{CD3}),
 and points (1) and (2) of \ref{prop:elementary_pure}. 
\end{proof}

\begin{num}
Recall that for a family $G_1, ..., G_n$ of semi-abelian varieties
 over a field $k$, Somekawa has introduced in \cite{Som} some abelian groups,
 now called \emph{Somekawa K-groups},
 defined by generators and relations that we will denote by:
$$
K(k;G_1,...,G_n).
$$
These groups generalize both Milnor K-theory in degree $n$
 (take all $G_i=\GG$) and Bloch's group attached to the Jacobian
 $J$ of a smooth projective $k$-curve (take n=2, $G_1=J$ and $G_2=\GG$).
\end{num}
\begin{thm}\label{thm:compute_sab_hlg}
Consider a regular scheme $S$ with dimension function
 $\delta=-\codim_S$ and let $G$ be a semi-abelian $S$-scheme.
\begin{enumerate}
\item The $\delta$-effective motive $\hgrp G$
 is in the heart of the $\delta$-homotopy $t$-structure
 on $\DM^{\delta-\eff}(S,R)$.
\item For any point $x \in S(E)$, and any integer $n \in \ZZ$,
 one has the following isomorphisms:
$$
\widehat{\hgrp G}(x,n)=
\begin{cases}
G_x(E) \otimes_\ZZ R & \text{if } n=0, \\
L_x(E) \otimes_\ZZ R & \text{if } n=-1, \\
0 & \text{if } n<-1, \\
K(E;G_x,\underset{n \text{ times}}{\underbrace{\GG,...,\GG}}) \otimes_\ZZ R
 & \text{if } n>0,
\end{cases}
$$
where $L_x$ is the group of cocharacters of the toric part
 of the semi-abelian variety $G_x$ over $E$.
\end{enumerate}
\end{thm}
\begin{proof}
According to Corollary \ref{cor:restriction_fields},
 to prove point (1), we need only to prove that for any point $x \in S$,
 the $\delta$-effective motive $i_x^!(\hgrp G)$ is in the heart
 of $\DM^{\delta^x-\eff}(x,R)$, where $\delta^x$ is the dimension function
 on $x$ induced by the dimension function $\delta$ on $S$.
 
Because of the preceding theorem and formula \eqref{eq:AHP_basechg},
 we get:
$$
i_x^!(\hgrp G)=i_x^*(\hgrp G)\dtw{-c_x}=\hgrp{G_x}\dtw{-c_x}
$$
where $c_x$ is the codimension of $x$ in $S$. Let $k$ be the residue
 field of $x$ and $\delta_k$ the canonical dimension function
 on $x=\spec(k)$.
 It is clear that we have the relation, as dimension function on $x$:
 $\delta^x=\delta_k-c_x$.
 In particular, the canonical functor:
$$
\DM^{\delta^x-\eff}(x,R) \rightarrow \DM^{\delta_k-\eff}(x,R),
 M \mapsto M\dtw{c_x}
$$
is an equivalence of $t$-categories (Remarks \ref{rem:Teff_indep_delta}
 and \ref{rem:htp_t_nonnegdeg_eff}(3)).
But through this equivalence, the motive $i_x^!(\hgrp G)$ is send to
 $\hgrp{G_x}$ in $\DM^{\delta_k-\eff}(x,R)$. So we are restricted to the case
  where $S$ is the spectrum of a field $k$.

The $t$-category $\DM^{\delta_k-\eff}(x,R)$ is invariant under purely
 inseparable extensions (Lemma \ref{lm:inseparable_texact_eff}).
 Applying again formula \eqref{eq:AHP_pullback}, we are thus
 restricted to the case where $k$ is a perfect field.
 Thanks to Example \ref{ex:delta_effective_fields},
 the $t$-category $\DM^{\delta_k-\eff}(x,R)$
 is equivalent to Voevodsky's category $\DM^\eff(k,R)$, equipped
 with the (standard) homotopy $t$-structure. Through this equivalence,
 according to \cite[2.10]{AHP},
 the motive $\hgrp G$ corresponds to the homotopy invariant Nisnevich sheaf
 represented by $G \otimes R$,
 with its canonical transfer structure (cf. \cite[3.1.2]{Org}).
 Therefore, it is in the heart of the homotopy $t$-structure
 and this concludes the proof of point (1).
 We will denote this sheaf by $\hgrp G^{tr}$.

We now consider point (2).
 The morphism $i_x:\spec(E) \rightarrow S$ is essentially
 quasi-projective between regular schemes.
 Thus, using point (1) of Remark \ref{rem:fiber_hlg&good}
 and the fact $i_x^*(\hgrp G)=\hgrp{G_x}$ according to \eqref{eq:AHP_basechg},
 we are reduced to the case where $S=\spec{E}$, $x$ being
 the tautological point.

Note that all members involved in the relation to be proved
 are invariant under purely inseparable extensions of the field $E$
 (for the left hand side, this follows from
 Lemma \ref{lm:inseparable_texact_eff},
 the right hand side is obvious except for Somekawa K-groups,
 case which follows from the existence of norm maps and the fact we
 work with rational coefficients). Thus,
 we can assume that $E$ is perfect.

According to what was said before, we thus are restricted to compute
 the homotopy module $M$ over $E$ associated with the homotopy invariant
 sheaf with transfers $\hgrp G$ (\emph{i.e.} the graded sheaf
 $\sigma^\infty(\hgrp G)$, see \cite[(1.18.b)]{Deg11}).
 According to \cite[1.1]{Kahn1},
 we get $\big(\hgrp G\big)_{-1}\simeq L$ where $L$ is the group
 of cocharacters of $G$, seen as a homotopy invariant Nisnevich sheaf
 with transfers. The first three relations follow because
 obviously $L_{-1}=0$.
 Then the last relation follows:
\begin{align*}
\hat{\hgrp G}(x,n)=\sigma^\infty(\hgrp G)_n(E)
 \stackrel{(1)}=\lbrack \hgrp G \otimes^{\mathrm{Htr}} \GG^{\otimes^\mathrm{Htr},n}\rbrack(E)
 &\stackrel{(2)}=\Hom_{\DMe(E,R)}(\un,\hgrp G \otimes \GG^{\otimes,n}) \\
&\stackrel{(3)}=K(E;G_x,\underset{n \text{ times}}{\underbrace{\GG,...,\GG}}) \otimes_\ZZ R.
\end{align*}
where (1) follows from \cite[1.18]{Deg8} --- $\otimes^{\mathrm{Htr}}$ refers to the tensor
 product of homotopy modules with transfers over $E$ --- 
 (2) follows from the definition of $\otimes^{\mathrm{Htr}}$ and (3)
 is the main theorem of \cite{KY} (see (1.1) in the introduction).
\end{proof}

\begin{thm}\label{thm:sabelian&hmod}
Under the assumptions of the preceding corollary,
 we get a canonical exact fully faithful functor:
$$
\sab_S^0 \rightarrow \she(S,R), \ G \mapsto \hgrp G
$$
where the source category is the category of semi-abelian
 $S$-schemes up to isogeny.
 Moreover, its essential image is stable under extensions.
\end{thm}
The exactness of the functor follows from \ref{prop:AHP1},
 and the full faithfulness, as well as the stability under extensions,
 follows from the main theorem of \cite{Pep}. 

\begin{rem}
\begin{enumerate}
\item In the statement of the previous theorem,
 one could replace the abelian category $\she(S,R)$
 by $\sh(S,R)$ or $\sht(S,R)$ (Theorem \ref{thm:compute_heart}).
\item Applying the main theorem of \cite[Th. 3.3]{AHP}, we
 get for any abelian $S$-scheme $A$ of dimension $g$
 a canonical isomorphism:
\begin{equation}\label{eq:DM-decomposition}
M_S(A) \xrightarrow{\sim} \bigoplus_{n=0}^{2g} \mathrm{Sym}^n(\hgrp A)
\end{equation}
which is a generalization of the Deninger-Murre decomposition of
 the Chow motive of $A$. Indeed, when $S$ is a smooth $k$-scheme,
 This decomposition lives in the category of motives over $S$
 of weight $0$ which is equivalent to the category of pure Chow motives
 according to \cite{Jin}. Through this equivalence $M_S(A)$ corresponds
 to the dual of the Chow motive of $A/S$ and the preceding decomposition
 is equal to the dual of the Deninger-Murre decomposition.

When $A=E$ is an elliptic curve,
 the preceding theorem implies that $M_S(A)$ is concentrated in
 homological degree $0$ and $1$ for the $t$-structure $t_\delta^\eff$.
 In the next statement, we extend this result for $S$-curves of arbitrary
 genus.
\end{enumerate}
\end{rem}

\begin{num}\label{num:compute_hlg_curves}
Let $S$ be a regular scheme and consider a smooth projective
 geometrically connected relative curve $p:\bar X \rightarrow S$
 with a given section $x:S \rightarrow \bar X$.

In our assumptions, the motive $M_S(\bar X)$ in $\DM(S,R)$ has a
 Chow-K\"unneth decomposition as in the classical case -- though
 it is dual because Voevodsky's motives are homological:
 $M_S(\bar X)=p_!p^!(\un_S)$.
 First, the induced map $x_*:\un_S \rightarrow M_S(\bar X)$ is
 split by $p_*$. Let us define the reduced motive $\tilde M_S(\bar X)$
 of the pointed $S$-curve $(\bar X,x)$ as the cokernel of $x_*$,
 or equivalently, the kernel of $p_*$.

Recall that Voevodsky motives of smooth $S$-schemes 
 are contravariant (up to a twist) with respect to projective morphisms,
 via the so-called Gysin morphism (cf. \cite{Deg8}). So, 
 associated to the morphisms $p$ and $x$, we get Gysin morphisms:
$$
p^*:\un_S(1)[2] \rightarrow M_S(\bar X), \ 
 x^*:M_S(\bar X) \rightarrow \un_S(1)[2]
$$
satisfying the relation $x^*p^*=Id$. Thus $p^*$ is a split monomorphism.
 Moreover, the composite map $p_* \circ p^*$ is a class in
 the motivic cohomology group $H^{-2,-1}_M(S,R)$ which is zero because
 $S$ is regular.
 Hence $p^*$ canonically factorizes into $\tilde M_S(X)$. Let us denote by
 $h_1^M(M_S(\bar X))$ its cokernel in $\tilde M_S(X)$.
 We have finally obtained a canonical decomposition:
\begin{equation}\label{eq:chow-kunneth_curves}
M_S(\bar X)=\un_S \oplus h_1^M(M_S(\bar X)) \oplus \un_S(1)[2],
\end{equation}
the \emph{dual Chow-K\"unneth decomposition} associated with $\bar X/S$.
 
Let us now consider the relative Jacobian scheme $J(\bar X)$ associated
 with $\bar X/S$ -- the connected component of the identity of the Picard
 scheme associated with $\bar X/S$, see \cite[\textsection 8.4]{BLR}.
 The section $x \in \bar X(S)$ gives a canonical map:
 $\bar X \xrightarrow{\iota} J(\bar X)$.
 We deduce a canonical morphism of $R$-linear motives over $S$:
$$
\varphi_x:h_1^M(M_S(\bar X)) \rightarrow M_S(\bar X)
 \xrightarrow{\iota_*} M_S(J(\bar X)) \rightarrow \hgrp{J(\bar X)}
$$
where the first map is the canonical inclusion and
 the last one is the canonical projection with respect to
 the decomposition \eqref{eq:DM-decomposition}.
 Another application of the preceding theorem is the following
 generalization of a classical result of Voevodsky
 \cite[chap. 5, Th. 3.4.2]{FSV}:
\end{num}
\begin{prop}\label{prop:compute_hlg_curves}
Consider the above assumptions.
 Then the morphism $\varphi_x$ is an isomorphism.

Moreover, for the dimension function $\delta=-\codim_S$,
 the $\delta$-effective motive $M_S(\bar X)$ is concentrated
 in homological degree $0$ and $1$ for the effective $\delta$-homotopy
 $t$-structure and one has canonical 
 isomorphisms:
$$
H_i^{\delta-\eff}(M_S(\bar X))=\begin{cases}
\un_S \oplus \hgrp{J(\bar X)} & \text{if } i=0, \\
\un_S\gtw 1 & \text{if } i=1.
\end{cases}
$$
Finally, the distinguished triangle associated
 with the truncation functor $\tau_{\geq 0}^\delta$
 has the form:
$$
\un_S \oplus \hgrp{J(\bar X)} \rightarrow M_S(\bar X)
 \rightarrow \un_S(1)[2] \xrightarrow{\partial} \un_S[1];
$$
it splits since $\partial=0$.
\end{prop}
\begin{proof}
Consider the first assertion.
 Note first that the decompositions \eqref{eq:chow-kunneth_curves}
 and \eqref{eq:DM-decomposition} are stable by base change:
 for the first one, it follows from the fact Gysin morphisms
 are compatible with base change in the transversal case
 (cf. \cite[5.17(i)]{Deg8}),
 and for the second one by \cite[2.7 and 3.3(4)]{AHP}.
 Note also that the family of functors
 $i_{\bar s}:\DM(S,R) \rightarrow \DM(\bar s,R)$
 indexed by geometric points of $S$ is conservative:
 it follows directly from the following properties
 of $\DM_R$: localization (\cite[14.2.11]{CD3}),
 continuity (\cite[14.3.1]{CD3}) and separation (\cite[14.3.3]{CD3})
 --- see also \cite[Lem. A.6]{AHP}.
 Thus to check $\varphi_x$ is an isomorphism, we only need to
 consider its pullback over geometric points of $S$.
 In other words, we are reduced to the case of a separably
 closed field $S=\spec(k)$.
 But then it follows from a classical fact about Chow-K\"unneth
 decompositions, which boils down to the identification of
$CH^1(h^1(\bar X))=Pic^0(\bar X)$
with $CH^1_1(J(\bar X))$ with the notations of \cite{Beau}
 (see Prop. 3 of \emph{loc. cit.} and the comments following it).

The remaining facts are then consequences of the preceding theorem,
 Example \ref{ex:unit_eff_heart}, and the vanishing of
 $H^{-1,-1}_M(S,R)$.
\end{proof}

\begin{num}
We consider again the notations and assumptions of
 \S\ref{num:compute_hlg_curves}.
 Suppose we are given a closed subscheme $\nu:X_\infty \hookrightarrow \bar X$
 such that the induced morphism $f:X_\infty \rightarrow S$ is \'etale.
 Let us put $X=\bar X-X_\infty$ with open immersion $j:X \rightarrow \bar X$.
 We also assume that the section $x \in \bar X(S)$ is disjoint from $X_\infty$
 so that it induces a section $x \in X(S)$.

Note that we can associate with $f$ a Gysin morphism
 (see \cite{Deg8} or use \ref{num:BM}, (BM2)+(BM4)):
 $f^*:\un_S \rightarrow M_S(X_\infty)$.

As in the case of $\bar X$, we will denote by $\tilde M_S(X)$
 the reduced motive associated with the pointed $S$-curve $(X,x)$. 
 Recall $M_S(X)=\un_S \oplus \tilde M_S(X)$,
 and similarly for $\bar X$.
 From the Gysin triangle
 (see \cite{Deg8} or use \ref{num:BM}, (BM3)+(BM4))
 associated with $\nu$, we deduce the following distinguished triangle:
\begin{equation} \label{eq:Gysin_affine_curve}
M_S(X_\infty)\gtw 1
 \xrightarrow{\partial_\nu} \tilde M_S(X)
 \xrightarrow{j_*} \tilde M_S(\bar X)
 \xrightarrow{\nu^*} M_S(X_\infty)\dtw 1.
\end{equation}
We also get a canonical morphism
$$
\pi:\tilde M_S(X) \xrightarrow{j_*} \tilde M_S(\bar X)
 \xrightarrow{\bar \pi} \hgrp{J(\bar X)}
$$
where the last map is the canonical projection according to the preceding proposition.
\end{num}
\begin{prop}\label{prop:compute_hlg_curves_aff}
Consider the notations above.

Then $M_S(X)$ is a $\delta$-effective motive which is concentrated in degree
 $0$ for the $\delta$-homotopy $t$-structure:
 $M_S(X)=H_0^{\delta-\eff}\big(M_S(X)\big)$.

Moreover, the following sequence is exact in the category of effective
 $\delta$-homotopy modules $\she(S,R)$:
$$
0 \rightarrow \un_S\gtw 1 \xrightarrow{f^*}
 M_S(X_\infty)\gtw 1 \xrightarrow{\partial_\nu} \tilde M_S(X)
 \xrightarrow \pi \hgrp{J(\bar X)} \rightarrow 0.
$$
In other words, $M_S(X)$ is isomorphic to the $\delta$-homotopy module 
 of the (dual of the) Albanese semi-abelian scheme $\mathrm{Alb}(X/S)$
 associated with the smooth (affine) curve $X/S$:
$$
M_S(X) \simeq \un_S \oplus \hgrp{\mathrm{Alb}(X/S)}.
$$
\end{prop}
\begin{proof}
In this proof, we will denote $H_i$ for $H_i^{\delta-\eff}$ to simplify
 the notations.

The fact $M_S(X)$ is $\delta$-effective is basic
 (Example  \ref{ex:motive_sm_eff}).
 Because $X_\infty/S$ is finite \'etale, $S$ is regular and $\delta(S)=0$,
 we obtain that the $\delta$-effective motive $M_S(X_\infty)$ is
 concentrated in degree $0$ (Ex. \ref{ex:unit_eff_heart} together with
 the end of \textsection \ref{num:adj_tdelta_eff}).
 Thus, the homological long exact sequence associated with the distinguished triangle
 \eqref{eq:Gysin_affine_curve} together with the preceding theorem immediately 
 yields that $H_i^{\delta-\eff}(M_S(C))=0$ if $i\neq 0,1$. Moreover,
 we get an exact sequence in the heart:
$$
0 \rightarrow H_1(\tilde M_S(X))
 \xrightarrow{j_*} H_1(\tilde M_S(\bar X))
 \xrightarrow{\nu^*} M_S(X_\infty)\gtw 1
 \xrightarrow{\partial_\nu} H_0(\tilde M_S(X))
 \xrightarrow{j_*} H_0(\tilde M_S(\bar X)) \rightarrow 0.
$$
According to the preceding theorem, the the following composition:
$$
\un_S\gtw 1 \xrightarrow{p^*} \tilde M_S(\bar X)[-1]
\rightarrow \tilde H_1(M_S(\bar X))
$$
is an isomorphism. Thus, in the preceding exact sequence, the morphism
 $\nu^*$ is isomorphic to the morphism
 $f^*:\un_S\gtw 1 \rightarrow M_S(X_\infty)\gtw 1$
 because $f^*=\nu^*p^*$ (cf. \cite[5.14]{Deg8}).
 
Note that the composite map:
$$
\un_S \xrightarrow{f^*}  M_S(X_\infty) \xrightarrow{f_*} \un_S
$$
is equal to $d.Id$ where $d$ is the degree of the \'etale morphism $f$.\footnote{This
 is classical. In our case, we can argue as follows: because $\HB^{00}(S)=\QQ$, 
 the composite $q_S:=f_*f^*$ is a rational number. Moreover, one can easily check
 this number is invariant under pullback by a smooth morphism $T \rightarrow S$. 
 Given that $f$ is an \'etale cover,
 we can find an \'etale map $T \rightarrow S$ such that $f \times_S T$ is trivial.
 Then the formula is obvious by additivity.}
 In particular, as we work with rational coefficients, the map $f^*$
 is a split monomorphism and this implies $H^1(\tilde M_S(X))=0$ which
 implies $H^1(M_S(X))=0$ (use again Example \ref{ex:unit_eff_heart})
 as required.

We conclude because from the previous theorem the canonical
 map $\tilde M_S(X) \xrightarrow{\bar \pi}\hgrp{J(\bar X)}$
 induces an isomorphisms on $H_0$.
\end{proof}

%

\bibliographystyle{amsalpha}
\bibliography{htp}

\end{document}